\documentclass[12pt]{amsart}

\usepackage{graphicx}
\usepackage{amssymb, amsthm, amssymb,latexsym, bm, comment}
\usepackage{epstopdf}
\usepackage{tikz}
\usetikzlibrary{arrows}
\usepackage{caption}
\usepackage{subcaption}
\usepackage{setspace}
\usepackage{color}
\usepackage{geometry}
\usepackage{amssymb}
\usepackage{amsmath}

\usepackage{longtable, tabularx, array}
\let\oldmarginpar\marginpar
\renewcommand\marginpar[1]{\oldmarginpar[\raggedleft\footnotesize #1]%
{\raggedright\footnotesize #1}}

\theoremstyle{plain}
\newtheorem{theorem}{Theorem}[section]
\newtheorem{corollary}[theorem]{Corollary}
\newtheorem{lemma}[theorem]{Lemma}
\newtheorem{prop}[theorem]{Proposition}

\newtheorem{conjecture}[theorem]{Conjecture}

\theoremstyle{definition}
\newtheorem{define}[theorem]{Definition}

\newtheorem{remark}[theorem]{Remark}

\newcommand{\C}{\mathbb{C}}
\newcommand{\R}{\mathbb{R}}
\newcommand{\Q}{\mathbb{Q}}
\newcommand{\Z}{\mathbb{Z}}
\newcommand{\N}{\mathbb{N}}
\newcommand{\CC}{\mathbb{C}}
\newcommand{\RR}{\mathbb{R}}
\newcommand{\HH}{\mathbb{H}}
\newcommand{\ZZ}{\mathbb{Z}}

\title{Intercusp Geodesics and Cusp Shapes of Fully Augmented Links}
\date{\today}
\author{Rochy Flint}

\address{DEPARTMENT OF MATHEMATICS, SCIENCE, AND TECHNOLOGY, TEACHERS COLLEGE, COLUMBIA UNIVERSITY, NEW YORK, NY}
\email{crf51@tc.columbia.edu}

\begin{document}

\begin{abstract}
  We study the geometry of fully augmented link complements
  in $S^3$ by looking at their link diagrams. We extend the method
  introduced by Thistlethwaite and Tsvietkova \cite{tt} to fully
  augmented links and define a system of algebraic equations in terms of parameters coming from edges and crossings of the link diagrams. Combining it with the work of Purcell \cite{purcell-fal}, we show that the solutions to these algebraic equations are related to the cusp shapes of fully augmented link complements.  As an application we use the cusp shapes to study the commensurability
  classes of fully augmented links.

\end{abstract}

\maketitle{}

\section{Introduction}

Understanding the relationship between the combinatorics of a link
diagram and the geometry and topology of its complement is an
important problem and an active area of research.  In this paper we
study this relationship for an infinite family of links called
\emph{fully augmented links}. These are links obtained from a given
link diagram by augmenting every twist region with a circle component
and removing all twists, see Figure \ref{3pa}.

Thurston studied the interactions between geometry and combinatorics
using ideal triangulations and gluing equations for hyperbolic link
complements and 3-manifolds. The solutions to the gluing equations
allow us to construct the discrete faithful representation of the
fundamental group of the link complement to
$\rm{Isom}^+(\mathbb{H}^3)$ and help us to compute many geometric
invariants.  Although an ideal triangulation can be obtained from a
link diagram easily, it is much harder to find solutions to the gluing
equations. In addition, it is difficult to relate the geometric
invariants obtained from the solutions of gluing equations to the
diagrammatic invariants obtained from the link diagram.

In \cite{tt} Thistlethwaite and Tsvietkova used link diagrams to study
the geometry of hyperbolic alternating link complements by
implementing a method to construct a system of algebraic equations
directly from the link diagram. The solutions to these equations allow
them to construct the discrete faithful representation of the link
group into Isom$^+(\HH^3)$.
 We refer to their method as the T-T
method. The idea of the T-T method is as follows: by looking at the
faces of the link diagram, and assigning parameters to crossings and
edges in every face, they find relations on the parameters using the 
geometry of the link complement, which
determine algebraic equations, and the solutions to these equations have 
geometric information. 
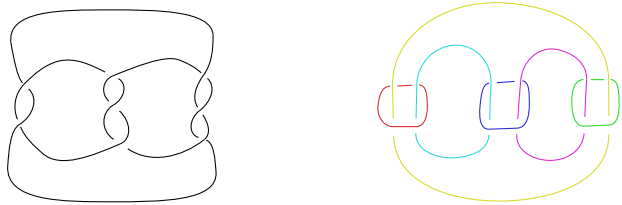
\begin{figure}
    \centering
\begin{tikzpicture}[scale=.3][line width=4.9, line cap=round, line join=round]
    \draw (1.03, 5.53) .. controls (0.79, 5.76) and (0.63, 6.65) .. 
          (0.52, 7.26) .. controls (0.26, 8.77) and (2.74, 8.77) .. 
          (4.87, 8.77) .. controls (6.90, 8.77) and (9.47, 8.77) .. 
          (9.47, 7.61) .. controls (9.47, 7.00) and (9.41, 6.37) .. (9.05, 5.88);
    \draw (9.05, 5.88) .. controls (8.73, 5.44) and (8.50, 4.89) .. (8.79, 4.45);
    \draw (9.01, 4.12) .. controls (9.26, 3.75) and (9.40, 3.27) .. (9.08, 3.02);
    \draw (9.08, 3.02) .. controls (8.08, 2.23) and (6.72, 1.92) .. (5.69, 2.61);
    \draw (5.06, 3.03) .. controls (4.49, 3.40) and (4.49, 4.22) .. (5.04, 4.54);
    \draw (5.04, 4.54) .. controls (5.58, 4.85) and (5.68, 5.51) .. (5.27, 5.72);
    \draw (4.69, 6.01) .. controls (4.04, 6.34) and (3.35, 6.67) .. 
          (2.64, 6.50) .. controls (2.02, 6.35) and (1.53, 5.92) .. (1.09, 5.47);
    \draw (1.09, 5.47) .. controls (0.69, 5.06) and (0.58, 4.45) .. (0.80, 3.92);
    \draw (0.94, 3.58) .. controls (1.13, 3.10) and (1.53, 2.74) .. 
          (1.92, 2.40) .. controls (2.71, 1.70) and (4.17, 2.32) .. (5.36, 2.83);
    \draw (5.36, 2.83) .. controls (5.88, 3.05) and (5.84, 3.81) .. (5.32, 4.29);
    \draw (4.85, 4.72) .. controls (4.48, 5.06) and (4.57, 5.66) .. (5.02, 5.84);
    \draw (5.02, 5.84) .. controls (5.97, 6.23) and (7.29, 6.76) .. 
          (7.90, 6.57) .. controls (8.28, 6.44) and (8.64, 6.25) .. (8.94, 5.98);
    \draw (9.22, 5.73) .. controls (9.62, 5.36) and (9.39, 4.71) .. (8.88, 4.31);
    \draw (8.88, 4.31) .. controls (8.38, 3.93) and (8.37, 3.25) .. (8.82, 3.11);
    \draw (9.17, 2.99) .. controls (9.51, 2.88) and (9.56, 2.14) .. 
          (9.60, 1.57) .. controls (9.69, 0.26) and (7.10, 0.26) .. 
          (4.97, 0.26) .. controls (2.81, 0.26) and (0.26, 0.26) .. 
          (0.37, 1.85) .. controls (0.42, 2.54) and (0.48, 3.45) .. (0.89, 3.69);
    \draw (0.89, 3.69) .. controls (1.50, 4.07) and (1.75, 4.83) .. (1.30, 5.26);
\end{tikzpicture}
\hspace{1cm}
\definecolor{linkcolor0}{rgb}{0.85, 0.15, 0.15}
\definecolor{linkcolor1}{rgb}{0.15, 0.15, 0.85}
\definecolor{linkcolor2}{rgb}{0.15, 0.85, 0.15}
\definecolor{linkcolor3}{rgb}{0.15, 0.85, 0.85}
\definecolor{linkcolor4}{rgb}{0.85, 0.15, 0.85}
\definecolor{linkcolor5}{rgb}{0.85, 0.85, 0.15}
\begin{tikzpicture}[scale=.34][line width=4.5, line cap=round, line join=round]
  \begin{scope}[color=linkcolor0]
    \draw (0.73, 4.75) .. controls (0.39, 4.74) and (0.29, 4.33) .. 
          (0.26, 3.95) .. controls (0.23, 3.55) and (0.48, 3.17) .. (0.86, 3.17);
    \draw (0.86, 3.17) .. controls (1.16, 3.17) and (1.45, 3.17) .. (1.75, 3.17);
    \draw (1.75, 3.17) .. controls (2.09, 3.17) and (2.20, 3.59) .. 
          (2.19, 3.99) .. controls (2.18, 4.37) and (2.17, 4.81) .. (1.86, 4.80);
    \draw (1.59, 4.79) .. controls (1.40, 4.78) and (1.21, 4.77) .. (1.02, 4.76);
  \end{scope}
  \begin{scope}[color=linkcolor1]
    \draw (4.61, 4.87) .. controls (4.33, 4.85) and (4.29, 4.35) .. 
          (4.25, 3.95) .. controls (4.21, 3.52) and (4.26, 3.05) .. (4.61, 3.07);
    \draw (4.61, 3.07) .. controls (4.98, 3.08) and (5.34, 3.10) .. (5.70, 3.11);
    \draw (5.70, 3.11) .. controls (6.09, 3.13) and (6.18, 3.61) .. 
          (6.18, 4.05) .. controls (6.18, 4.47) and (6.18, 4.97) .. (5.90, 4.96);
    \draw (5.60, 4.94) .. controls (5.37, 4.92) and (5.14, 4.91) .. (4.91, 4.89);
  \end{scope}
  \begin{scope}[color=linkcolor2]
    \draw (8.30, 5.01) .. controls (7.92, 5.01) and (7.83, 4.54) .. 
          (7.84, 4.09) .. controls (7.85, 3.65) and (7.95, 3.17) .. (8.33, 3.19);
    \draw (8.33, 3.19) .. controls (8.65, 3.20) and (8.97, 3.22) .. (9.29, 3.23);
    \draw (9.29, 3.23) .. controls (9.65, 3.25) and (9.70, 3.71) .. 
          (9.69, 4.13) .. controls (9.68, 4.54) and (9.67, 5.01) .. (9.36, 5.01);
    \draw (9.11, 5.01) .. controls (8.94, 5.01) and (8.76, 5.01) .. (8.59, 5.01);
  \end{scope}
  \begin{scope}[color=linkcolor3]
    \draw (1.78, 4.80) .. controls (1.79, 5.63) and (2.44, 6.32) .. 
          (3.27, 6.36) .. controls (4.07, 6.39) and (4.71, 5.70) .. (4.68, 4.88);
    \draw (4.68, 4.88) .. controls (4.66, 4.39) and (4.64, 3.91) .. (4.63, 3.43);
    \draw (4.60, 2.85) .. controls (4.58, 2.22) and (3.86, 1.96) .. 
          (3.15, 1.95) .. controls (2.43, 1.94) and (1.73, 2.28) .. (1.74, 2.92);
    \draw (1.75, 3.49) .. controls (1.76, 3.93) and (1.77, 4.36) .. (1.78, 4.80);
  \end{scope}
  \begin{scope}[color=linkcolor4]
    \draw (8.42, 5.01) .. controls (8.45, 5.65) and (7.86, 6.10) .. 
          (7.20, 6.20) .. controls (6.49, 6.30) and (5.89, 5.70) .. (5.84, 4.95);
    \draw (5.84, 4.95) .. controls (5.80, 4.46) and (5.76, 3.96) .. (5.73, 3.47);
    \draw (5.69, 2.87) .. controls (5.64, 2.26) and (6.27, 1.87) .. 
          (6.94, 1.85) .. controls (7.64, 1.82) and (8.28, 2.25) .. (8.32, 2.91);
    \draw (8.35, 3.55) .. controls (8.37, 4.04) and (8.39, 4.53) .. (8.42, 5.01);
  \end{scope}
  \begin{scope}[color=linkcolor5]
    \draw (0.83, 4.76) .. controls (0.79, 6.77) and (2.83, 8.08) .. 
          (5.01, 8.02) .. controls (7.16, 7.96) and (9.27, 6.94) .. (9.28, 5.01);
    \draw (9.28, 5.01) .. controls (9.28, 4.54) and (9.29, 4.06) .. (9.29, 3.59);
    \draw (9.29, 2.87) .. controls (9.30, 1.06) and (7.18, 0.29) .. 
          (5.11, 0.26) .. controls (3.04, 0.23) and (0.91, 0.99) .. (0.87, 2.81);
    \draw (0.86, 3.49) .. controls (0.85, 3.91) and (0.84, 4.33) .. (0.83, 4.76);
  \end{scope}
\end{tikzpicture}
    \caption{link $K$ (left)  and the corresponding fully
      augmented link $L$ (right).}
    \label{3pa}
\end{figure}

Throughout this paper we abbreviate fully augmented links to FALs. In
this paper we show the following for FALs:
\begin{enumerate}
\item A way to extend the T-T method to FALs. This is
  the first application of the T-T method to an infinite class of
  non-alternating links. This is done in Proposition
  \ref{polyhedprop}, Theorem \ref{tps}, Lemma \ref{reverseorientation}, and Lemma \ref{edge};

\item A new method to determine the cusp shapes of FAL complements
  using the solutions of the system of equations obtained from the T-T
  method. This is proved in Theorem \ref{mainthm} and Theorem
  \ref{mainthm2};

\item A way to study commensurability of different classes of
  FALs. This is done in Theorems \ref{cm}, \ref{km}, and \ref{com1};

\item A way to choose the geometric solutions i.e. the solution which
  enables us to construct the discrete, faithful representation, from
  the solutions of the system of equations obtained from the T-T
  method. We demonstrate this in Theorem \ref{geometric}.
\end{enumerate}

This paper is divided into $5$ sections.  In \S \ref{sec:background}
we give necessary background about FALs, the geometry of their
complements, and introduce the T-T method for alternating links. We
also give an example illustrating the T-T method on alternating links. In \S \ref{sec:TT-FAL} we show that the T-T method can
be extended to FALs, and illustrate with examples. In \S
\ref{sec:cusp} we state our main theorem relating the cusp shapes to
the intercusp-geodesics, and give explicit examples. \S \ref{sec:itf}
discusses applications of our main theorems by studying the invariant
trace fields, commensurability and finding geometric solutions to
systems of equations in the T-T method for FALs.
 
 \subsection{Acknowledgements} This paper consists of results from my PhD dissertation. Many thanks to Abhijit Champanerkar for his invaluable support, innumerable conversations, and unlimited patience. I wish to thank Ilya Kofman and Walter Neumann for their constant support and advice. I also thank Jessica Purcell and Anastasiia Tsvietkova for helpful conversations, encouragement and their work which motivated this research project.

\section{Background}
\label{sec:background}

A \emph{hyperbolic 3-manifold} $M$ is a 3-manifold equipped with a
complete Riemannian metric of constant sectional curvature -1,
i.e. the universal cover of $M$ is $\HH^3$ with covering translations
acting as isometries. Equivalently, $M= \HH^3/\Gamma$ where $\Gamma$
is a torsion free Kleinian group i.e. a discrete torsion-free subgroup of
PSL$(2,\C) = \rm{Isom}^+(\HH^3)$.  In this paper we assume that our
hyperbolic 3-manifolds are complete, orientable and
 have finite hyperbolic volume. Hyperbolic 3-manifolds have a
thick-thin decomposition that allows us to understand the topology of
non-compact hyperbolic 3-manifolds. This decomposition consists of a
thin part with tubular neighborhoods of closed geodesics and ends
which are homeomorphic to a thickened torus.

A \emph{cusp} of a hyperbolic 3-manifold is the thin end isometric to
$T^2 \times [0,\infty)$ with the induced metric given as
$ds^2 = e^{-2t}(dx^2+dy^2) +dt^2$.

If $M= \HH^3/\Gamma$, then $M$ is non-compact if and only if
$\Gamma$ contains parabolic isometries (i.e. they have one fixed point
on the sphere at infinity of $\HH^3$), which correspond to the cusps of
$M$.  Distinct cusps of $M$ correspond to distinct conjugacy classes of
maximal parabolic subgroups of $\Gamma$. Note that the cross sectional tori
$T^2\times \{t\}$ are scaled Euclidean tori.

We say a link $K \in S^3$ is \emph{hyperbolic} if its complement
$S^3-K$ is a hyperbolic 3-manifold.  In a link complement, the cusps
are the tubular neighborhoods of the component of the link with
the link components deleted.  The cusps lifts to a set of horoballs with disjoint
interiors in the universal cover $\HH^3$. For each cusp, the set of
horoballs are identified by the covering transformations.
Thurston's famous example of the figure-eight knot complement
decomposing into two ideal tetrahedra \cite{Thurston1979} is the first
example of finding the hyperbolic structure from a link diagram. Jeff
Weeks implemented the computer program SnapPea which finds the
geometric structure on link complements from link diagrams
\cite{Weeks2008}.  This was extended by Marc Culler and Nathan
Dunfield to the program SnapPy \cite{SnapPy}. 

Hyperbolic structures are useful to study knots and links using
geometric invariants.  One such invariant we will study in this paper
is the \emph{cusp shape}.
\begin{define}
  \cite{neumann1992arithmetic} A horospherical section of a cusp of a
  hyperbolic 3-manifold $M$ is a flat torus. This torus is isometric
  to $\mathbb{C}/\Lambda$, for some lattice
  $\Lambda \subset \mathbb{C}$, and the ratio of two generators of
  $\Lambda$ is the conformal parameter of the flat torus, which we
  call the \emph{cusp shape} of the cusp of $M$. Choosing generators
  $[m]$ and $[\ell]$ of $\pi_1(T^2)$, the Euclidean structure on the
  torus is obtained by mapping $[m]$ and $[\ell]$ to Euclidean
  translations $T_1(z) = z + \mu$ and $T_2(z)= z + \lambda$
  respectively, where $\mu$ and $\lambda \in \mathbb{C}$. Then $\mu$
  and $\lambda$ generate the lattice $\Lambda$ and the cusp shape is
  obtained as $\lambda/ \mu$.

\end{define}

The cusp shape gives us very important information about links.  Given
a specific link diagram, we can compute the cusp shape by drawing the
link diagram in SnapPy \cite{SnapPy}, which computes the hyperbolic
structure on its complement and many geometric invariants including
the cusp shape. For example, SnapPy gives a numerical value of the
cusp shapes of the Hamantash Link (See Figure \ref{magic}) as $1.5 + 1.32287565553i$. 
Below we will develop another way to
compute the cusp shape for FALs directly from the diagram of the links
using what we call the T-T polynomial.

\begin{define}
Two hyperbolic 3-manifolds are \emph{commensurable} if they
have a common finite-sheeted cover.
\end{define}

The cusp shape gives us key information that will enable us to analyze
whether certain links are commensurable.  The cusp shapes are
algebraic numbers and generate a number field called the \emph{cusp
  field}. Although the cusp shape depends on the choice of generators
of the peripheral subgroup, a different choice changes it by an
integral M\"obius transformation, hence the cusp field is independent of
choices of generators.  The cusp field is a commensurability invariant
\cite{Neumann2011}.  Hence cusp shapes can be used to determine
commensurability of two links complements.

\subsection{Fully Augmented Links}

The class of links that we will be studying is called fully augmented links.
\begin{define}
A link diagram is \emph{prime} if for any simple closed curve in the plane that intersects a component transversely in two points the simple closed curve
bounds a subdiagram containing no crossings. See Figure \ref{prime}(a).
\end{define}
\begin{figure}[h]
    \centering
    \begin{tikzpicture}
\draw [dashed] (0,3) rectangle (1,2);
\draw [dashed] (2.5,3) rectangle (3.5,2);
\draw (.5,3)..controls (1.75,4)..(3,3);
\draw (.5,2)..controls (1.75,1)..(3,2);
\node at (.5,2.5) {$A$};
\node at (3,2.5) {$B$};
\node at (4,2.5) {$\Longrightarrow$};
\draw [dashed] (4.5,3) rectangle (5.5,2);
\draw (5,3)--(5,2);
\node at (5,3.5) {$A$ or $B$};
\end{tikzpicture}
\hspace {1.5cm}
\begin{tikzpicture}
\draw [dashed] (0,3) rectangle (1,2);
\draw [dashed] (2,3) rectangle (3,2);
\draw (.3,3)..controls (.4,3.2)..(.5,3.25);
\draw (.7,3)..controls (.6,3.3)..(.4,3.5);
\draw (.4,3.5)..controls (2,3.8)..(2.7,3);
\draw (.7,3.27)..controls (1.9,3.56)..(2.3,3);
\draw (.3,2)..controls (1.6,1.4)..(2.3,2);
\draw (.7,2)..controls (.75,1.95)..(.65,1.88);
\draw (.6,1.8)..controls (.5,1.7)..(.55,1.65);

\draw(.55,1.65) ..controls(1.5,1)..(2.7,2);
\node at (.5,2.5) {$A$};
\node at (2.5,2.5) {$B$};
\node at (3.5,2.5) {$\Longrightarrow$};
\draw [dashed] (4,3) rectangle (5,1.5);
\node at (4.5,3.5) {$A$ or $B$};
\draw (4.3,3)..controls (4.2,2.9) and (4.7,2.8)..(4.5,2.6);
\draw (4.6,3)..controls(4.63,2.9)..(4.55,2.85);
\draw(4.45,2.75)..controls (4.2,2.65) and (4.7,2.55)..(4.5,2.35);
\draw (4.45,2.5) ..controls (4.35,2.45)..(4.4,2.35);
\node at (4.5,2.2) {$\cdot$};
\node at (4.5,2.1) {$\cdot$};
\node at (4.5,2.) {$\cdot$};
\draw (4.3,1.8)..controls (4.2,1.7) and (4.7,1.6)..(4.5,1.5);
\draw (4.55,1.75)--(4.45,1.7);
\draw (4.4,1.6)--(4.3,1.5);
\end{tikzpicture}
    (a) \quad \quad \quad \quad \quad \quad \quad \quad \quad \quad \quad \quad \quad \quad \quad \quad (b)
    \caption{(a) Prime diagram (b) Twist Reduced diagram}
    \label{prime}
\end{figure}
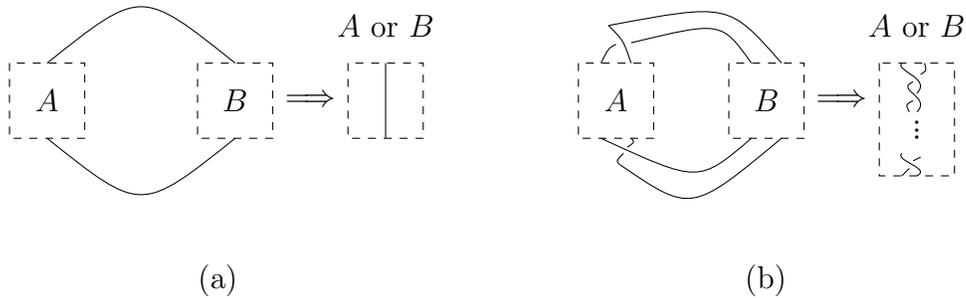
\begin{define}
In a link diagram, a string of bigons, or a single crossing is called a \emph{twist region}.  A link diagram is \emph{twist reduced} if for any simple closed curve in the plane that intersects the link transversely
in four points, with two points adjacent to one crossing and the other two
points adjacent to another crossing, the simple closed curve bounds a subdiagram
consisting of a (possibly empty) collection of bigons strung end to end
between these crossings. See Figure \ref{prime}(b).
\end{define}
\begin{figure}
\centering
\includegraphics[width=5in]{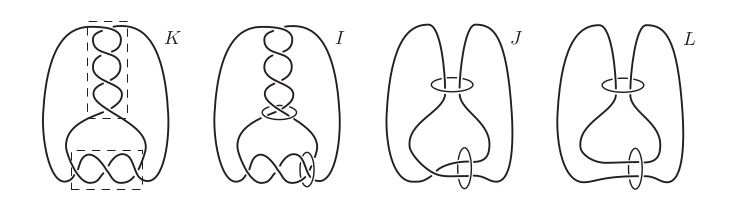}
\\ \hfill \hfill (a) \hfill  \hfill (b) \hfill  \hfill (c) \hfill \hfill (d) \hfill \hfill \hfill 

\caption[$K$ and corresponding FAL]{(a) Link diagram $K$ (b) Crossing circles added at each twist region (c) Augmented Link with all full twists removed (d) fully augmented link $L$ \cite{JP2004}.}
\label{fal}
\end{figure}
\begin{define} A \emph{fully augmented link} (FAL) is a link that is obtained from a
diagram of a link $K$ as follows: 
\begin{enumerate}
    \item 
 augment every twist region with a circle
component (called a \emph{crossing circle}), \item  get rid of all full twists, and \item  remove all remaining half-twists. See Figure \ref{fal}. A diagram obtained above will be referred to as a FAL diagram. The diagram obtained after step (2) is called a FAL diagram with half-twists.
\end{enumerate}

\end{define}
Thus the FAL diagram consists of link components in the projection plane and crossing circle components that are orthogonal to the projection plane and bound twice punctured discs.
In \cite{purcell-fal} Purcell studied
the geometry of FALs using a decomposition of the FAL complement
into a pair of totally geodesic hyperbolic right-angled ideal polyhedra.  We will describe how the geodesic faces of these polyhedra can be seen on the FAL diagrams. 

FAL, while interesting in their own right, enable us to study the geometry of the original knot or link it's built from.

\begin{theorem}
\cite{purcell-fal, adams1986augmented, 2009arXiv0903.5288C} A fully augmented link is hyperbolic if and only if the associated knot or link diagram is non-splittable, prime, twist reduced, with at least two twist regions.
\end{theorem}
We will only consider hyperbolic FALs in this paper. 



 \subsubsection{The Cut-Slice-Flatten Method and Polyhedron $P_L$}

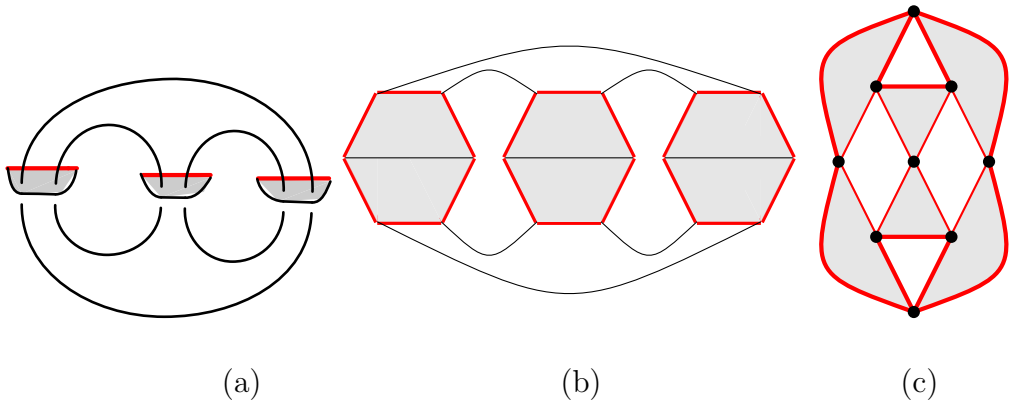
\begin{figure}[t]
\centering
\begin{tikzpicture}[line width=1.0, line cap=round, line join=round, scale=.45]
 \fill[gray!40](7.58, 4.28)--(8.2, 3.62)-- (9.74, 4.21);
 
 \fill[gray!40](9.5, 3.7)--(8.2, 3.62)-- (9.74, 4.21);
 \fill[gray!40]
 (4.14, 4.43)-- (4.74, 3.74)--(6.20, 4.32);
 \fill[gray!40]
 (6, 3.9)-- (4.74, 3.74)--(6.20, 4.32);
 
 \fill[gray!40] (0.21, 4.60) .. controls (0.26, 4.26) and (0.32, 3.86) .. (0.62, 3.85);
    \fill[gray!20](0.62, 3.85) .. controls (0.95, 3.85) and (1.29, 3.84) .. (1.62, 3.83);
    \fill[gray!40](.21,4.6)--(2.23,4.53)--(.62,3.85);
    \fill[gray!40](2.23,4.53)--(.62,3.85)--(1.8,3.9);
\draw[red, ultra thick] (0.2,4.6) -- (.6,4.6);
 \draw[red, ultra thick] (0.7,4.6) -- (2,4.6);
 \draw[red, ultra thick] (2,4.6) -- (2.25,4.6);
 \draw[red, ultra thick] (4.2,4.4) -- (6.2,4.4);
 \draw[red, ultra thick] (7.6,4.3) -- (9.7,4.3);

    \draw (4.75, 4.06) .. controls (4.78, 5.00) and (4.16, 5.84) .. 
         (3.26, 5.88) .. controls (2.35, 5.92) and (1.67, 5.10) .. (1.63, 4.15);
    \draw (1.61, 3.51) .. controls (1.58, 2.70) and (2.28, 2.05) .. 
         (3.12, 2.04) .. controls (3.96, 2.03) and (4.70, 2.61) .. (4.73, 3.43);
  
    \draw (8.36, 3.94) .. controls (8.38, 4.84) and (7.81, 5.66) .. 
          (6.96, 5.70) .. controls (6.11, 5.74) and (5.49, 4.95) .. (5.46, 4.05);
    \draw (5.44, 3.41) .. controls (5.41, 2.57) and (6.03, 1.84) .. 
          (6.85, 1.84) .. controls (7.67, 1.84) and (8.33, 2.49) .. (8.35, 3.30);
  
    \draw (0.62, 4.17) .. controls (0.63, 6.15) and (2.80, 7.20) .. 
          (4.99, 7.24) .. controls (7.22, 7.29) and (9.29, 5.94) .. (9.21, 3.91);
    \draw (9.18, 3.27) .. controls (9.11, 1.29) and (6.99, 0.21) .. 
          (4.83, 0.21) .. controls (2.65, 0.21) and (0.61, 1.51) .. (0.62, 3.53);
  
    \draw (0.21, 4.60) .. controls (0.26, 4.26) and (0.32, 3.86) .. (0.62, 3.85);

    \draw (0.62, 3.85) .. controls (0.95, 3.85) and (1.29, 3.84) .. (1.62, 3.83);
    \draw (1.62, 3.83) .. controls (1.96, 3.83) and (2.18, 4.17) .. (2.23, 4.53);
  
    \draw (4.14, 4.43) .. controls (4.27, 4.11) and (4.43, 3.75) .. (4.74, 3.74);
    \draw (4.74, 3.74) .. controls (4.97, 3.74) and (5.21, 3.73) .. (5.45, 3.73);
    \draw (5.45, 3.73) .. controls (5.79, 3.72) and (6.02, 4.02) .. (6.20, 4.32);

    \draw (7.58, 4.28) .. controls (7.66, 3.91) and (7.97, 3.64) .. (8.35, 3.62);
    \draw (8.35, 3.62) .. controls (8.63, 3.61) and (8.92, 3.60) .. (9.20, 3.59);
    \draw (9.20, 3.59) .. controls (9.49, 3.58) and (9.62, 3.92) .. (9.74, 4.21);

\end{tikzpicture}
\begin{tikzpicture}[scale=.85]
  \draw[line width=0.85mm, red ](0.5,2.5)--(1,1.5);
  \draw[line width=0.85mm, red](0.5,2.5)--(1,3.5);
  \draw[line width=0.85mm, red](2,1.5)--(1,1.5);
  \draw[line width=0.85mm, red](1,3.5)--(2,3.5);
  \draw[line width=0.85mm, red](2,3.5)--(2.5,2.5);
   \draw[line width=0.85mm, red](2,1.5)--(2.5,2.5);
  
   \draw[line width=0.85mm, red](3,2.5)--(3.5,3.5);
   
   \draw[line width=0.85mm, red](3,2.5)--(3.5,1.5);
   \draw[line width=0.85mm, red](4.5,3.5)--(3.5,3.5);
   
   \draw[line width=0.85mm, red](4.5,1.5)--(3.5,1.5);
   \draw[line width=0.85mm, red](4.5,3.5)--(5,2.5);
   \draw[line width=0.85mm, red](4.5,1.5)--(5,2.5);
   
    \draw[line width=0.85mm, red](6,1.5)--(5.5,2.5);
    
   \draw[line width=0.85mm, red](6,3.5)--(5.5,2.5);
   
   \draw[line width=0.85mm, red](6,1.5)--(7,1.5);
    \draw[line width=0.85mm, red](6,3.5)--(7,3.5);
   
    \draw[line width=0.85mm, red](7.5,2.5)--(7,1.5);
   
   \draw[line width=0.85mm, red](7.5,2.5)--(7,3.5); 
   \fill[gray!20](7.5,2.5)--(7,3.5)--(7,1.5);
  \fill[gray!20] 
    (0.5,2.5)--(1,1.5)--(0.5,2.5)--(1,3.5)--(2,1.5)--(1,1.5)--(1,3.5)--(2,3.5)--(2,3.5)--(2.5,2.5)--(2,1.5)--(2.5,2.5);
  \fill[gray!20](.5,2.5) --(1,1.5)--(1,2.5);
  
  \fill[gray!20](1.5,2.5) --(2,1.5)--(2.5,2.5);
  
  \fill[gray!20]  (3,2.5)--(3.5,3.5)--(3.5,1.5)--(3,2.5)--(4.5,3.5)--(3.5,3.5)--(4.5,3.5)--(5,2.5)--(4.5,1.5)--(5,2.5);
  \fill[gray!20] (3.5,3.5) rectangle (4.5,1.5);
 \fill[gray!20]  (4.5,1.5)--(5,2.5)--(4.5,2.5);

  \fill[gray!20]  
  (6,1.5)--(5.5,2.5)--(6,3.5)--(5.5,2.5)--(6,1.5)--(7,1.5)--(6,3.5)--(7,3.5)--(7.5,2.5)--(7,1.5)--(7.5,2.5)--(7,3.5); 
 \fill[gray!20]  
   (6,1.5)--(5.5,2.5)--(6,3.5)--(7.5,2.5);
 \fill[gray!20]  (7,1.5)--(6.5,2.5)--(7,3.5);
  \draw[](1,1.5).. controls (4,0) .. (7,1.5); 
  \draw[](1,3.5).. controls (4,4.5) .. (7,3.5);
  
   \draw[](2,3.5).. controls (2.75,4) .. (3.5,3.5);

  \draw[](2,1.5).. controls (2.75,0.8) .. (3.5,1.5);
    \draw[](4.5,1.5).. controls (5.25,0.8) .. (6,1.5);
  \draw[](4.5,3.5).. controls (5.25,4) .. (6,3.5);
  \draw[](.5,2.5)--(2.5,2.5);
  
   \draw[](3,2.5)--(5,2.5);
   \draw[](5.5,2.5)--(7.5,2.5);
  
  \end{tikzpicture}
    \begin{tikzpicture}
  
   \draw[red, ultra thick](1,2.5)--(1.5,1.5); 
    
     \draw[red, ultra thick](1,2.5)--(1.5,3.5); 
    
     \draw[red, ultra thick](2,2.5)--(1.5,1.5); 
     \draw[red, ultra thick](2,2.5)--(1.5,3.5); 
     \draw[red, ultra thick](2,2.5)--(2.5,1.5); 
    
     \draw[red, ultra thick](2,2.5)--(2.5,3.5); 
     \fill[gray!20]
     (2,2.5)--(1.5,1.5)--(2.5,1.5);
     \fill[gray!20]
     (2,2.5)--(1.5,3.5)--(2.5,3.5);
     
    \draw[red, ultra thick](3,2.5)--(2.5,1.5); 
     
     \draw[red, ultra thick](3,2.5)--(2.5,3.5); 
      \draw[red, ultra thick](1.5,1.5)--(2.5,1.5); 
        \draw[red, ultra thick](1.5,3.5)--(2.5,3.5); 
     
     \fill[gray!20](2,4.5).. controls (.6,4)..(1,2.5);
     \draw[red, ultra thick](2,4.5).. controls (.6,4)..(1,2.5);
     
    \fill[gray!20](2,.5).. controls (.6,1)..(1,2.5);

     \draw[red, ultra thick](2,.5).. controls (.6,1)..(1,2.5);
     
    \fill[gray!20](2,4.5).. controls (3.4,4)..(3,2.5);

     \draw[red, ultra thick](2,4.5).. controls (3.4,4)..(3,2.5);
     \fill[gray!20](3,2.5).. controls (3.4,1)..(2,.5);
      \draw[red, ultra thick](3,2.5).. controls (3.4,1)..(2,.5);
      
      \draw[red, ultra thick](2,4.5)--(1.5,3.5);
      \draw[red, ultra thick](2,.5)--(1.5,1.5);
     \draw[red, ultra thick](2,.5)--(2.5,1.5);
     \draw[red, ultra thick](2,4.5)--(2.5,3.5);
     
     \fill (1,2.5) circle (.081);
    \fill (2,2.5) circle (.081);
      \fill (3,2.5) circle (.081);
     \fill (1.5,1.5) circle (.081);  \fill (2.5,1.5) circle (.081);  \fill (1.5,3.5) circle (.081);  \fill (2.5,3.5) circle (.081);
    \fill (2,.5) circle (.081);   \fill (2,4.5) circle (.081);
  
  \end{tikzpicture}
\hfill \hfill (a) \hfill \hfill (b) \hfill \hfill (c) \hfill \hfill
 
\caption[cut-slice-flatten]{Polyhedral decomposition of a FAL using the cut-slice-flatten method.
(a) Cut the FAL complement in half along the projection
plane. This also cuts the crossing circles and the bounded twice
punctured discs in half. (b) Slice open the half discs and
flattened down on the projection plane. (c) Polyhedron $P_L$ is obtained by contracting the link components to ideal
vertices. }
\label{csf}
\end{figure}

Given a FAL diagram $L$, we can obtain the polyhedra decomposition by
using a construction given by Agol and D. Thurston in \cite{L2004}
called the \emph{cut-slice-flatten method}.  Assume that the twice
punctured discs are perpendicular to the plane. First, cut the link
complement in half along the projection plane, which cuts the twice
punctured disc bounded by the crossing circle into half. This creates
a pair of polyhedra, see Figure \ref{csf}(a).  For each half, slice
open the half disc like a pita bread and flatten it down on the
projection plane, see Figure \ref{csf}(b).  Lastly, shrink the link
components to ideal vertices, see Figure \ref{csf}(c).  This gives us
two copies of a polyhedron which we denote as $P_L$.  For each
crossing circle we get a \emph{bowtie} on each copy of $P_L$, which
consists of two triangular faces that share the ideal vertex
corresponding to the crossing circle component. The cut-slice-flatten
method is part of the proof of Proposition 2.2 in \cite{purcell-fal},
which we state below:

\begin{prop}\label{PLprop}
  \cite{purcell-fal, 2009arXiv0903.5288C} Let $L$ be a hyperbolic FAL
  diagram.  There is a decomposition of $S^3 \setminus L$ into two
  copies of geodesic, ideal, hyperbolic polyhedron $P_L$ with the
  following properties.
\begin{enumerate}
\item Faces of $P_L$ can be checkerboard colored, with shaded faces
  corresponding to bowties, and white faces corresponding to the
  regions of the FAL components in the projection plane.
    \item Ideal vertices of $P_L$ are all $4$\textendash valent.
    \item The dihedral angle at each edge of $P_L$ is $\frac{\pi}{2}$. 
\end{enumerate}
\end{prop}

\subsubsection{Gluing the Polyhedra}
For FAL with or without half-twists the polyhedron $P_L$ is the same.
The difference is in how they glue up. For FAL without half-twist the
shaded faces glue up such that the bowties on each polyhedron glue to
each other, see Figure \ref{bowtieglue} leftmost, and then the white faces on
each polyhedron get glued to their respective copies.  Whereas in the
case a half-twist occurs, the shaded faces get glued to the opposite
shaded face on the other polyhedron, see Figure \ref{bowtieglue} rightmost, and
then the white faces on each polyhedron get glued to their respective
copies.  Right handed and left handed twists produce the same link
complement due to the presence of the crossing circle as one can
add/delete full twists without changing the link complement. In \S4 we
use this gluing to study the fundamental domain of a cusp.

\begin{figure}
    \begin{tikzpicture}[thick,scale=0.2, every node/.style={transform shape}][line width=2.5, line cap=round, line join=round]
   \fill[gray!20](5,7.4) ellipse (4.5 and 3.7);
    \draw (6.91, 4.16) .. controls (8.49, 4.08) and (9.60, 5.62) .. 
          (9.62, 7.32) .. controls (9.63, 8.94) and (9.07, 10.62) .. (7.66, 10.63);
    \draw (6.74, 10.64) .. controls (5.60, 10.66) and (4.46, 10.67) .. (3.32, 10.68);
    \draw (2.40, 10.69) .. controls (1.03, 10.71) and (0.37, 9.15) .. 
          (0.33, 7.60) .. controls (0.29, 5.99) and (1.15, 4.45) .. (2.64, 4.37);
    \draw (2.64, 4.37) .. controls (4.06, 4.30) and (5.48, 4.23) .. (6.91, 4.16);
 \draw[ultra thick, red] (0.3,7.5) -- (2.6,7.5);
 \draw[ultra thick, red] (2.9,7.5) -- (6.9,7.5) ;
 \draw[ultra thick, red] (7.1,7.5) -- (9.7,7.5) ;

    \draw (3.01, 14.95) .. controls (2.96, 13.53) and (2.91, 12.11) .. (2.86, 10.69);
    \draw (2.86, 10.69) .. controls (2.79, 8.73) and (2.72, 6.78) .. (2.65, 4.83);
    \draw (2.62, 3.91) .. controls (2.58, 2.78) and (2.54, 1.64) .. (2.50, 0.51);
 
    \draw (7.40, 15.00) .. controls (7.33, 13.55) and (7.27, 12.09) .. (7.20, 10.64);
    \draw (7.20, 10.64) .. controls (7.11, 8.63) and (7.02, 6.63) .. (6.93, 4.62);
    \draw (6.89, 3.70) .. controls (6.84, 2.57) and (6.79, 1.44) .. (6.73, 0.31);
  \node at (5,6) {\huge$A$}; 
    \node at (5,9) {\huge$B$};

    
    
\end{tikzpicture}
\hspace{1cm}
\begin{tikzpicture}[thick,scale=0.6, every node/.style={transform shape}]
 \draw [red, ultra thick] (0,0) -- (1,2);
 \draw[red, ultra thick] (0,0) -- (2,0);
 \draw[red, ultra thick] (2,0) -- (1,2);

 \draw[red, ultra thick] (1,2) -- (0,4);
 \draw[red, ultra thick] (1,2) -- (2,4);
\draw[red, ultra thick] (0,4) -- (2,4);

 \fill [gray!30] (1,2) -- (0,4) -- (2,4);
 \fill [gray!30] (1,2) -- (0,0) -- (2,0);
  \fill  (0,0) circle (.1cm);
 \fill  (2,0) circle (.1cm);
 \fill  (0,4) circle (.1cm);
 \fill  (2,4) circle (.1cm);
 \fill  (1,2) circle (.1cm);
 
 \draw [red, ultra thick] (3,0) -- (4,2);
 \draw[red, ultra thick] (3,0) -- (5,0);
 \draw[red, ultra thick] (5,0) -- (4,2);

 \draw[red, ultra thick] (4,2) -- (3,4);
 \draw[red, ultra thick] (4,2) -- (5,4);
\draw[red, ultra thick] (3,4) -- (5,4);

 \fill [gray!30] (4,2) -- (3,4) -- (5,4);
 \fill [gray!30] (4,2) -- (3,0) -- (5,0);
  \fill  (3,0) circle (.1cm);
 \fill  (5,0) circle (.1cm);
 \fill  (3,4) circle (.1cm);
 \fill  (5,4) circle (.1cm);
 \fill  (4,2) circle (.1cm);

 \node at (.3,3.3) {\large$\vee$}; 
 \node at (.3,.7) {\large$\wedge$};
\node at (1.7,3.3) {\large$\vee$}; 
\node at (1.7,3.1) {\large$\vee$};  

\node at (1.7,.7) {\large$\wedge$}; 
\node at (1.7,.9) {\large$\wedge$};  
\node at (1,4)
 {\large\textgreater};
 \node at (1.2,4)
 {\large\textgreater};
\node at (1.4,4)
 {\large\textgreater}; 
 
 \node at (1,0)
 {\large\textgreater};
 \node at (1.2,0)
 {\large\textgreater};
\node at (1.4,0)
 {\large\textgreater}; 
 \node at (1,3) {$A$};
 \node at (1,1) {$A$};

 \node at (3.3,3.3) {\large$\bigtriangledown $}; 
 \node at (3.3,.7) {\large$\bigtriangleup $};
\node at (4.7,3.3) {\large$\bigtriangledown $}; 
\node at (4.7,3.1) {\large$\bigtriangledown $};  

\node at (4.7,.7) {\large$\bigtriangleup $}; 
\node at (4.7,.9) {\large$\bigtriangleup $};  
\node at (4,4)
 {\large$\rhd∗$};
 \node at (4.2,4)
 {\large$\rhd∗$};
\node at (4.4,4)
 {\large$\rhd∗$}; 
 
 \node at (4,0)
 {\large$\rhd∗$};
 \node at (4.2,0)
 {\large$\rhd∗$};
\node at (4.4,0)
 {\large$\rhd∗$}; 
 \node at (4,3) {$B$};
 \node at (4,1) {$B$};

\end{tikzpicture} 
\hspace{2cm}
\begin{tikzpicture}[scale=.2][line width=9.5, line cap=round, line join=round]
  \fill[gray!20](5,12.2) ellipse (5 and 4);
    \draw (2.56, 15.83) .. controls (1.00, 15.88) and (0.70, 13.88) .. 
          (0.59, 12.08) .. controls (0.46, 10.21) and (1.09, 8.28) .. (2.69, 8.30);
    \draw (2.69, 8.30) .. controls (4.02, 8.31) and (5.35, 8.32) .. (6.69, 8.34);
    \draw (6.69, 8.34) .. controls (8.39, 8.35) and (9.38, 10.14) .. 
          (9.40, 12.00) .. controls (9.43, 13.75) and (9.08, 15.64) .. (7.57, 15.69);
    \draw (6.35, 15.72) .. controls (5.48, 15.75) and (4.61, 15.77) .. (3.75, 15.80);
 
    \draw (8.61, 1.42) .. controls (7.66, 2.14) and (6.72, 2.87) .. (5.77, 3.59);
    \draw (4.66, 4.44) .. controls (3.54, 5.29) and (2.61, 6.46) .. (2.67, 7.84);
    \draw (2.72, 9.05) .. controls (2.81, 11.31) and (2.90, 13.56) .. (3.00, 15.82);
    \draw (3.00, 15.82) .. controls (3.06, 17.29) and (3.12, 18.76) .. (3.18, 20.23);
  
    \draw (7.36, 20.23) .. controls (7.27, 18.72) and (7.19, 17.21) .. (7.10, 15.70);
    \draw (7.10, 15.70) .. controls (6.98, 13.50) and (6.85, 11.29) .. (6.73, 9.09);
    \draw (6.65, 7.74) .. controls (6.58, 6.37) and (6.16, 5.00) .. (5.17, 4.05);
    \draw (5.17, 4.05) .. controls (3.95, 2.87) and (2.73, 1.68) .. (1.50, 0.50);
  \draw[ultra thick, red] (0.3,12) -- (2.6,12);
 \draw[ultra thick, red] (2.9,12) -- (6.9,12) ;
 \draw[ultra thick, red] (7.1,12) -- (9.7,12) ;
 \node at (5,10) {\tiny$A$};
\node at (5,14) {\tiny$B$};
\end{tikzpicture}
\hspace{1cm}
\begin{tikzpicture}[thick,scale=0.6, every node/.style={transform shape}]
 \draw [red, ultra thick] (0,0) -- (1,2);
 \draw[red, ultra thick] (0,0) -- (2,0);
 \draw[red, ultra thick] (2,0) -- (1,2);

 \draw[red, ultra thick] (1,2) -- (0,4);
 \draw[red, ultra thick] (1,2) -- (2,4);
\draw[red, ultra thick] (0,4) -- (2,4);

 \fill [gray!30] (1,2) -- (0,4) -- (2,4);
 \fill [gray!30] (1,2) -- (0,0) -- (2,0);
  \fill  (0,0) circle (.1cm);
 \fill  (2,0) circle (.1cm);
 \fill  (0,4) circle (.1cm);
 \fill  (2,4) circle (.1cm);
 \fill  (1,2) circle (.1cm);
 
 \draw [red, ultra thick] (3,0) -- (4,2);
 \draw[red, ultra thick] (3,0) -- (5,0);
 \draw[red, ultra thick] (5,0) -- (4,2);

 \draw[red, ultra thick] (4,2) -- (3,4);
 \draw[red, ultra thick] (4,2) -- (5,4);
\draw[red, ultra thick] (3,4) -- (5,4);

 \fill [gray!30] (4,2) -- (3,4) -- (5,4);
 \fill [gray!30] (4,2) -- (3,0) -- (5,0);
  \fill  (3,0) circle (.1cm);
 \fill  (5,0) circle (.1cm);
 \fill  (3,4) circle (.1cm);
 \fill  (5,4) circle (.1cm);
 \fill  (4,2) circle (.1cm);

 \node at (.3,3.3) {\large$\vee$}; 
 \node at (.3,.7) {\large$\bigtriangleup$};
 \node at (.3,.9) {\large$\bigtriangleup$};
\node at (1.7,3.3) {\large$\vee$}; 
\node at (1.7,3.1) {\large$\vee$};  

\node at (1.7,.7) {\large$\bigtriangleup$}; 
 
\node at (1,4)
 {\large\textgreater};
 \node at (1.2,4)
 {\large\textgreater};
\node at (1.4,4)
 {\large\textgreater}; 
 
 \node at (1,0)
 {\large$\lhd∗r$};
 \node at (1.2,0)
 {\large$\lhd∗$};
\node at (1.4,0)
 {\large$\lhd∗$}; 
 \node at (1,3) {$B$};
 \node at (1,1) {$A$};

 \node at (3.3,3.3) {\large$\bigtriangledown $}; 
 \node at (3.3,.7) {\large$\wedge $};
  \node at (3.3,.9) {\large$\wedge $};
\node at (4.7,3.3) {\large$\bigtriangledown $}; 
\node at (4.7,3.5) {\large$\bigtriangledown $};

\node at (4.7,.7) {\large$\wedge $}; 

\node at (4,4)
 {\large$\rhd∗$};
 \node at (4.2,4)
 {\large$\rhd∗$};
\node at (4.4,4)
 {\large$\rhd∗$}; 
 
 \node at (4,0)
 {\large\textless};
 \node at (4.2,0)
 {\large\textless};
\node at (4.4,0)
 {\large\textless}; 
 \node at (4,3) {$A$};
 \node at (4,1) {$B$};

\end{tikzpicture}
    \caption{Gluing bowties without half-twist (leftmost three figures) and when half-twist are present (rightmost three figures).}
    \label{bowtieglue}
\end{figure}
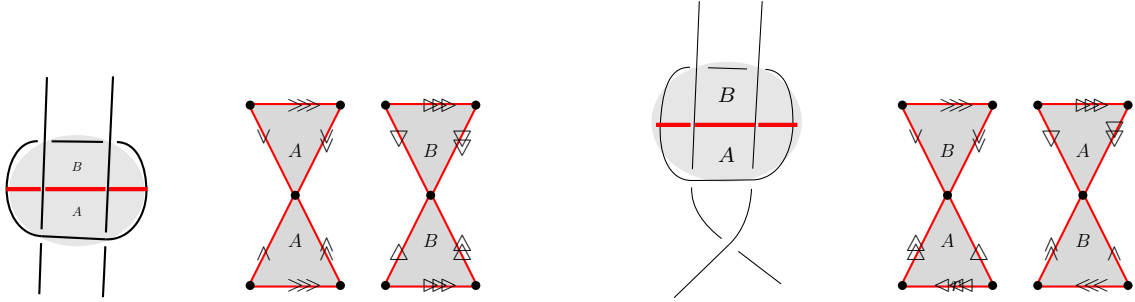

\subsubsection{Circle Packings and Cusp Shapes}

\begin{define}
A \emph{circle packing} is a finite collection of circles inside a given boundary such that no two overlap and some (or all) of them are mutually tangent.
\end{define} 

The geometry of FAL
complements is studied
using the hyperbolic structure
on $P_L.$ Since all faces of $P_L$ are geodesic, for each face, the hyperbolic plane it lies on determines a circle or line in $\mathbb{C} \cup \{ \infty \}.$ Purcell showed that there is a corresponding circle packing for the white geodesic faces of $P_L$, and a dual circle packing for the geodesic shaded faces of $P_L$. We can visualize $P_L$ if you place the two circle packings on top of one another, and intersect it with half-spaces in $\HH^3$.

In \cite{purcell-fal} Purcell described a technique to compute cusp shapes of FALs by examining the circle packings and the gluing of polyhedra. The main result of this paper is that we can extend the T-T method to fully augmented links, and determine the cusp shapes of FALs by solving an algebraic system of equations, see \S4. Since the equations are obtained directly from the FAL diagram, we can directly relate
the combinatorics of FAL diagrams
and the geometry of FAL complements.

\subsection{T-T Method}
We will work in the upper half-space model of hyperbolic 3-space $\HH^3$.

\begin{define}\label{tautdef} \cite{tt} A diagram of a hyperbolic link is \emph{taut} if each associated checkerboard
surface is incompressible and boundary incompressible in the link complement, and
moreover does not contain any simple closed curve representing an accidental parabolic.
\end{define} 

The taut condition implies that the faces in the diagram correspond to ideal polygons in $\HH^3$ with distinct vertices. 
Let $L$ be a taut, oriented link diagram of a hyperbolic link. A \emph{crossing} arc is an arc which runs from the overcrossing to the undercrossing. Let $R$ be a face in $L$ with $n$ crossings.  Then $R$ corresponds to an ideal polygon $F_R$ in $\mathbb{H}^3$ as follows: 
\begin{enumerate}

\item Distinct edges of $L$ around the boundary of  $R$ lift to distinct ideal vertices in $\HH^3$ because of the no accidental parabolic condition in Definition \ref{tautdef}.
 \item The lifts of the crossing arcs can be straightened out in $\HH^3$ to geodesic edges giving the ideal polygon $F_R$. See Figure \ref{FandPF}. 
\end{enumerate}

\begin{figure}
\begin{tabular}{llc}

 \begin{subfigure}[b]{0.4\textwidth}  
\begin{tikzpicture}[thick,scale=0.6, every node/.style={transform shape}]
\draw [very thick] (0,1.8) .. controls (-.3,6) .. (0,7.8);
\draw [thick, red] (0,7.4) -- (.5,8);
\draw[thick, red] (6,8) -- (6.8,7.5);
\draw[thick, red] (0,2) -- (1,1.2);
\draw[thick, red] (6,.65)--(6.5,2);
\draw[thick, red] (6.6,4.8)--(6.8,6.4);
\draw[very thick] (0,1.1) .. controls (-.051,-.5)..(.1,-1);
\draw [ very thick] (.1,8.3) .. controls (.15, 8.7) .. (.3,9);
\draw[ very thick] (-1,8.15) .. controls (4,7.9).. (6.7,8.1);
\draw[very thick] (7.1, 8.15) .. controls (7.5, 8.3) .. (7.9, 8.6);
\draw[very thick] (7,9) .. controls (6.7,7).. (7,4);
\draw[very thick] (6.9,-.6) .. controls (6.3,3.1).. (6.7,5.2);
\draw[very thick] (7,5.6) .. controls (7.5,6.2).. (8,6.5);
\draw [very thick] (-1.5,2) .. controls (3,.5) .. (6.45,.7);
\draw [very thick] (6.9,.8) .. controls (7.2,.85) .. (7.8,1.1);
\node at (3,5) {\huge$R$};\
\node at (1,1.9) {\huge$\omega_5$};
\node at (.5,4.7) {\huge$u_5$};
\node at (.9,7.1) {\huge$\omega_4$};
\node at (2.9,7.4) {\huge$u_4$};
\node at (6.1,7.5) {\huge$\omega_3$};
\node at (6.3,6.7) {\huge$u_3$};
\node at (6.2,5.6) {\huge$\omega_2$};
\node at (5.9,3.2) {\huge$u_2$};
\node at (5.6,1.5) {\huge$\omega_1$};
\node at (3,1.3) {\huge$u_1$};
\end{tikzpicture}
\end{subfigure}
\hspace{1cm}
\begin{subfigure}[b]{0.4\textwidth}
\begin{tikzpicture}[thick,scale=0.6, every node/.style={transform shape}]
\fill [gray!30] (0,2.9) circle (.8cm);
\fill [black] (0,2.2) circle (.1cm);
\fill [gray!30] (4,1) circle (.9cm);
\fill [gray!30] (6,6.9) circle (.7cm);
\fill [black] (4,.2) circle (.1cm);
\fill [gray!30] (1,6.9) circle (.8cm);
\draw [thick, red] (0,2.2) .. controls (2,5) ..  (4,.2);
\fill [gray!30](6.8,3.5) circle (.8cm);
\fill [black] (6.9,2.75) circle (.1cm);
\draw [thick, red] (4,.2) .. controls (5,5) .. (6.9,2.75);
\fill [black] (5.95,6.35) circle (.1cm);
\draw [thick, red] (6.9,2.75) .. controls (7,7.5) .. (5.95,6.35);
\draw [thick, red] (5.95,6.35).. controls (3.5,9.5) .. (1,6.2);

\fill [black] (1,6.2) circle (.1cm);
\draw [very thick, blue] (-.1,3.55) .. controls (.4,3.54)..(.65,3.1);
\draw [very thick, blue] (3.45,1.5) .. controls (3.8,1.8)..(4.3,1.65);
\draw [very thick, blue] (6.1,3.7) .. controls (6.5,4.2)..(6.9,4.1);
\draw [very thick, blue] (5.3,7.15) .. controls (6,6.9)..(6.65,7);
\draw [very thick, blue] (.2,7.1) .. controls (1,6.8)..(1.7,7.15);
\draw[thick, red] (1,6.2) .. controls (-.5,8) .. (0,2.2);
 \draw (-3.3,-.5) -- (8,-.5);
\draw (-3.3,-.5) -- (-.3, 9);
\draw (-.3, 9 ) -- (10,9 );
\draw (8,-.5) -- (10,9);
\node at (3,5.5) {\huge$F_R$};
\node at (2,4.4) {\huge$\omega_1$};
\node at (3.9,2.2) {\huge$u_2$};
\node at (4.8,4.6) {\huge$\omega_2$};
\node at (6,4.5) {\huge$u_3$};
\node at (6.2,5.5) {\huge$\omega_3$};
\node at (6,7.5) {\huge$u_4$};
\node at (1,7.5) {\huge$u_5$};
\node at (0.5,4) {\huge$u_1$};
\node at (3.8,7.8) {\huge$\omega_4$};
\node at (0.3,5.6) {\huge$\omega_5$};
\end{tikzpicture}
\end{subfigure}\\
  \quad  \quad \quad  \quad  \quad  \quad (a)  \quad  \quad \quad  \quad   \quad  \quad  \quad  \quad  \quad  \quad  \quad  \quad   \quad  \quad  \quad  \quad (b)   \\
\end{tabular}
    \caption[$R$ and $F_R$]{(a) Face (region) $R$ in a link diagram. (b) The corresponding ideal polygon $F_R$ in $\HH^3$.}
    \label{FandPF}
\end{figure}
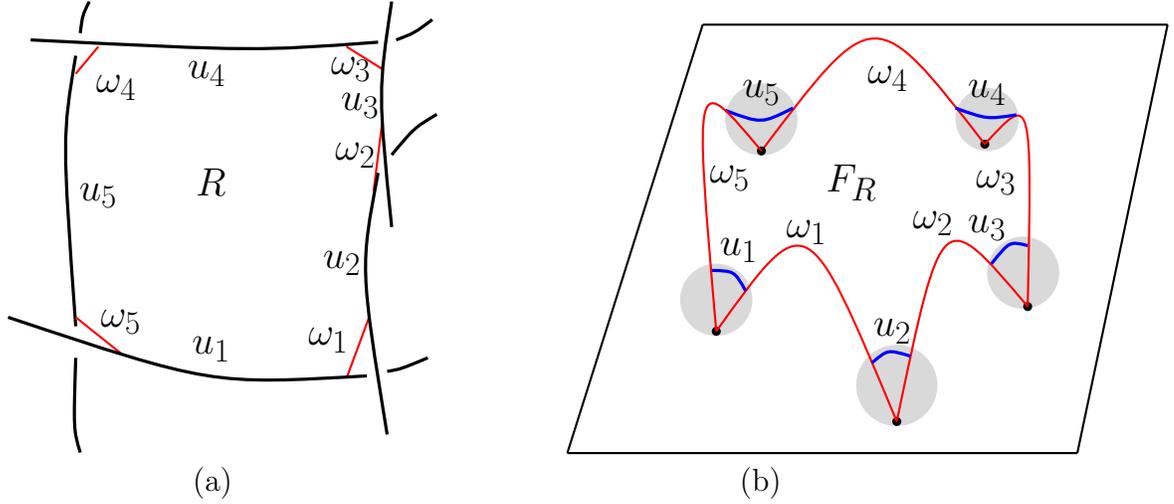

\begin{remark} 
Although the vertices of $F_R$ are ideal and the edges are geodesic, the face of $F_R$ need not be geodesic, i.e $F_R$ needs not lie on a hyperbolic plane in $\HH^3$.
\end{remark}

There are two types of parameters we will focus on in $R$. The first type of parameter is assigned to each crossing in $R$ and is known as the \emph{crossing label}, also referred to in the literature as \emph{crossing geodesic parameter}, or \emph{intercusp geodesic parameter}, denoted by $\omega_i$.  The second parameter we will focus on is assigned to the edges of $L$ in $R$, and is known as an \emph{edge label}, also referred to in the literature as \emph{translational geodesic parameter}, or \emph{edge parameter} and  denoted by $u_j$. See Figure \ref{FandPF}. 

\begin{remark} When we are in the diagram we refer to $\omega_i$ and $u_j$ as crossing and edge labels respectively.  When we are in $\HH^3$ we refer to them as intercusp and translational parameters, respectively. 
\end{remark}

We will choose a set of horospheres in $\HH^3$ such that for every cusp the meridian curve on the cross-sectional torus has length one.
Furthermore, we will choose one horosphere to be the Euclidean plane
$z = 1$.  It follows from
results of Adams
on waist size
of hyperbolic 3-manifolds \cite{Adams2002} that the horoballs are at most tangent and have disjoint interiors. 

 The lift of the crossing arc is a geodesic $\gamma$ in $\HH^3$ which is an edge of the ideal polygon $F_R$ and which travels from the center of one horoball to the center of an adjacent horoball.  For each horosphere the meridional direction along with geodesic $\gamma$ defines a hyperbolic half-plane.  The intercusp parameter $\omega_{\gamma}$, is defined as 
$|\omega_{\gamma}| = e^{-d} $ where $d$ is the hyperbolic distance between the horoballs along the geodesic $\gamma$, and the argument of $\omega_{\gamma}$ is the dihedral angle between these two half-planes, both of which contain $\gamma$. $\omega_\gamma$ encodes information about the intercusp translation taking into account distance and angles formed by parallel transport. The isometry that maps one horoball to another is represented up to conjugation by the $2 \times 2$ matrix 
$ \begin{bmatrix}
0&\omega_\gamma
\\1&0 
\end{bmatrix} $  in $GL(2,\mathbb{C})$, which maps horosphere $H_2$ to horosphere $H_1$ in Figure \ref{W}(a).

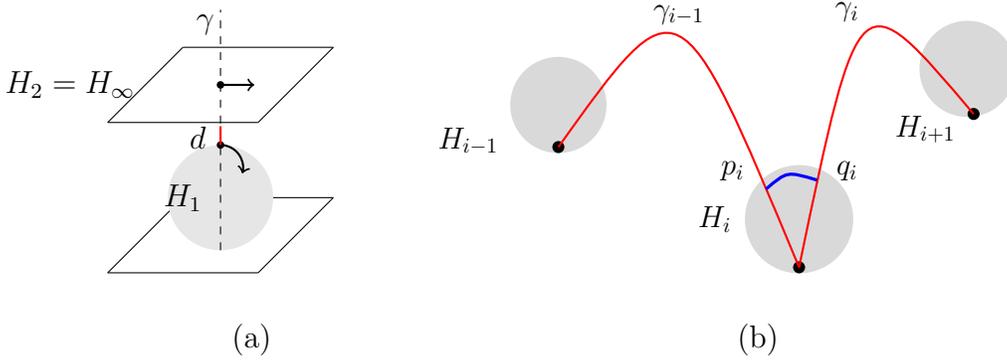
\begin{figure}
\centering
\begin{tikzpicture}
\draw (0,0) -- (1,1);
\draw (0,0)--(2,0);
\draw (2,0)--(3,1);
\draw (1,1)--(3,1);

\draw (0,2) -- (1,3);
\draw (0,2)--(2,2);
\draw (2,2)--(3,3);
\draw (1,3)--(3,3);
\fill [gray!20](1.5,1) circle (.7cm);
\draw [dashed](1.5,.3) -- (1.5,3.5);
\draw[thick, ->] (1.5,1.7) arc (90:-15:.3);
\fill (1.5,1.7) circle (.05cm);
\fill (1.5,2.5) circle (.05cm);
\draw[thick, ->] (1.5,2.5) -- (1.95,2.5);

\node at (1,1) {$H_1$};
\node at (1.3,3.3) {$\gamma$};
\node at (-.5,2.5) {$H_2=H_\infty$};
\draw [red, thick] (1.5,1.7) -- (1.5,1.95);
\node at (1.2,1.8) {$d$};
\end{tikzpicture}
\hspace{1cm}
\begin{tikzpicture}[thick,scale=.8, every node/.style={transform shape}]

\fill [gray!30] (0,2.9) circle (.8cm);
\fill [black] (0,2.2) circle (.1cm);
\fill [gray!30] (4,1) circle (.9cm);

\fill [black] (4,.2) circle (.1cm);

\draw [thick, red] (0,2.2) .. controls (2,5) ..  (4,.2);
\fill [gray!30](6.8,3.5) circle (.8cm);
\fill [black] (6.9,2.75) circle (.1cm);
\draw [thick, red] (4,.2) .. controls (5,5) .. (6.9,2.75);

\draw [very thick, blue] (3.45,1.5) .. controls (3.8,1.8)..(4.3,1.65);

\node at (2,4.4) {\large$\gamma_{i-1}$};
\node at (2.6,1) {\large$H_i$};
\node at (4.8,4.45) {\large$\gamma_i$};
\node at (2.9,1.8) {\large$p_i$};
\node at (4.8,1.8) {\large$q_i$};

\node at (6.1,2.5) {\large$H_{i+1}$};

\node at (-1.5,2.3) {\large$H_{i-1}$};

\end{tikzpicture}
(a)\quad \quad \quad \quad \quad \quad \quad \quad \quad \quad \quad \quad \quad \quad \quad(b)

\caption[Isometries]{(a) The isometry for the intercusp geodesic maps $H_\infty$ to $H_1$. (b) The isometry that maps along the translational geodesic points $p_i$ to $q_i$ or the reverse.}
\label{W}
\end{figure}

For each edge inside a region $R$ we assign \emph{edge labels}. The edges lift to ideal vertices of the polygon $F_R$ in $\HH^3$. The edge label $u_j$ represent the translation parameter along the horosphere centered at that ideal vertex that travels from one intercusp geodesic to another.  From  $u_j$ we can find the distance traveled along a horoball, and the direction of travel.  Since the edge label is a translation, up to conjugation it is represented by a $2 \times 2$ matrix:
$ \begin{bmatrix}
1&\epsilon_j u_j
\\0&1
\end{bmatrix} $,
where $\epsilon_j$ is positive if the direction of the edge in the diagram is the same as the direction of travel along the region, and negative otherwise, see Figure \ref{FandPF}(b), the matrix is an isometry translating one endpoint of the $u_i$ curve to the other end along the horosphere, i.e. it maps $p_i$ to $q_i$ or $q_i$ to $p_i$ along horosphere $H_i$, see Figure \ref{W}(b). 

We use the following conventions.  
\begin{enumerate}
    \item The basis of peripheral subgroups is the canonical meridian and longitude. The meridian is oriented using the right hand screw rule with respect to the orientation of the link.

\item The length of meridians along the horospherical cross section on a cusp are 1 \cite{Adams2002}. Consequently, there is a natural relationship between the two faces incident to the diagram that share an edge in the diagram: let $R$ and $S$ be adjacent regions that share an edge $u$, then the edge labels $u_R$ and $u_S$ satisfy $u_R - u_S = \pm 1 \rm{\ or\ }  0$  depending on whether the edge is going from
overpass to underpass, underpass to overpass, or staying leveled respectively from within region $R$. See Figure \ref{edgeequ}.
\begin{figure}
\centering

\definecolor{linkcolor0}{rgb}{0.85, 0.15, 0.15}
\definecolor{linkcolor1}{rgb}{0.15, 0.15, 0.85}
\definecolor{linkcolor2}{rgb}{0.15, 0.85, 0.15}
\begin{tikzpicture}[line width=2.4, line cap=round, line join=round, scale=.65]
    \draw (6.45, 7.15) .. controls (6.15, 6.94) and (5.86, 6.72) .. (5.56, 6.51);
    \draw (5.06, 6.15) .. controls (3.77, 5.23) and (2.61, 4.08) .. (2.16, 2.57);
    \draw (2.16, 2.57) .. controls (1.99, 2.02) and (1.83, 1.48) .. (1.67, 0.93);
  \begin{scope}[color=linkcolor1]
    \draw (8.08, 0.23) .. controls (8.35, 0.80) and (8.50, 1.46) .. (8.16, 1.98);
    \draw (7.87, 2.44) .. controls (7.02, 3.74) and (6.18, 5.05) .. (5.34, 6.35);
    \draw (5.34, 6.35) .. controls (5.07, 6.76) and (4.81, 7.17) .. (4.55, 7.58);
  \end{scope}
    \draw (9.72, 2.53) .. controls (9.16, 2.40) and (8.61, 2.27) .. (8.06, 2.14);
    \draw (8.06, 2.14) .. controls (6.21, 1.72) and (4.27, 1.80) .. (2.48, 2.45);
    \draw (1.83, 2.69) .. controls (1.30, 2.88) and (0.76, 3.07) .. (0.23, 3.26);
  \node at (6.25,3.75) {$u_R$};
  \node at (7.65,3.9) {$u_S$};
  \node at (4.7,3.4) {$R$};
  \node at (8.5,5) {$S$};
\end{tikzpicture}
\hspace{1cm}
\definecolor{linkcolor0}{rgb}{0.85, 0.15, 0.15}
\definecolor{linkcolor1}{rgb}{0.15, 0.15, 0.85}
\definecolor{linkcolor2}{rgb}{0.15, 0.85, 0.15}
\definecolor{linkcolor3}{rgb}{0.15, 0.85, 0.85}
\definecolor{linkcolor4}{rgb}{0.85, 0.15, 0.85}
\definecolor{linkcolor5}{rgb}{0.85, 0.85, 0.15}
\begin{tikzpicture}[thick,scale=0.6, every node/.style={transform shape}][line width=3.7, line cap=round, line join=round]
  \begin{scope}[color=linkcolor1]
    \draw (7.01, 0.20) .. controls (6.65, 0.65) and (6.28, 1.11) .. (5.92, 1.56);
   \draw (5.01, 2.70) .. controls (4.29, 3.61) and (3.57, 4.51) .. (2.86, 5.41);
   \draw (2.86, 5.41) .. controls (2.68, 5.63) and (2.51, 5.85) .. (2.33, 6.07);
   \draw (2.33, 6.07) .. controls (1.86, 6.66) and (1.39, 7.24) .. (0.93, 7.83);
  \end{scope}
  \begin{scope}[color=linkcolor1]
    \draw (1.88, 8.26) .. controls (2.24, 7.80) and (2.60, 7.34) .. (2.96, 6.88);
    \draw (2.96, 6.88) .. controls (3.13, 6.66) and (3.31, 6.43) .. (3.48, 6.21);
    \draw (3.48, 6.21) .. controls (4.28, 5.18) and (5.08, 4.15) .. (5.88, 3.12);
    \draw (6.80, 1.94) .. controls (7.09, 1.57) and (7.38, 1.20) .. (7.67, 0.83);
  \end{scope}
    \draw (4.36, 8.73) .. controls (3.95, 8.19) and (3.55, 7.66) .. (3.14, 7.12);
    \draw (2.15, 5.83) .. controls (1.50, 4.98) and (0.85, 4.12) .. (0.20, 3.27);
    \draw (5.06, 8.23) .. controls (4.59, 7.63) and (4.13, 7.04) .. (3.66, 6.44);
    \fill [black] (2.2,5.5) circle (.1cm);
    \fill [black] (6.1,2.2) circle (.1cm);
    \draw[red, very thick] (2.2,5.5) parabola (2.4,5.9);
     \draw[red,very thick, dashed] (2.5,5.9) parabola (3.1,6.2);
     \draw[green,  dashed, very thick] (3.1,6) parabola (4.4,4.9);
    \draw[green,very thick] (4.4,4.9) parabola (5.4,3); 
      \draw[green, dashed, very thick] (3.1,6) parabola (3.7,4.5);
      \draw[green, very thick]  (5.4,3) parabola  (3.7,4.5);
      \draw[red, very thick]  (5.6,2.7) parabola (5.4,3);
      \draw[red, dashed, very thick]  (6.1,2.2) parabola (5.6,2.7);
      \draw [dashed] (2.4,7.3) arc[x radius=.45cm, y radius=.45cm, start angle=0, end angle=-180];
\draw  (2.4,7.3) arc[x radius=.45cm, y radius=.45cm, start angle=0, end angle=180];

       \draw [dashed] (4.26,7.42) arc[x radius=.45cm, y radius=.45cm, start angle=0, end angle=-180];
\draw  (4.26,7.42) arc[x radius=.45cm, y radius=.45cm, start angle=0, end angle=180];
      
    \draw [dashed] (5.2,1.8) arc[x radius=.45cm, y radius=.45cm, start angle=0, end angle=-180];
\draw  (5.2,1.8) arc[x radius=.45cm, y radius=.45cm, start angle=0, end angle=180];
    \draw (2.67, 5.17) .. controls (2.02, 4.34) and (1.38, 3.51) .. (0.73, 2.68);
    \draw (9.29, 4.43) .. controls (8.21, 3.92) and (7.14, 3.40) .. (6.06, 2.89);
    \draw (6.06, 2.89) .. controls (5.77, 2.75) and (5.49, 2.61) .. (5.20, 2.47);
    \draw (5.20, 2.47) .. controls (4.28, 2.03) and (3.36, 1.60) .. (2.45, 1.16);
    \draw (9.75, 3.54) .. controls (8.71, 3.08) and (7.66, 2.63) .. (6.62, 2.18);
    \draw (6.62, 2.18) .. controls (6.32, 2.05) and (6.03, 1.92) .. (5.74, 1.79);
    \draw (5.74, 1.79) .. controls (4.72, 1.35) and (3.70, 0.91) .. (2.68, 0.46);
 \node at (3,4) {\huge$\alpha_{u_R}$};
  \node at (5.8,4.5) {\huge$\alpha_{u_S}$};
  
\end{tikzpicture}
(a) \quad \quad \quad \quad \quad \quad \quad \quad \quad \quad \quad \quad \quad \quad \quad \quad (b)
\caption[$u_i$ and $u_j$]{(a) The relation between the two sides of an edge. (b) The corresponding meridian loop along the solid torus boundary giving the relation.}
\label{edgeequ}
\end{figure}
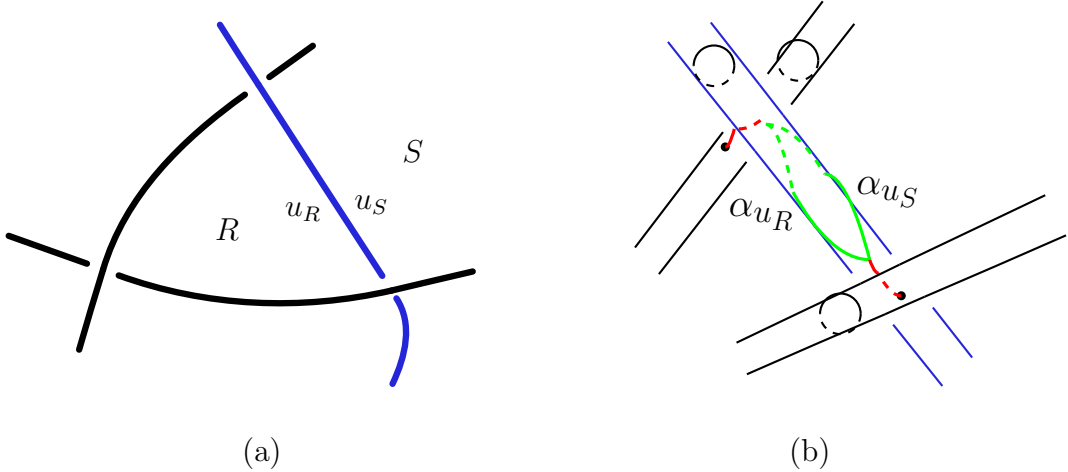

\item The edge labels inside a bigon are zero.
\end{enumerate}

\begin{remark}
For convention 2 above, this relationship holds if the actual region in the diagram corresponds to the ideal polygonal face in $\HH^3$ as described above. In this case, the translation on either side of the region will start and end with the same intercusp geodesics as the other side.  However, below we will see that for fully augmented links, the faces coming from the crossing circles will not be the faces from the diagram directly, but will require the polyhedral decomposition of the complement first.  In this case the edges will not share the same intercusp geodesics, thus the above relationship will not necessarily hold, and will require modification.  
\end{remark}
\begin{define}
 
Let the ideal vertices of the $n$-sided ideal polygon $F_R$ corresponding to the face $R$ be $z_1,...,z_n$.   We will assign a shape parameter to each edge of the polygon as follows: Let $\gamma_i$ be a geodesic edge between ideal vertices $z_i$ and $z_{i+1}$  then its \emph{shape parameter} $\xi_i$ is defined as
$$\xi_i=\frac{(z_{i-1}-z_i)(z_{i+1}-z_{i+2})}{(z_{i-1}-z_{i+1})(z_i-z_{i+2})},$$
which is the cross-ratio of four consecutive vertices of $F_R$.
\end{define}
Thistlethwaite and Tsvietkova show that the above shape parameter can be written in terms of crossing and edge labels in Proposition 4.1 in \cite{tt}. For our purposes all the faces in our class of links will have total geodesic faces as we shall see below.

\begin{prop}\label{shapeparameterprop}
 \cite{tt} Up to complex conjugation, $\displaystyle{\xi_i=\frac{\pm\omega_i}{u_iu_{i+1}}}$ where the sign is positive if both edges are directed away or both are directed toward the crossing, negative if
one edge is directed into the crossing and one is directed out.
\end{prop}

\begin{proof}
  Let $z_0, z_1, z_2, z_3$ be four consecutive ideal vertices in $\HH^3$ that correspond to the edges with edge labels $u_0, u_1, u_2, u_3$ respectively in the link diagram $L$, 
  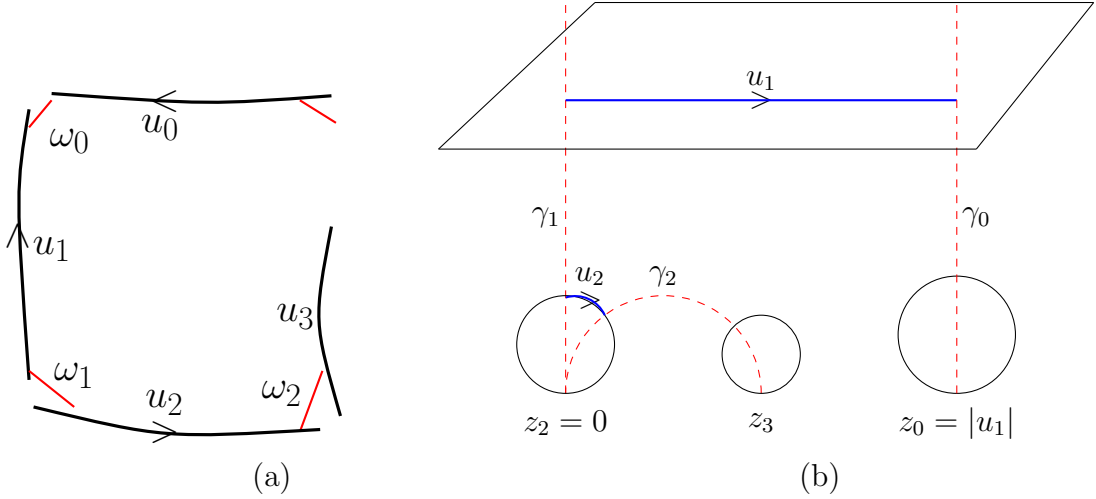
\begin{figure}
\centering

\begin{tikzpicture}[thick,scale=0.6, every node/.style={transform shape}]

\draw [very thick] (0,1.8) .. controls (-.3,6) .. (0,7.8);
\draw [thick, red] (0,7.4) -- (.5,8);
\draw[thick, red] (6,8) -- (6.8,7.5);
\draw[thick, red] (0,2) -- (1,1.2);
\draw[thick, red] (6,.65)--(6.5,2);

\draw[ very thick] (.5,8.15) .. controls (4,7.9).. (6.7,8.1);

\draw[very thick] (6.9,1) .. controls (6.3,3.1).. (6.7,5.2);

\draw [very thick] (0.1,1.2) .. controls (3,.5) .. (6.45,.7);

\node at (1,1.9) {\huge$\omega_1$};
\node at (.5,4.7) {\huge$u_1$};
\node at (.9,7.1) {\huge$\omega_0$};
\node at (2.9,7.4) {\huge$u_0$};
\node at (3,7.98) {\huge\textless};
\node at (3,0.65) {\huge\textgreater};
\node at (-.25,5) {\huge$\wedge$}; 
\node at (5.9,3.2) {\huge$u_3$};
\node at (5.6,1.5) {\huge$\omega_2$};
\node at (3,1.3) {\huge$u_2$};
\end{tikzpicture}
\hspace{1cm}
\begin{tikzpicture}[scale=1.3]
\draw [red, dashed] (0,0) -- (0,4);
\draw [red, dashed] (2,0) arc (0:180:1);
\draw [] (0,.5) circle (.5);
\node at (.,-.3) {$z_2 = 0$};
\node at (2,-.3) {$z_3$};
\node [thick] at (2,3) {\textgreater};
\node [thick] at (.25,.95) {\textgreater};
\draw [blue, thick] (.4,.8) arc (24:107:.33);
\node at (.25,1.2) {$u_2$};
\draw (2,.4) circle (.4); 
\draw [blue, thick] (0,3) -- (4,3);
\node at (2,3.2) {$u_1$};
\draw (4,.6) circle (.6);
\node at (4,-.3) {$z_0 = |u_1|$};
\draw [red, dashed] (4,0) -- (4,4);
\node at (4.2, 1.8) {$\gamma_0$};
\node at (-.2, 1.8) {$\gamma_1$};
\node at (1, 1.2) {$\gamma_2$};
 \draw (-1.3,2.5) -- (4.2,2.5);
\draw (-1.3,2.5) -- (.3, 4);
\draw (.3,4 ) -- (5.4,4 );
\draw (4.2,2.5) -- (5.4,4);
\centering
\end{tikzpicture}
\\(a) \quad \quad \quad \quad \quad \quad \quad \quad \quad \quad \quad \quad \quad \quad \quad \quad (b)
\caption[$\omega$ and $u$ in $\mathbb{H}^3$]{(a) Edge and crossing labels in the link diagram. (b) The corresponding translational and intercusp geodesics in $\HH^3$.}
\label{LandH3}
\end{figure}
See Figure \ref{LandH3}.
  We can always perform an isometry and let $z_0$ be placed at $|u_1|$, $z_1$ at $\infty$ where the horoball $H_\infty$ is at Euclidean height $1$, $z_2$ at $(0,0,0)$. Let $\gamma_0$ connect $z_0$ to $z_1$, correspond to $\omega_0$ in $L$, $\gamma_1$ be the geodesic connecting $z_1$ to $z_2$ correspond to $\omega_1$ in $L$ and $\gamma_2$ connect $z_2$ to $z_3$ correspond to $\omega_2$ in $L$ .  The horoball $H_2$ has diameter $|\omega_1|$ since the hyperbolic distance between $H_\infty$ and $H_2$ is $log\frac{1}{|\omega_1|}$ and $|\omega| = e^{-d}$.  In \cite{tt} T-T showed that $u_2 = \frac{|\omega_1|}{|z_3|}$. Thus the shape parameter $\xi_1$ is $$\xi_1=\frac{(z_0-z_1)(z_2-z_3)}{(z_0-z_2)(z_1-z_3)} = \frac{z_3}{z_0} = \frac{\frac{|\omega_1|}{u_2}}{u_1} = \frac{|\omega_1|}{u_1u_2}.$$ If either $u_1$ or $u_2$ exclusively were going in the opposite direction then it will cause the shape parameter to be of different sign.  
\end{proof}
Let $R_i$ be a face in $L$, which corresponds to $F_{R_i}$ in the ideal polygon. Fix $F_{R_i}$, we can perform an isometry sending ideal vertices $z_{i-1}$, $z_i$, and $z_{i+1}$ to $1$, $\infty$, and $0$ respectively, then $z_{i+2}$ will be placed at $\xi_i$.  Since the region closes up, the collection of shape parameters for each region determines the isometry class of the associated ideal polygon. The shape parameters $\xi_i$ satisfy algebraic equations amongst themselves. For example, for a  3-sided region 
the shape parameters are equal to each other and are equal to $1$, while in a 4-sided region the sum of consecutive shape
parameters is equal to 1. For regions with $ n \geqslant 5$ we use Proposition 4.2 in \cite{tt} to determine the algebraic equations in terms of crossing and edge labels. 

For each region in the diagram there is an alternating sequence of edges and crossings until the region closes up.  The product of the corresponding matrices is a scalar multiple of the identity.  Consequently, we have a system of equations whose solution allows us to construct a discrete faithful representation of the complement. We state Proposition 4.2 in \cite{tt}.
\begin{prop}\label{matrixprop} 
  Let $R$ be a region of an oriented link diagram with $ n \geqslant 3$ sides,
  and, starting from some crossing of $R$, 
  let $$ u_1 , \omega_1 , u_2 , \omega_2 , . . . , u_n , \omega_n $$ be the
  alternating sequence of edge and crossing labels for $R$ encountered
  as one travels around the boundary of the region. Also, for
  $1 \leqslant i \leqslant n$ let $\epsilon_i = 1$ (resp.
  $\epsilon_i = -1$) if the direction of the edge corresponding to
  $u_i$ is with (resp. against) the direction of travel. Then the
  equation for $R$ is written as
  $$\overset{n}{\underset{i=1}{\prod}}
\Big( \begin{bmatrix} 0&\omega_i
    \\1&0
\end{bmatrix} 	
\begin{bmatrix}
1&\epsilon_i u_i
\\0&1
\end{bmatrix} 
\Big)
\sim
\begin{bmatrix}
1&0
\\0&1
\end{bmatrix}.
$$
\end{prop}

This can be done for each region of $L$, thus we have algebraic equations for each face $R$ of $L$ in terms of crossing and edge labels, the solutions to the algebraic equations allow us to construct the discrete faithful representation of the link group into $PSL_2(\mathbb{C})$.

It is proved in \cite{tt} that the solutions to the above system of equations is discrete. Thus we can eliminate the variables
and reduce this system of equations to a 1-variable polynomial, referred to as the \emph{T-T polynomial}. Neumann-Tsvietkova
\cite{nt} related the solutions to the invariant trace field of the link complement.

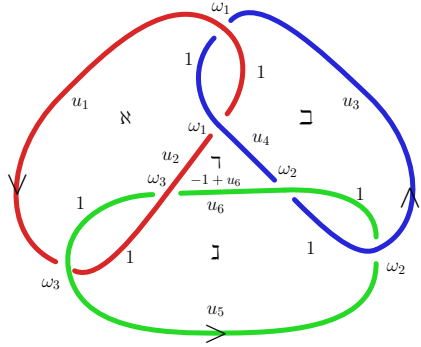
\begin{figure}[h]
\centering
    \definecolor{linkcolor0}{rgb}{0.85, 0.15, 0.15}
\definecolor{linkcolor1}{rgb}{0.15, 0.15, 0.85}
\definecolor{linkcolor2}{rgb}{0.15, 0.85, 0.15}
\begin{tikzpicture}[every node/.style={scale=.5},scale=.6, line width=1.9, line cap=round, line join=round]
  \begin{scope}[color=linkcolor0]
    \draw (5.17, 6.97) .. controls (4.03, 7.88) and (2.67, 6.60) .. 
          (1.52, 5.42) .. controls (0.34, 4.21) and (0.20, 2.38) .. (1.45, 1.79);
    \draw (1.87, 1.59) .. controls (2.43, 1.33) and (3.27, 2.45) .. (3.91, 3.30);
    \draw (3.91, 3.30) .. controls (4.23, 3.72) and (4.56, 4.15) .. (4.88, 4.58);
    \draw (5.25, 5.07) .. controls (5.71, 5.68) and (5.74, 6.53) .. (5.17, 6.97);
  \end{scope}
  \begin{scope}[color=linkcolor1]
    \draw (8.55, 2.06) .. controls (9.73, 2.57) and (9.58, 4.31) .. 
          (8.41, 5.44) .. controls (7.31, 6.49) and (5.95, 7.79) .. (5.33, 7.14);
    \draw (4.96, 6.75) .. controls (4.45, 6.21) and (4.52, 5.36) .. (5.06, 4.82);
    \draw (5.06, 4.82) .. controls (5.48, 4.42) and (5.89, 4.01) .. (6.30, 3.60);
    \draw (6.74, 3.17) .. controls (7.31, 2.60) and (8.07, 1.85) .. (8.55, 2.06);
  \end{scope}
  \begin{scope}[color=linkcolor2]
    \draw (6.52, 3.38) .. controls (7.50, 3.42) and (8.55, 3.18) .. (8.55, 2.34);
    \draw (8.55, 1.75) .. controls (8.55, 0.40) and (6.81, 0.22) .. 
          (5.25, 0.21) .. controls (3.69, 0.19) and (1.93, 0.28) .. (1.73, 1.66);
    \draw (1.73, 1.66) .. controls (1.60, 2.58) and (2.55, 3.25) .. (3.60, 3.28);
    \draw (4.22, 3.31) .. controls (4.98, 3.33) and (5.75, 3.36) .. (6.52, 3.38);
  \end{scope}
    \node[thick, black] at (5,.2) {\huge\textgreater};
      \node[thick, black] at (0.6,3.5) {\huge$\vee$};
        \node[thick, black] at (9.3,3.2) {\huge$\wedge$};
        \node at (1.35,1.3) {\large$\omega_3$};
         \node at (3.7,3.6) {\large$\omega_3$};
         \node at (9,1.65) {\large$\omega_2$}; \node at (6.6,3.8) {\large$\omega_2$};
         \node at (4.6,4.75) {\large$\omega_1$}; \node at (5.12,7.4) {\large$\omega_1$};
         \node at (2,5.3) {\large$u_1$};
        \node at (4.4,6.3) {\large$1$};\node at (2,3.1) {\large$1$};
        \node at (4,4.1) {\large$u_2$};\node at (3.1,1.9) {\large$1$};
        \node at (5,.7) {\large$u_5$};\node at (5,3) {\large$u_6$};
        \node at (7.1,2.1) {\large$1$};\node at (8.2,3.3) {\large$1$};\node at (8,5.3) {\large$u_3$};
        \node at (6,6) {\large$1$};\node at (6,4.5) {\large$u_4$};\node at (3,5) {\large$\aleph$};
       \node at (7,5) {\large$\beth$};\node at (5,2) {\large$\gimel$};\node at (5,4) {$\daleth$}; 
       \node at (5,3.6) {\small$-1+u_6$};

\end{tikzpicture}
    \caption{Hamantash Link}
    \label{magic}
\end{figure}
\subsection{Example: Hamantash Link }
Region $\aleph$:
This is a four-sided region with shape parameters:
$$\xi_1=\frac{\omega_1}{u_{1}}, \quad \xi_2=\frac{\omega_1}{u_{2}},
\quad \xi_3=\frac{\omega_3}{u_{2}}, \quad \xi_4=\frac{\omega_3}{u_{1}}.$$  Thus the equations are:
$$\frac{\omega_1}{u_{1}} + \frac{\omega_1}{u_{2}} = 1 , \quad \frac{\omega_1}{u_{2}} + \frac{\omega_3}{u_{2}} = 1,  \quad \frac{\omega_3}{u_{2}} + \frac{\omega_3}{u_{1}} = 1,  \quad \frac{\omega_3}{u_{1}} + \frac{\omega_1}{u_{1}} = 1$$ solving gives us the relations
$$u_{1} = u_{2} = 2\omega_1 = 2\omega_3, \quad \textrm{and} \quad \omega_1 = \omega_3.$$
Region $\beth$:
This is a four-sided region with shape parameters:
$$\xi_1=\frac{\omega_1}{u_{4}}, \quad \xi_2=\frac{\omega_1}{u_{3}},
\quad \xi_3=\frac{\omega_2}{u_{3}}, \quad \xi_4=\frac{\omega_2}{u_{4}}.$$  Thus the equations are:
$$\frac{\omega_1}{u_{4}} + \frac{\omega_1}{u_{3}} = 1 , \quad \frac{\omega_1}{u_{3}} + \frac{\omega_2}{u_{3}} = 1,  \quad \frac{\omega_2}{u_{3}} + \frac{\omega_2}{u_{4}} = 1,  \quad \frac{\omega_2}{u_{4}} + \frac{\omega_1}{u_{4}} = 1$$ solving gives us the relations
$$u_{3} = u_{4} = 2\omega_1 = 2\omega_2, \quad \textrm{and} \quad \omega_1 = \omega_2.$$
Region $\gimel$:
This is a four-sided region with shape parameters:
$$\xi_1=\frac{\omega_3}{u_{6}}, \quad \xi_2=\frac{\omega_2}{u_{6}},
\quad \xi_3=\frac{\omega_2}{u_{5}}, \quad \xi_4=\frac{\omega_3}{u_{5}}.$$  Thus the equations are:
$$\frac{\omega_3}{u_{6}} + \frac{\omega_2}{u_{6}} = 1 , \quad \frac{\omega_2}{u_{6}} + \frac{\omega_2}{u_{5}} = 1,  \quad \frac{\omega_2}{u_{5}} + \frac{\omega_3}{u_{5}} = 1,  \quad \frac{\omega_3}{u_{5}} + \frac{\omega_3}{u_{6}} = 1$$ solving gives us the relations
$$u_{5} = u_{6} = 2\omega_2 = 2\omega_3, \quad \textrm{and} \quad \omega_2 = \omega_3.$$
Region $\daleth$:
This is a three-sided region with edge labels $-1+u_2, \quad -1+u_4, \quad -1+u_6.$
$$\xi_1=\frac{-\omega_1}{(-1+u_2)(-1+u_4)} = 1, \quad \xi_2=\frac{-\omega_2}{(-1+u_4)(-1+u_6)} = 1, \quad
\xi_3=\frac{-\omega_3}{(-1+u_6)(-1+u_2)} = 1$$ solving these equations we get the T-T polynomial as $$4\omega_1^2 -3\omega_1 +1 = 0$$ thus $$\omega_i = \frac{3}{8}\pm \frac{\sqrt{7}}{8}i, \quad u_i =  \frac{3}{4}\pm \frac{\sqrt{7}}{4}i.$$

Neumann-Tsvietkova
proved in
\cite{nt} that one of the roots of the polynomial should give us the invariant trace field.

We can check the linear dependence using mathematica or pari-gp. 
It is suggested in \cite{tt} that the geometric solution will be the one that produces the highest volume, but finding the volume from the solutions can be difficult. In \S $5$ we will show how to find the geometric solution for the class of fully augmented links.

Using Snap the invariant trace field for the Hamantash link is $$x^2-x+2, \quad x = \frac{1}{2}+\frac{\sqrt{7}}{2}i \quad  \textrm{and} \quad \omega_1 = \frac{1+x}{4}.$$

\section{T-T Method and FAL}
\label{sec:TT-FAL}

In this section we show that the T-T method can be effectively extended to the class of fully augmented links. 
Our extension of the T-T method will be on a trivalent graph which is the intermediate step between the FAL diagram and the polyhedron $P_L.$ We denote the \emph{planar trivalent graph} $T_L$, for example see Figure \ref{csf}(b).  $T_L$ is in fact 
 the ideal polyhedron $P_L$, truncated at the ideal vertices, along with the orientations on the links of the vertices. Since the vertices are all 4-valent, the link of the vertices are all rectangles which tessellate the cusp torus. Thinking of the long thin rectangular pieces as thick edges in Figure \ref{TL} ($T_L$ for Borromean FAL), one gets the trivalent graph $T_L$. The components of the link diagram and the crossing geodesics are both visible on $T_L$. The crossing geodesics are on the boundary of the hexagonal regions corresponding to the bowties.  We will assign the edge labels and the crossing labels on this type of diagram for the T-T method to work. 

\begin{figure}
    \centering
    \includegraphics[height = 2 in, width = 3 in]{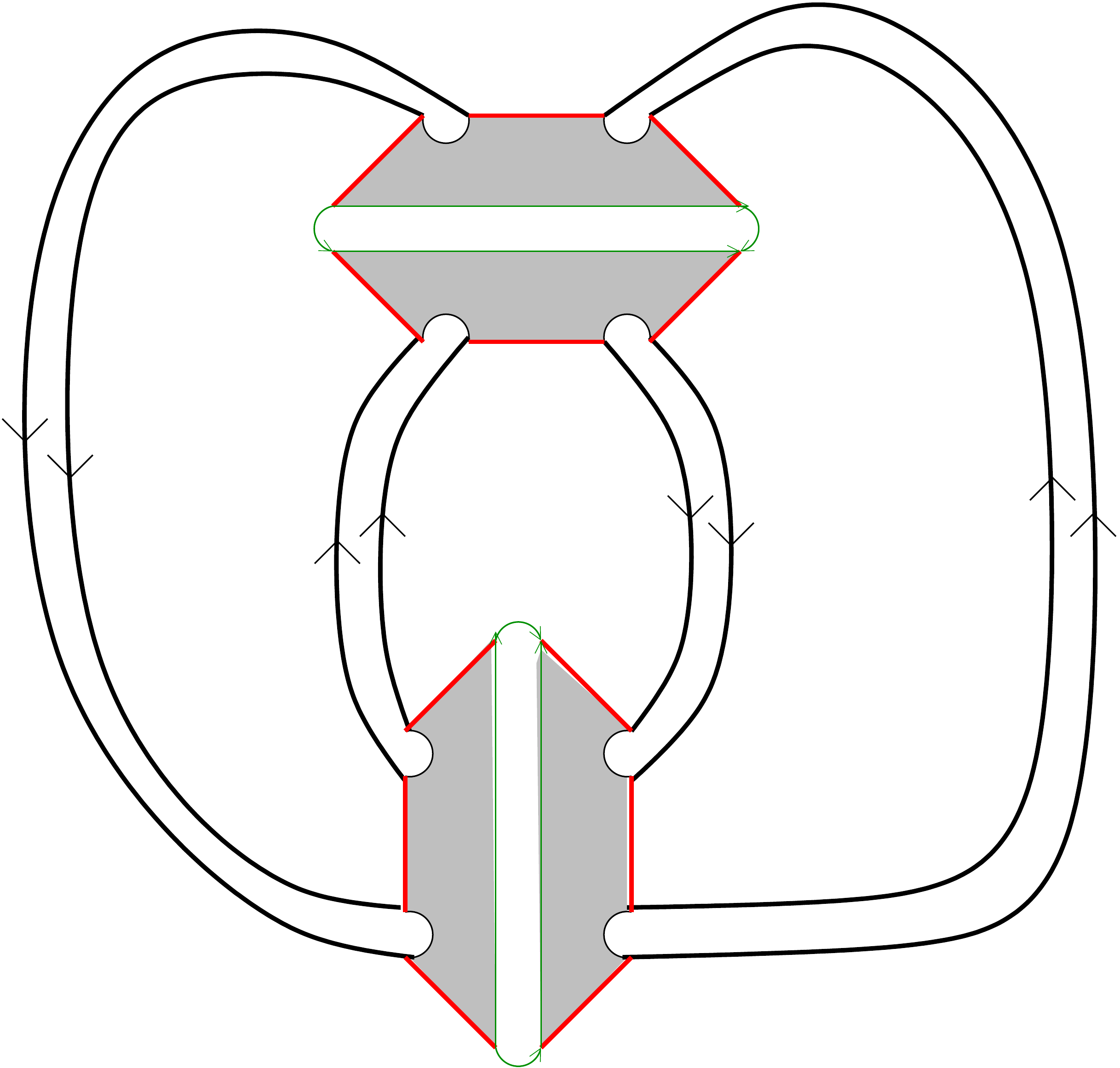}
    \caption{$T_L$ for Borromean FAL}
    \label{TL}
\end{figure}
In order to use the T-T method we need to ensure that it can be applied to the class of FALs. The tautness condition on the diagram is to ensure that the faces in the link diagram correspond to ideal polygons in $\HH^3$ with distinct vertices.  
\begin{prop}\label{polyhedprop} Let $L$ be a FAL diagram, then the planar trivalent graph $T_L$ is taut.

\end{prop}

\begin{proof} Let $L$ be a hyperbolic FAL. By Lemma 2.1 in \cite{purcell-fal}, the following surfaces are embedded totally geodesic surfaces in the link complement:  \begin{enumerate}
    \item twice punctured discs coming from the regions bounded by the crossing circles and punctured by two strands in the projection plane, and
    \item the surfaces in the projection plane.
\end{enumerate}    
Embedded totally geodesic surfaces are Fuchsian and Thurston's trichotomy for surfaces in 3-manifolds implies a surface can either be quasifuchsian, accidental, or semi-fibered \cite{futer2014quasifuchsian}. This implies they do not contain any accidental parabolics. We have a checkerboard coloring by shading the discs coming from the regions bounded by the crossing circles, and leaving the surfaces in the projection plane white. Thus by definition, the checkerboard surfaces of $T_L$ are incompressible and boundary incompressible. Hence $T_L$ is taut.
\end{proof}

Since the regions of $T_L$ correspond to geodesic faces in $\mathbb{H}^3$, the definitions for crossing and edge labels in the T-T method, the corresponding matrices, the shape parameters and equations of Propositions \ref{shapeparameterprop} and \ref{matrixprop} hold for $T_L$.  The fundamental difference is in the relationship between parameters for edges incident to adjacent faces as will be discussed below. 

\subsection{Thrice Punctured Sphere}
The twice punctured disc bounded by the crossing circle is geodesic and has the hyperbolic structure of the thrice
punctured sphere formed by gluing two ideal triangles. So we will refer to the twice punctured discs as thrice punctured spheres from now on.
Let $L$ be a FAL, then $L$ contains at least two crossing circles. This implies $S^3 \backslash L$ contains at least two thrice punctured spheres.  We will first study how the T-T method defines parameters on the thrice punctured sphere and use this as a basic building block for FALs.

The thrice punctured sphere has three components, two  strands that lie in the projection plane, and another circle component know as a crossing circle that encircles the two other strands. The thrice punctured sphere is known to be totally geodesic constructed by gluing two ideal triangles together along their edges \cite{Adams1985}.   There are two cases based on the orientation: one where the strands in the projection plane are parallel and the other when they are anti-parallel. On $T_L$ the thrice punctured sphere corresponds to a hexagon, and on $P_L$ it corresponds to a bowtie. We will study the part of $T_L$ corresponding to the hexagon.

\begin{theorem}\label{tps}
 
\begin{enumerate}
    \item The crossing labels on opposite sides of the augmented circles with parallel strands in the link diagram will be equal and the intercusp geodesic along the projection plane equals $-1/4$.
    \item The crossing labels on opposite sides of the augmented circles with anti-parallel strands in the link diagram will differ by sign and the intercusp geodesic along the projection plane equals $1/4$.
\end{enumerate}

  \end{theorem}

\begin{proof}

For each crossing circle there are four crossing labels $\omega_i$.  The two labels that share a bigon are equivalent since the region collapses and has the same geodesic arc going from horoball to horoball \cite{tt}. For the relationship between the two crossing labels that don't share a bigon we have two cases: 

Case 1: Parallel strands in the link diagram
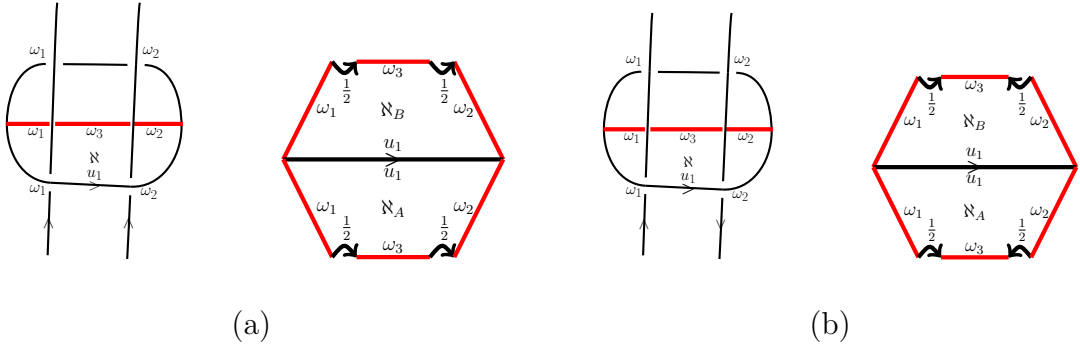
\begin{figure}
    \centering
    \begin{tikzpicture}[thick,scale=0.25, every node/.style={transform shape}][line width=1.5, line cap=round, line join=round]
    \draw (6.91, 4.16) .. controls (8.49, 4.08) and (9.60, 5.62) .. 
          (9.62, 7.32) .. controls (9.63, 8.94) and (9.07, 10.62) .. (7.66, 10.63);
    \draw (6.74, 10.64) .. controls (5.60, 10.66) and (4.46, 10.67) .. (3.32, 10.68);
    \draw (2.40, 10.69) .. controls (1.03, 10.71) and (0.37, 9.15) .. 
          (0.33, 7.60) .. controls (0.29, 5.99) and (1.15, 4.45) .. (2.64, 4.37);
    \draw (2.64, 4.37) .. controls (4.06, 4.30) and (5.48, 4.23) .. (6.91, 4.16);
 \draw[ultra thick, red] (0.3,7.5) -- (2.6,7.5);
 \draw[ultra thick, red] (2.9,7.5) -- (6.9,7.5) ;
 \draw[ultra thick, red] (7.1,7.5) -- (9.7,7.5) ;
    \draw (3.01, 13.95) .. controls (2.96, 13.53) and (2.91, 12.11) .. (2.86, 10.69);
    \draw (2.86, 10.69) .. controls (2.79, 8.73) and (2.72, 6.78) .. (2.65, 4.83);
    \draw (2.62, 3.91) .. controls (2.58, 2.78) and (2.54, 1.64) .. (2.50, 0.51);
    \draw (7.40, 13.95) .. controls (7.33, 13.55) and (7.27, 12.09) .. (7.20, 10.64);
    \draw (7.20, 10.64) .. controls (7.11, 8.63) and (7.02, 6.63) .. (6.93, 4.62);
    \draw (6.89, 3.70) .. controls (6.84, 2.57) and (6.79, 1.44) .. (6.73, 0.31);
    \node at (5,4.2) {\huge\textgreater};
    
   \node at (2.55,2.2) {\bf\huge$\wedge$}; 
    
    \node at (6.8,2.2) {\bf\huge$\wedge$}; 
   \node at (2,4) {\huge$\omega_1$}; 
   \node at (2,11.3) {\huge$\omega_1$};
      \node at (8,11.3) {\huge$\omega_2$};
      \node at (7.9,3.8) {\huge$\omega_2$}; \node at (5,7) {\huge$\omega_3$};
      \node at (1.9,7) {\huge$\omega_1$};
      \node at (8.2,7) {\huge$\omega_2$};
\node at (5,4.8) {\huge$u_1$};
\node at (5,5.7) {\huge$\aleph$};

\end{tikzpicture}
\hspace{1cm}
\begin{tikzpicture}[thick,scale=.65, every node/.style={transform shape}]
  \draw [red, ultra thick] (0,3) -- (1,5);
  \draw [ultra thick, ->] (1,5) .. controls (1.25,4.7) .. (1.5,5);
  \draw [red, ultra thick] (1.5,5) -- (3,5);
  \draw [red, ultra thick] (3.5,5) -- (4.5,3);
  \draw [ultra thick, ->] (3,5) .. controls (3.25,4.7) .. (3.5,5);
  \draw [ultra thick] (0,3) -- (4.5,3);
  
  \draw [red, ultra thick] (0,3) -- (1,1);
  \draw [ultra thick, <-] (1.5,1) .. controls (1.25,1.3) .. (1,1);
  \draw [red, ultra thick] (1.5,1) -- (3,1);
  \draw [ultra thick, ->] (3,1) .. controls (3.25,1.3) .. (3.5,1);
  \draw [red, ultra thick] (3.5,1) -- (4.5,3);
  \node at (2.25,4) {$\aleph_B$};
  \node at (2.25,2) {$\aleph_A$};
  \node at (2.25, 2.7) {$u_1$};
  \node at (2.25, 3) {\bf\textgreater};
  \node at (2.25,1.2) {$\omega_3$};
  \node at (.9,2) {$\omega_1$};
  \node at (3.7,2) {$\omega_2$};
  \node at (1.3,1.6) {$\frac{1}{2}$};
    \node at (3.3,1.6) {$\frac{1}{2}$};
  \node at (2.25, 3.3) {$u_1$};
 
  \node at (2.25,4.8) {$\omega_3$};
  \node at (.9,4) {$\omega_1$};
  \node at (3.7,4) {$\omega_2$};
  \node at (1.3,4.4) {$\frac{1}{2}$};
    \node at (3.3,4.4) {$\frac{1}{2}$};
\end{tikzpicture}
\hspace{1cm}
\begin{tikzpicture}[thick,scale=0.24, every node/.style={transform shape}][line width=1.5, line cap=round, line join=round]
    \draw (6.91, 4.16) .. controls (8.49, 4.08) and (9.60, 5.62) .. 
          (9.62, 7.32) .. controls (9.63, 8.94) and (9.07, 10.62) .. (7.66, 10.63);
    \draw (6.74, 10.64) .. controls (5.60, 10.66) and (4.46, 10.67) .. (3.32, 10.68);
    \draw (2.40, 10.69) .. controls (1.03, 10.71) and (0.37, 9.15) .. 
          (0.33, 7.60) .. controls (0.29, 5.99) and (1.15, 4.45) .. (2.64, 4.37);
    \draw (2.64, 4.37) .. controls (4.06, 4.30) and (5.48, 4.23) .. (6.91, 4.16);
    
    
 \draw[ultra thick, red] (0.3,7.5) -- (2.6,7.5);
 \draw[ultra thick, red] (2.9,7.5) -- (6.9,7.5) ;
 \draw[ultra thick, red] (7.1,7.5) -- (9.7,7.5) ;
    \draw (3.01, 13.95) .. controls (2.96, 13.53) and (2.91, 12.11) .. (2.86, 10.69);
    \draw (2.86, 10.69) .. controls (2.79, 8.73) and (2.72, 6.78) .. (2.65, 4.83);
    \draw (2.62, 3.91) .. controls (2.58, 2.78) and (2.54, 1.64) .. (2.50, 0.51);
    \draw (7.40, 13.95) .. controls (7.33, 13.55) and (7.27, 12.09) .. (7.20, 10.64);
    \draw (7.20, 10.64) .. controls (7.11, 8.63) and (7.02, 6.63) .. (6.93, 4.62);
    \draw (6.89, 3.70) .. controls (6.84, 2.57) and (6.79, 1.44) .. (6.73, 0.31);
    \node at (5,4.2) {\bf\huge\textgreater};
    
   \node at (2.55,2.2) {\bf\huge$\wedge$}; 
    
    \node at (6.8,2.2) {\bf\huge$\vee$}; 
   \node at (2,4) {\huge$\omega_1$}; 
   \node at (2,11.3) {\huge$\omega_1$};
      \node at (8,11.3) {\huge$\omega_2$};
      \node at (7.9,3.8) {\huge$\omega_2$}; 
     \node at (5,7) {\huge$\omega_3$};
\node at (5,4.8) {\huge$u_1$};
\node at (5,5.7) {\huge$\aleph$};
\node at (1.9,7) {\huge$\omega_1$};
      \node at (8.2,7) {\huge$\omega_2$};
\end{tikzpicture}
\hspace{1 cm}
\begin{tikzpicture}[thick,scale=.6, every node/.style={transform shape}]
  \draw [red, ultra thick] (0,3) -- (1,5);
  \draw [ultra thick, ->] (1,5) .. controls (1.25,4.7) .. (1.5,5);
  \draw [red, ultra thick] (1.5,5) -- (3,5);
  \draw [red, ultra thick] (3.5,5) -- (4.5,3);
  \draw [ultra thick,<-] (3,5) .. controls (3.25,4.7) .. (3.5,5);
  \draw [ultra thick] (0,3) -- (4.5,3);
  
  \draw [red, ultra thick] (0,3) -- (1,1);
  \draw [ultra thick, <-] (1.5,1) .. controls (1.25,1.3) .. (1,1);
  \draw [red, ultra thick] (1.5,1) -- (3,1);
  \draw [ultra thick, <-] (3,1) .. controls (3.25,1.3) .. (3.5,1);
  \draw [red, ultra thick] (3.5,1) -- (4.5,3);
   \node at (2.25,4) {$\aleph_B$};
  \node at (2.25,2) {$\aleph_A$};
  \node at (2.25, 2.7) {$u_1$};
  \node at (2.25, 3) {\bf\textgreater};
  \node at (2.25,1.2) {$\omega_3$};
  \node at (.9,2) {$\omega_1$};
  \node at (3.7,2) {$\omega_2$};
  \node at (1.3,1.6) {$\frac{1}{2}$};
    \node at (3.3,1.6) {$\frac{1}{2}$};
  \node at (2.25, 3.3) {$u_1$};
 
  \node at (2.25,4.8) {$\omega_3$};
  \node at (.9,4) {$\omega_1$};
  \node at (3.7,4) {$\omega_2$};
  \node at (1.3,4.4) {$\frac{1}{2}$};
    \node at (3.3,4.4) {$\frac{1}{2}$};
\end{tikzpicture} 
 (a) \quad \quad \quad \quad \quad \quad \quad \quad \quad \quad \quad \quad \quad \quad \quad \quad \quad (b) 
 \caption[Thrice punctured sphere with parallel strands]{(a) Thrice punctured sphere with parallel strands with the intercusp geodesics penciled in and corresponding region in $T_L$. 
  (b)Thrice punctured sphere with anti-parallel strands with the intercusp geodesics penciled in and corresponding region in $T_L$. Here we can see the intercusp and translational parameters in the $T_L$ regions.}
 \label{thricepuncturedspherecsf}
 \end{figure}

As the cusp torus for the crossing circle is cut in half, the translation parameters coming from the longitudinal strands in the projection plane will also be cut in half and are $1/2$ keeping with the convention that the meridional curve along the cross sectional torus has length $1$ and keeping with the right hand screw rule.  
For region $\aleph_A$ in Figure \ref{thricepuncturedspherecsf}(a) right, we have shape parameters:

$$\xi_1=\frac{\omega_1}{\frac{1}{2}u_1} = 1, \quad
\xi_2=\frac{-\omega_3}{\frac{1}{2}\times\frac{1}{2}} = 1, \quad 
\xi_3=\frac{\omega_2}{{\frac{1}{2}}u_1} = 1, \quad
$$ solving these equations gives us the relations $$\omega_3 = -\frac{1}{4}, \quad  \omega_1 =  \omega_2,   \quad \textrm{and} \quad u_1  =  2\omega_1.$$

Using Proposition \ref{matrixprop} we can check that these parameters are correct. Starting from the edge $\omega_1$ in the left side of region $\aleph_A$ and traveling counterclockwise we have:

$$\begin{bmatrix}
0&\omega_1
\\1&0 
\end{bmatrix} \begin{bmatrix}
1&1/2
\\0&1 
\end{bmatrix} \begin{bmatrix}
0&-1/4
\\1&0 
\end{bmatrix}  \begin{bmatrix}
1&\frac{1}{2}
\\0&1 
\end{bmatrix}  \begin{bmatrix}
0&\omega_2
\\1&0 
\end{bmatrix}  \begin{bmatrix}
1&-u_1
\\0&1 
\end{bmatrix}  =    \begin{bmatrix}
-\frac{\omega_1}{2}&0
\\0&-\frac{\omega_1}{2}
\end{bmatrix}. $$
 Case 2: Anti-parallel strands in the link diagram.

Notice here that the translation parameters coming from the longitudinal strands will be $\frac{1}{2}$ but their directions differ each going according to the right hand screw rule and the orientation on the strands, see Figure \ref{thricepuncturedspherecsf}(b).
For Region $\aleph_A$ we have shape parameters
$$\xi_1=\frac{\omega_1}{{\frac{1}{2}}u_1} = 1, \quad \xi_2=\frac{-\omega_2}{{\frac{1}{2}}u_1} = 1, \quad
\xi_3=\frac{\omega_3}{{\frac{1}{2}\times\frac{1}{2}}} = 1$$ solving these equations gives us the relations $$\omega_3 = \frac{1}{4}, \quad \omega_1 = -\omega_2, \quad \textrm{and} \quad u_1 = 2\omega_1.$$ 
Using Proposition \ref{matrixprop}, starting from the red edge $\omega_1$ in the left side of region $\aleph_A$ and traveling counterclockwise we have:

$$\begin{bmatrix}
0&\omega_1
\\1&0 
\end{bmatrix} \begin{bmatrix}
1&1/2
\\0&1 
\end{bmatrix} \begin{bmatrix}
0&1/4
\\1&0 
\end{bmatrix}  \begin{bmatrix}
1&-\frac{1}{2}
\\0&1 
\end{bmatrix}  \begin{bmatrix}
0&\omega_2
\\1&0 
\end{bmatrix}  \begin{bmatrix}
1&-u_1
\\0&1 
\end{bmatrix}$$ substituting in the above relations $  =    \begin{bmatrix}
-\frac{\omega_1}{2}&0
\\0&-\frac{\omega_1}{2}
\end{bmatrix} $.
\end{proof}

\subsection{Adaptation of T-T Method for FAL Diagram}
All the shaded faces on a FAL diagram come from thrice punctured spheres, and have parameters as determined in Theorem \ref{tps}. Hence to set up the T-T equations, we only need to understand the edge and crossing parameters on the regions in the projection plane. These regions are the white faces which have boundary alternating between the crossing geodesics and strands of the link diagram, except when it intersects the crossing circle. At the intersection of the region and the crossing circle, the boundary goes across a meridian of the crossing circle, see Figure \ref{TL}. Thus with an adjustment, we can write down the equations directly from the FAL diagram, without using $T_L$. 
\begin{lemma}\label{reverseorientation}
For each crossing circle, the two translational geodesics that correspond to the parts of the crossing circle component that bound the bigons, one going from $\omega_1$ to $\omega_1$ and the other going from $\omega_2$ to $\omega_2$ correspond to the meridional curve for that component, thus both are oriented the same way and are equal to $1$.
\end{lemma}

\begin{proof}

The part of the cusp torus corresponding to a crossing circle lies on the hexagon in $T_L$, such that the meridians lie flat in the projection plane. The orientations on the meridians are obtained by the right hand screw rule. The meridians on opposite sides of the crossing circle are homotopic and are oriented as in Figure \ref{changeorientation}(b).\end{proof} Consequently, starting from a FAL diagram, we can reorient parts of the crossing circle on the FAL diagram to agree with the orientations on the meridians. See Figure \ref{changeorientation}(c). With this adjustment we can now write the T-T equations directly from the FAL diagrams without using the trivalent graph $T_L$.

\begin{figure}
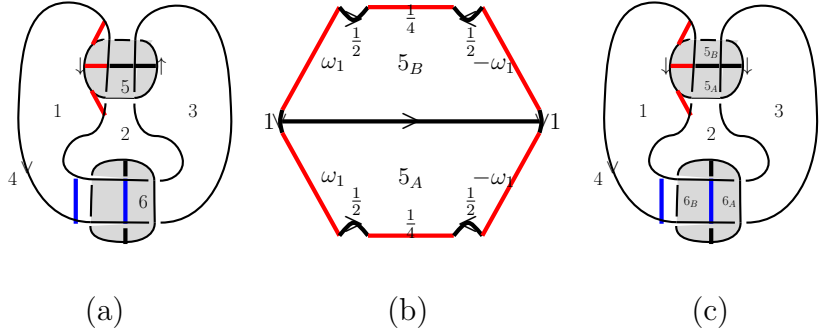

    \centering
    \include{changeorientation}
\hfill \hfill  (a) \hfill \hfill (b) \hfill \hfill (c) \hfill \hfill \hfill 
    \caption[Diagram manipulation]{(a) FAL diagram with the
      orientations. (b) A portion of $T_L$ after the cut-slice-flatten. (c) The
      manipulation on the diagram where the orientation for one side
      of the crossing circle is flipped.}
    \label{changeorientation}
\end{figure}


Now we need to analyze the relationship between edge labels coming from adjacent regions of an edge.  
\begin{lemma}\label{edge}
The edge labels on opposite sides of an edge coming from the longitudinal strands without a half-twist in the FAL diagram are equal.
\end{lemma}

\begin{proof}
Purcell showed how the cusps for FAL are tiled by rectangles \cite{purcell-fal}. Since the white faces on the two polyhedra are glued by identity, these rectangles can be seen in between the white faces on $T_L$, see Figure \ref{TL}.  The parts of the longitude corresponding to the adjacent regions are homotopic across the sliced torus. Hence they are equal.
\end{proof}

In \cite{purcell-fal}, Purcell showed that the complements of a FAL with and without a half-twist have the same polyhedral decomposition, but with different gluing on shaded faces.  Thus the faces of the regions of a FAL diagram with half-twist do not represent white faces of $P_L$.  So for FAL with half-twists we will take $T_L$ the same as the one for the corresponding FAL without half-twists.  We look at faces of $T_L$, we can find the intercusp geodesic parameters and translational parameters from analyzing the shear that is caused by the half-twist gluing. This will be done below when we look at the cusps in \S4. 

\subsection{Examples}
\subsubsection{Borromean Ring FAL} See Figure \ref{br}.
\begin{figure}
\centering
\definecolor{linkcolor0}{rgb}{0.85, 0.15, 0.15}
\definecolor{linkcolor1}{rgb}{0.15, 0.15, 0.85}
\definecolor{linkcolor2}{rgb}{0.15, 0.85, 0.15}
\begin{tikzpicture}[every node/.style={scale=.5},scale=.45][ line width=5.1, line cap=round, line join=round,every node/.style={transform shape}]
\draw (3,1.1) -- (3,2.2);
\draw (7.2,1)--(7.2, 2.4);
\draw (4.7,6.2) -- (5.7,6.2);

    \draw (4.40, 9.86) .. controls (3.77, 9.82) and (3.60, 9.05) .. 
          (3.58, 8.34) .. controls (3.56, 7.59) and (3.98, 6.88) .. (4.66, 6.89);
    \draw (4.66, 6.89) .. controls (5.01, 6.89) and (5.36, 6.89) .. (5.71, 6.90);
    \draw (5.71, 6.90) .. controls (6.45, 6.90) and (6.88, 7.68) .. 
          (6.89, 8.48) .. controls (6.90, 9.25) and (6.63, 10.04) .. (5.99, 9.99);
    \draw (5.53, 9.95) .. controls (5.29, 9.93) and (5.06, 9.91) .. (4.83, 9.90);
    \draw (3.47, 2.40) .. controls (3.44, 3.05) and (4.29, 3.18) .. 
          (5.06, 3.20) .. controls (5.86, 3.23) and (6.70, 3.00) .. (6.71, 2.31);
    \draw (6.71, 2.31) .. controls (6.73, 1.87) and (6.74, 1.43) .. (6.75, 0.99);
    \draw (6.75, 0.99) .. controls (6.76, 0.34) and (5.93, 0.22) .. 
          (5.17, 0.22) .. controls (4.43, 0.22) and (3.58, 0.22) .. (3.55, 0.87);
    \draw (3.53, 1.27) .. controls (3.52, 1.50) and (3.51, 1.74) .. (3.49, 1.97);
    \draw (4.60, 9.88) .. controls (4.58, 10.84) and (3.76, 11.59) .. 
          (2.84, 11.43) .. controls (1.12, 11.14) and (0.79, 8.37) .. 
          (0.52, 6.11) .. controls (0.23, 3.60) and (1.33, 1.06) .. (3.54, 1.03);
    \draw (3.54, 1.03) .. controls (4.50, 1.02) and (5.46, 1.00) .. (6.42, 0.99);
    \draw (7.07, 0.98) .. controls (9.15, 0.95) and (9.72, 3.57) .. 
          (9.64, 6.02) .. controls (9.57, 8.33) and (9.48, 11.11) .. 
          (7.63, 11.43) .. controls (6.68, 11.60) and (5.77, 10.92) .. (5.76, 9.97);
    \draw (5.76, 9.97) .. controls (5.75, 9.05) and (5.73, 8.14) .. (5.72, 7.22);
    \draw (5.71, 6.57) .. controls (5.69, 5.74) and (5.88, 4.86) .. 
          (6.57, 4.90) .. controls (7.17, 4.94) and (7.54, 4.31) .. 
          (7.63, 3.65) .. controls (7.72, 3.01) and (7.49, 2.34) .. (6.93, 2.32);
    \draw (6.39, 2.30) .. controls (5.42, 2.27) and (4.45, 2.24) .. (3.48, 2.21);
    \draw (3.48, 2.21) .. controls (2.85, 2.19) and (2.47, 2.87) .. 
          (2.55, 3.58) .. controls (2.63, 4.29) and (3.10, 4.91) .. 
          (3.71, 4.78) .. controls (4.44, 4.62) and (4.68, 5.60) .. (4.66, 6.56);
    \draw (4.65, 7.21) .. controls (4.63, 8.10) and (4.61, 8.99) .. (4.60, 9.88);
  \node at (5,.2) {\bf\textgreater};

\node at (6.8,8) {\bf$\wedge$};

\node at (3.6,8) {\bf$\wedge$};

\node at (5,1) {\bf\textgreater};

\node at (5,3.2) {\bf\textgreater};

\node at (2,6) {$\aleph$};
\node at (5,4.2) {$\gimel$};
\node at (8,6) {$\beth$};

\node at (5,3.5) {\Large$1$};

\node at (3.1,.8) {$\omega_1$};

\node at (7.2,.7) {$\omega_1$};
\node at (3.15,3) {$-\omega_1$};
\node at (7.,3.) {$-\omega_1$};

\node at (4.2,6.7) {$\omega_2$};

\node at (4.2,10.2) {$\omega_2$};
\node at (6.3,10.3) {$-\omega_2$};
\node at (6.3,6.65) {$-\omega_2$};

\node at (1,5.7) {\large$u_1$};

\node at (3.3,8.7) {\large$1$};
\node at (7.3,8.7) {\large$1$};
\node at (2,3.6)  {\large$u_2$};
\node at (3,3.6) {\large$u_2$};
\node at (7,3.6) {\large$u_3$};
\node at (8,3.6) {\large$u_3$};
\node at (9.2,5.6) {\large$u_4$};
\node at (2.6,1.8) {\large$\frac{1}{4}$};

\node at (7.6,1.8) {\large$\frac{1}{4}$};
\node at (5.2,5.6) {\large$\frac{1}{4}$};

\end{tikzpicture}
\caption[Borromean FAL]{Borromean ring FAL with crossing and edge parameters.}
\label{br}
\end{figure}
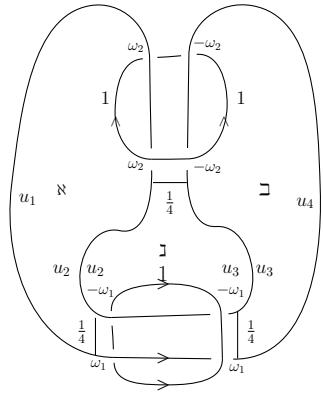
Recall, for a $3$-sided region all $\xi_i = 1$.  Region $\aleph$:
$$\xi_1=\frac{-\omega_2}{u_1}=1, \quad \xi_2=\frac{-\omega_2}{u_2}=1,
\quad \xi_3=\frac{-\frac{1}{4}}{u_1u_2}=1,$$ $\Longrightarrow$
\quad \quad $u_1 = u_2 = -\omega_2$ \quad and \quad  $u_1^2=-\frac{1}{4}$ \quad \quad 
$\Longrightarrow$ \quad \quad 
$u_1= \pm\frac{i}{2}.$\\
Region $\beth$:
$$\xi_1=\frac{-\omega_2}{u_3}=1, \quad \xi_2=\frac{-\omega_2}{u_4}=1,
\quad \xi_3=\frac{-\frac{1}{4}}{u_3u_4}=1,$$ $\Longrightarrow$
$$u_3 = u_4 = -\omega_2 = \pm\frac{i}{2}.$$
Region $\gimel$:
$$\xi_1=\frac{-\omega_1}{u_3}=, \quad \xi_2=\frac{-\omega_1}{u_2}=1,
\quad \xi_3=\frac{-\frac{1}{4}}{u_2u_3}=1,$$ $\Longrightarrow$
$$u_2 = u_3 = -\omega_1 = \pm\frac{i}{2}.$$


\subsubsection{$FAL_{4_1}$}
We denote the FAL shown in Figure \ref{4cc} as $FAL_{4_1}$.
\begin{figure}
\centering
\begin{tikzpicture}[every node/.style={scale=0.5},scale=.51][line width=4.1, line cap=round, line join=round]
 \draw[red] (4,4.15) -- (4,4.8);
  \draw[red] (6,4.15) -- (6,4.8);
   \draw[red] (4,1) -- (4,1.6);
    \draw[red] (6,1) -- (6,1.6);
  
   \draw[red] (3.8,7) -- (3.8,7.65);
   \draw[red] (5.95,7) -- (5.95,7.6);
  
    \draw (3.37, 6.96) .. controls (3.40, 6.61) and (3.01, 6.43) .. 
          (2.62, 6.39) .. controls (2.22, 6.35) and (1.82, 6.51) .. (1.82, 6.85);
    \draw (1.82, 7.15) .. controls (1.82, 7.32) and (1.82, 7.49) .. (1.82, 7.67);
    \draw (1.82, 7.95) .. controls (1.82, 8.28) and (2.17, 8.46) .. 
          (2.52, 8.52) .. controls (2.93, 8.58) and (3.26, 8.21) .. (3.30, 7.77);
    \draw (3.30, 7.77) .. controls (3.33, 7.50) and (3.35, 7.23) .. (3.37, 6.96);
    \draw (7.75, 7.59) .. controls (7.75, 7.95) and (7.37, 8.17) .. 
          (6.97, 8.18) .. controls (6.59, 8.20) and (6.20, 8.09) .. (6.19, 7.76);
    \draw (6.18, 7.53) .. controls (6.18, 7.40) and (6.17, 7.28) .. (6.17, 7.15);
    \draw (6.16, 6.89) .. controls (6.15, 6.53) and (6.55, 6.34) .. 
          (6.96, 6.34) .. controls (7.37, 6.34) and (7.77, 6.56) .. (7.76, 6.93);
    \draw (7.76, 6.93) .. controls (7.76, 7.15) and (7.76, 7.37) .. (7.75, 7.59);
    \draw (5.73, 4.71) .. controls (5.74, 5.13) and (5.57, 5.55) .. 
          (5.19, 5.57) .. controls (4.91, 5.58) and (4.60, 5.54) .. 
          (4.46, 5.29) .. controls (4.38, 5.13) and (4.30, 4.95) .. (4.29, 4.77);
    \draw (4.29, 4.59) .. controls (4.29, 4.48) and (4.29, 4.36) .. (4.29, 4.24);
    \draw (4.28, 3.99) .. controls (4.27, 3.66) and (4.61, 3.46) .. 
          (4.97, 3.44) .. controls (5.35, 3.43) and (5.69, 3.70) .. (5.70, 4.08);
    \draw (5.70, 4.08) .. controls (5.71, 4.29) and (5.72, 4.50) .. (5.73, 4.71);
    \draw (5.66, 1.56) .. controls (5.68, 1.93) and (5.37, 2.24) .. 
          (4.99, 2.26) .. controls (4.61, 2.27) and (4.27, 2.04) .. (4.26, 1.68);
    \draw (4.25, 1.40) .. controls (4.25, 1.27) and (4.24, 1.14) .. (4.24, 1.01);
    \draw (4.23, 0.73) .. controls (4.22, 0.40) and (4.54, 0.16) .. 
          (4.89, 0.18) .. controls (5.28, 0.20) and (5.59, 0.50) .. (5.61, 0.89);
    \draw (5.61, 0.89) .. controls (5.63, 1.11) and (5.64, 1.33) .. (5.66, 1.56);
    \draw (8.00, 7.57) .. controls (8.90, 7.54) and (9.79, 7.82) .. 
          (9.65, 8.50) .. controls (9.42, 9.51) and (6.87, 9.61) .. 
          (4.97, 9.69) .. controls (3.06, 9.76) and (0.52, 9.87) .. 
          (0.32, 8.89) .. controls (0.18, 8.23) and (0.99, 7.87) .. (1.82, 7.84);
    \draw (1.82, 7.84) .. controls (2.23, 7.82) and (2.64, 7.80) .. (3.05, 7.78);
    \draw (3.55, 7.76) .. controls (4.43, 7.73) and (5.31, 7.69) .. (6.18, 7.65);
    \draw (6.18, 7.65) .. controls (6.62, 7.63) and (7.06, 7.61) .. (7.50, 7.60);
    \draw (4.24, 0.88) .. controls (2.43, 0.86) and (0.87, 2.16) .. 
          (0.74, 3.93) .. controls (0.64, 5.37) and (0.61, 7.00) .. (1.82, 6.98);
    \draw (1.82, 6.98) .. controls (2.25, 6.97) and (2.69, 6.97) .. (3.13, 6.96);
    \draw (3.54, 6.95) .. controls (3.99, 6.95) and (4.10, 6.37) .. 
          (3.98, 5.83) .. controls (3.87, 5.30) and (3.86, 4.71) .. (4.29, 4.71);
    \draw (4.29, 4.71) .. controls (4.69, 4.71) and (5.08, 4.71) .. (5.48, 4.71);
    \draw (5.82, 4.71) .. controls (6.17, 4.71) and (6.05, 5.35) .. 
          (5.95, 5.88) .. controls (5.85, 6.43) and (5.73, 7.05) .. (6.17, 7.03);
    \draw (6.17, 7.03) .. controls (6.62, 7.00) and (7.07, 6.97) .. (7.51, 6.94);
    \draw (8.01, 6.91) .. controls (9.20, 6.84) and (9.12, 5.26) .. 
          (9.00, 3.87) .. controls (8.85, 2.21) and (7.51, 0.90) .. (5.86, 0.89);
    \draw (5.36, 0.89) .. controls (4.99, 0.88) and (4.61, 0.88) .. (4.24, 0.88);
    \draw (4.25, 1.53) .. controls (3.73, 1.52) and (3.44, 2.09) .. 
          (3.47, 2.67) .. controls (3.49, 3.13) and (3.53, 3.83) .. 
          (3.67, 3.98) .. controls (3.80, 4.13) and (4.06, 4.13) .. (4.28, 4.12);
    \draw (4.28, 4.12) .. controls (4.67, 4.11) and (5.06, 4.10) .. (5.45, 4.08);
    \draw (5.88, 4.07) .. controls (6.41, 4.06) and (6.57, 3.42) .. 
          (6.56, 2.81) .. controls (6.55, 2.21) and (6.36, 1.57) .. (5.83, 1.56);
    \draw (5.41, 1.55) .. controls (5.02, 1.54) and (4.64, 1.54) .. (4.25, 1.53);
  
   \node at (5,.2) {\bf\large\textgreater};
\node at (5,7.7) {\bf\large\textgreater};
\node at (5,.9) {\bf\large\textgreater};
\node at (5,2.3) {\bf\large\textgreater};
\node at (5,3.4) {\bf\large\textgreater};
\node at (5,5.6) {\bf\large\textgreater};
\node at (2.5,6.4) {\bf\large\textless};
\node at (2.5,8.55) {\bf\large\textless};
\node at (6.5,2.4) {\bf\large$\wedge$};
\node at (7,6.4) {\bf\large\textless};
\node at (7,8.2) {\bf\large\textless};

\node at (5.95,5.8) {\bf\large$\vee$};

\node at (5,8.7) {\Large$\aleph$};
\node at (8,5) {\Large$\gimel$};
\node at (2,5) {\Large$\beth$};
\node at (5,6.5) {\Large$\daleth$};

\node at (5,3.2) {$1$};
\node at (5,2.5) {$1$};
\node at (5,6) {$1$};
\node at (2.5,6) {$1$};
\node at (7,6) {$1$};
\node at (1.5,6.6) {\large$\omega_1$};
\node at (3.5,6.6) {\large$\omega_1$};
\node at (1.5,8.25) {\large$-\omega_1$};
\node at (3.5,8.25) {\large$-\omega_1$};

\node at (6.2,6.6) {\large$\omega_2$};
\node at (8.1,6.6) {\large$\omega_2$};
\node at (6,8) {\large$-\omega_2$};
\node at (8,8) {\large$-\omega_2$};
\node at (4.1,3.8) {$\omega_3$};
\node at (4.1,5.3) {$\omega_3$};

\node at (5.85,5.3) {$\omega_3$};
\node at (6,3.8) {$\omega_3$};
\node at (4,.6) {$\omega_4$};
\node at (4.,1.9) {$\omega_4$};
\node at (5.9,1.9) {$\omega_4$};
\node at (5.9,.6) {$\omega_4$};

\node at (5,9.3) {\large$u_2$};
\node at (5,8) {\large$u_1$};
\node at (5,7.4) {\large$u_1$};

\node at (1.2,4) {\large$u_3$};

\node at (3.1,3) {\large$u_5$};
\node at (3.85,3) {\large$u_5$};

\node at (3.5,5.5) {\large$u_4$};
\node at (4.3,5.7) {\large$u_4$};

\node at (2.5,8.8) {$1$};
\node at (7,8.6) {$1$};

\node at (4.1,7.3) {$\frac{1}{4}$};
\node at (5.75,7.3) {$\frac{1}{4}$};

\node at (3.7,4.5) {$-\frac{1}{4}$};

\node at (6.35,4.5) {$-\frac{1}{4}$};

 \node at (6.25,1.3) {$-\frac{1}{4}$};
 \node at (3.6,1.3) {$-\frac{1}{4}$};
\node at (8.5,4) {\large$u_6$};
  \node at (6.4,5.4) {$u_7$};
   \node at (5.65,5.8) {$u_7$};
 \node at (7,3) {\large$u_8$};  
    \node at (6.2,3) {\large$u_8$};
   \node at (5,2.8) {\large$E$};
\end{tikzpicture}
\caption{$FAL_{4_1}$}
\label{4cc}
\end{figure}
\\
Region $\aleph$: 

This is a four-sided region with shape parameters:
$$\xi_1=\frac{-\omega_1}{u_2}, \quad \xi_2=\frac{-\omega_2}{u_2},
\quad \xi_3=\frac{-\omega_2}{u_1}, \quad \xi_4=\frac{-\omega_1}{u_1}.$$  The sum of consecutive shape parameters are $1$. 
$$\frac{-\omega_1}{u_2} - \frac{\omega_2}{u_2} = 1, \quad
\frac{-\omega_2}{u_2} - \frac{\omega_2}{u_1} = 1, \quad \frac{-\omega_2}{u_1} - \frac{\omega_1}{u_1} = 1, \quad \frac{-\omega_1}{u_1} - \frac{\omega_1}{u_2} = 1$$ solving gives us the relations
$$u_1 = u_2 = -2\omega_1 = -2\omega_2, \quad \textrm{and} \quad \omega_1 = \omega_2.$$

Region $\beth$: 

This is a four-sided region with shape parameters:
$$\xi_1=\frac{-\frac{1}{4}}{u_3u_5}, \quad \xi_2=\frac{-\omega_1}{u_3},
\quad \xi_3=\frac{-\omega_1}{u_4}, \quad \xi_4=\frac{-\frac{1}{4}}{u_4u_5}.$$
Thus we have equations:
$$\frac{-\frac{1}{4}}{u_3u_5} - \frac{\omega_1}{u_3} = 1 , \quad \frac{-\omega_1}{u_3} - \frac{\omega_1}{u_4} = 1, \quad \frac{-\omega_1}{u_4} - \frac{\frac{1}{4}}{u_4u_5} = 1, \quad \frac{-\frac{1}{4}}{u_4u_5} - \frac{\frac{1}{4}}{u_3u_5} = 1$$

solving gives us the relations $u_4 = u_3 =-2\omega_1$, and $u_5 = \frac{1}{4\omega_1}$

Region $\gimel$: 
$$\xi_1=\frac{-\frac{1}{4}}{u_8u_6}, \quad \xi_2=\frac{-\frac{1}{4}}{u_8u_7},
\quad \xi_3=\frac{-\omega_2}{u_7}, \quad \xi_4=\frac{-\omega_2}{u_6}$$  This is a four-sided region with equations:
$$\frac{-\frac{1}{4}}{u_8u_6} - \frac{\frac{1}{4}}{u_8u_7} = 1 , \quad \frac{-\frac{1}{4}}{u_8u_7} - \frac{\omega_2}{u_7} = 1, \quad \frac{-\omega_2}{u_7} - \frac{\omega_2}{u_6} = 1,  \quad \frac{-\omega_2}{u_6} - \frac{\frac{1}{4}}{u_8u_6} = 1$$ 
solving gives us the relations
$$u_6 = u_7 = -2\omega_2, \quad \textrm{and} \quad u_8 =\frac{1}{4\omega_2}.$$

Region $\daleth$: 
$$\xi_1=\frac{-\frac{1}{4}}{u_{4}u_{1}}, \quad \xi_2=\frac{-\frac{1}{4}}{u_{1}u_{7}},
\quad \xi_3=\frac{\omega_3}{u_{7}}, \quad \xi_4=\frac{\omega_3}{u_{4}}$$ This is a four-sided region with equations:
$$\frac{-\frac{1}{4}}{u_{4}u_{1}} - \frac{\frac{1}{4}}{u_{1}u_{7}} = 1 , \quad \frac{-\frac{1}{4}}{u_{1}u_{7}} + \frac{\omega_3}{u_{7}} = 1,  \quad \frac{\omega_3}{u_{7}} + \frac{\omega_3}{u_{4}} = 1,  \quad \frac{\omega_3}{u_{4}} - \frac{\frac{1}{4}}{u_{4}u_{1}} = 1$$ solving gives us the relations
$$u_{4} = u_{7} = 2\omega_3, \quad \textrm{and} \quad u_{1} =-\frac{1}{4\omega_3}.$$

Region $E$:
This is a four-sided region with shape parameters:
$$\xi_1=\frac{-\omega_4}{u_{5}}, \quad \xi_2=\frac{\omega_3}{u_{5}},
\quad \xi_3=\frac{\omega_3}{u_{8}}, \quad \xi_4=\frac{-\omega_4}{u_{8}}.$$  Thus the equations are:
$$\frac{-\omega_4}{u_{5}} + \frac{\omega_3}{u_{5}} = 1 , \quad \frac{\omega_3}{u_{5}} + \frac{\omega_3}{u_{8}} = 1,  \quad \frac{\omega_3}{u_{8}} - \frac{\omega_4}{u_{8}} = 1,  \quad \frac{-\omega_4}{u_{8}} - \frac{\omega_4}{u_{5}} = 1$$ solving gives us the relations
$$u_{5} = u_{8} = 2\omega_3 = -2\omega_4, \quad \textrm{and} \quad \omega_3 = -\omega_4.$$

Using the fact that opposite sides of an edge are equal, we get $$\omega_1 =\pm \frac{ \sqrt{2}}{4}i.$$

\section{FAL Cusp Shapes}
\label{sec:cusp}
In \cite{purcell-fal} Purcell described a method to compute the cusp shapes for each cusp of a FAL using the polyhedral decomposition, by lifting the ideal vertex corresponding to a crossing circle to $\infty$, constructing a circle packing and computing the radii of each circle. 

In Theorem \ref{mainthm} below we prove that the extension of the T-T method to FALs in \S3 enable us to compute cusp shapes in a simpler way, by solving algebraic equations derived directly from the FAL diagram, without constructing the polyhedral decomposition, and circle packings.

\begin{theorem}\label{mainthm}
 Let $L$ be a FAL diagram and let $\omega$ be the parameter of the crossing geodesic for a crossing
  circle $C$ of $L$. 
  \begin{enumerate}

      \item  If $L$ has no half-twist at $C$, then the cusp shape of $C$ is  $4\omega$.

  \item 
   If $L$ has a RH half-twist at $C$, then the cusp shape of $C$ is 
$ \displaystyle{\frac{4\omega}{1+ 2\omega}}$. 
  
  \item 
   If $L$ has a LH half-twist at $C$, then the cusp shape of $C$ is 
$ \displaystyle{\frac{4\omega}{1- 2\omega}}$.  
  
  \end{enumerate}

\end{theorem}
\begin{remark}
The FAL complements with RH half-twist and LH half-twist are isometric as a RH half-twist can be changed to a
LH half-twist in presence of a crossing circle by adding a full twist, which is a homeomorphism. However the canonical longitude for
the crossing circle is different in each case, thus we get a different cusp shape.
\end{remark}

\begin{proof} We will determine the longitude and
  meridian curves in the fundamental domain for the given crossing
  circle.  Let $L$ be a FAL and $C$ be a crossing circle.  Let $S^3-L = P_1 \cup P_2$, where $P_1$ and $P_2$ are isometric to the right angled polyhedron $P_L$ described in Proposition \ref{PLprop}. The twice punctured disc bounded by $C$ becomes a bowtie on $P_L$ and the ideal point corresponding to $C$ is the center of the bowtie. Let $p$ denote the ideal point corresponding to $C$, since the faces of $P_L$ are geodesic, they lie on hyperbolic planes, which are determined by circles or lines on $\mathbb{C} \cup \infty$.  The four faces incident to $p$ are two white faces and two shaded faces. Correspondingly we have two tangent circles in the white circle packing, and two tangent circles in the dual shaded circle packing, see Figure \ref{cp}(b). 
\begin{figure}

\centering
\definecolor{linkcolor0}{rgb}{0.85, 0.15, 0.15}
\definecolor{linkcolor1}{rgb}{0.15, 0.15, 0.85}
\definecolor{linkcolor2}{rgb}{0.15, 0.85, 0.15}
\begin{tikzpicture}[scale=.2,line width=2, line cap=round, line join=round]
  \begin{scope}[color=linkcolor0]
    \draw (2.62, 6.77) .. controls (0.75, 6.78) and (0.35, 9.16) .. 
          (0.40, 11.36) .. controls (0.46, 13.45) and (0.52, 15.93) .. (2.11, 15.92);
    \draw (3.03, 15.91) .. controls (4.18, 15.91) and (5.34, 15.90) .. (6.49, 15.89);
    \draw (7.48, 15.88) .. controls (9.35, 15.87) and (9.45, 13.43) .. 
          (9.52, 11.30) .. controls (9.59, 9.08) and (9.04, 6.73) .. (7.15, 6.74);
    \draw (7.15, 6.74) .. controls (5.64, 6.75) and (4.13, 6.76) .. (2.62, 6.77);
  \end{scope}
  \begin{scope}[color=linkcolor1]
    \draw (2.47, 18.92) .. controls (2.49, 17.92) and (2.50, 16.92) .. (2.51, 15.92);
    \draw (2.51, 15.92) .. controls (2.55, 13.04) and (2.58, 10.16) .. (2.62, 7.29);
    \draw (2.63, 6.25) .. controls (2.65, 4.28) and (2.68, 2.31) .. (2.70, 0.35);
  \end{scope}
  \begin{scope}[color=linkcolor2]
    \draw (6.96, 19.21) .. controls (6.98, 18.10) and (6.99, 16.99) .. (7.01, 15.89);
    \draw (7.01, 15.89) .. controls (7.05, 13.01) and (7.10, 10.14) .. (7.14, 7.26);
    \draw (7.16, 6.23) .. controls (7.19, 4.28) and (7.22, 2.34) .. (7.25, 0.40);
  \end{scope}
  \node at (5,6) {\large$p$};
  
  \node at (3.4,3) {\large$q$};
   \node at (8,3) {\large$r$};
  
  \node at (8,18) {\large$s$};
  \node at (3.4,18) {\large$t$};
\end{tikzpicture}
\hspace{1cm}
\begin{tikzpicture}[scale=2]
\draw[black, thick](5,5.5)circle[radius=0.5];
\node at (4.8,5.5) {$2$};
\node at (6.2,5.5) {$1$};
\node at (5.5,6.2) {$3$};
\node at (5.5,4.8) {$4$};

\draw[black, thick](6,5.5) circle [radius=0.5];
\draw[black, thick, dashed] (5.5,6) circle [radius=0.5];
\draw[black, thick, dashed] (5.5,5) circle [radius=0.5];
\draw[fill,red!40](5.5,5.5) circle[radius=0.1];
\draw[fill,blue!40](5,6) circle[radius=0.1];
\draw[fill,green!40](6,6) circle[radius=0.1];
\draw[fill,blue!40](5,5) circle[radius=0.1];
\draw[fill,green!40](6,5) circle[radius=0.1];
\node at (5.5,5.5){\bf\large $p$};
\node at (5,6){\bf\large $t$};
\node at (5,5){\bf\large $q$};
\node at (6,6){\bf\large $s$};
\node at (6,5){\bf\large $r$};
\end{tikzpicture}
\hspace{1cm}
\begin{tikzpicture}[scale=2]

\node at (4.8,5.5) {$3$};
\node at (6.7,5.5) {$4$};
\node at (5.75,6.2) {$2$};
\node at (5.75,4.8) {$1$};

\draw [black, thick] (5,6) -- (6.5,6);
\draw [black, thick] (5,5) -- (6.5,5);
\draw [black, thick, dashed] (5,5) -- (5,6);

\draw [black, thick, dashed] (6.5,5) -- (6.5,6);

\draw[fill,blue!40](5,6) circle[radius=0.1];
\draw[fill,blue!40](6.5,6) circle[radius=0.1];
\draw[fill,green!40](5,5) circle[radius=0.1];
\draw[fill,green!40](6.5,5) circle[radius=0.1];
\node at (5,6){\bf\large $t$};
\node at (5,5){\bf\large $s$};
\node at (6.5,6){\bf\large $q$};
\node at (6.5,5){\bf\large $r$};
\end{tikzpicture}
\hfill (a) \hfill \hfill \hfill (b) \hfill \hfill \hfill (c) \hfill \hfill \hfill 

\caption[TPS circle packing at $\infty$]{(a) Thrice punctured sphere without half-twist. (b) Solid circles representing the white faces and dashed circles representing the shaded faces at an ideal point arising from a crossing circle.  (c) The rectangle formed by taking $p$ to $\infty$.}
\label{cp}
\end{figure}
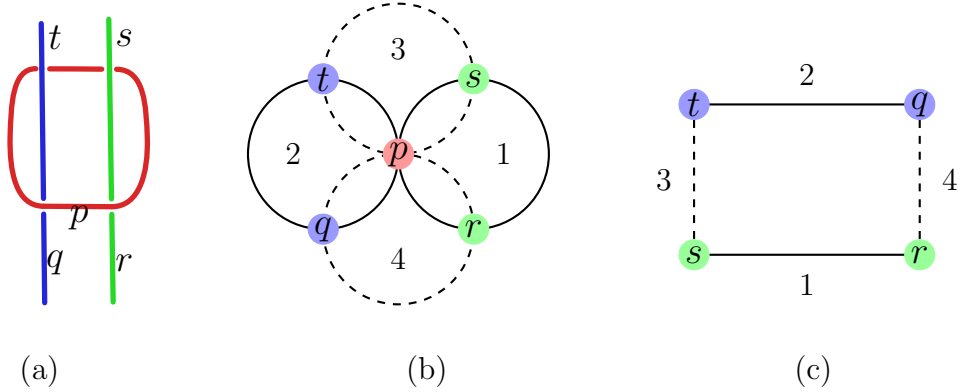

 Superimposing the two circle packings, 
 and taking the point $p$ to $\infty$, the four circles tangent to $p$ become lines that form the rectangle of the cusp on each polyhedron $P_1$ and $P_2$. Let $H_\infty$ denote the horizontal plane corresponding to the horosphere centered at $p$.  
 See Figure \ref{cp} (c).
 All other circles lie inside this rectangle, since the circles are at most tangent to one another and do not overlap.  To find the cusp shape we need to study the fundamental domain of the cusp. 
 
 Case 1: Purcell showed that for FAL without a half-twist present, the fundamental domain for the cusp torus for $C$ is formed by two rectangles attached along a white edge (representing a white face). 
 
 Let's describe the longitude and meridian curves along the crossing circle component of the FAL. Let $s',r',t',q'$ denote the points on $H_\infty$ that are directly above $s,r,t,q$ respectively, translated along respective crossing geodesics $\omega_i$. See Figure \ref{cp}(b). The fundamental
domain is formed by
taking two copies
of the rectangle
$s'r'q't'$ glued along
the edge $s'r'$. The lift of the meridian is $s'r'$, and
the longitude is
double the curve $s't'$.
 
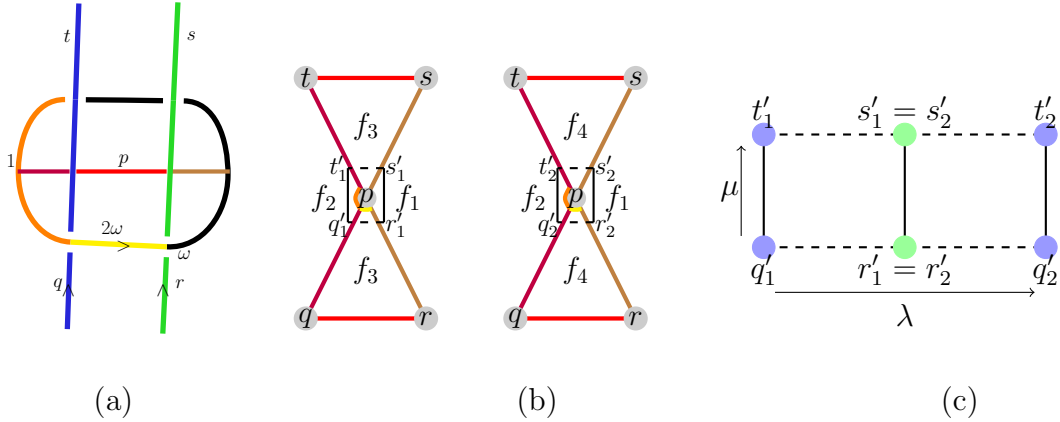
\begin{figure}

\definecolor{linkcolor0}{rgb}{0.85, 0.15, 0.15}
\definecolor{linkcolor1}{rgb}{0.15, 0.15, 0.85}
\definecolor{linkcolor2}{rgb}{0.15, 0.85, 0.15}
\begin{tikzpicture}[thick,scale=0.3, line width = 2, every node/.style={transform shape}][ line cap=round, line join=round]
    \draw (6.91, 4.16) .. controls (8.49, 4.08) and (9.60, 5.62) .. 
          (9.62, 7.32) .. controls (9.63, 8.94) and (9.07, 10.62) .. (7.66, 10.63);
    \draw (6.74, 10.64) .. controls (5.60, 10.66) and (4.46, 10.67) .. (3.32, 10.68);
    \draw[orange] (2.40, 10.69) .. controls (1.03, 10.71) and (0.37, 9.15) .. 
          (0.33, 7.60) .. controls (0.29, 5.99) and (1.15, 4.45) .. (2.64, 4.37);
    \draw[yellow] (2.64, 4.37) .. controls (4.06, 4.30) and (5.48, 4.23) .. (6.91, 4.16);
 \draw[ultra thick, purple] (0.3,7.5) -- (2.6,7.5);
 \draw[ultra thick, red] (2.9,7.5) -- (6.9,7.5) ;
 \draw[ultra thick, brown] (7.1,7.5) -- (9.7,7.5) ;
  \begin{scope}[color=linkcolor1]
    \draw (3.01, 14.95) .. controls (2.96, 13.53) and (2.91, 12.11) .. (2.86, 10.69);
    \draw (2.86, 10.69) .. controls (2.79, 8.73) and (2.72, 6.78) .. (2.65, 4.83);
    \draw (2.62, 3.91) .. controls (2.58, 2.78) and (2.54, 1.64) .. (2.50, 0.51);
  \end{scope}
  \begin{scope}[color=linkcolor2]
    \draw (7.40, 15.00) .. controls (7.33, 13.55) and (7.27, 12.09) .. (7.20, 10.64);
    \draw (7.20, 10.64) .. controls (7.11, 8.63) and (7.02, 6.63) .. (6.93, 4.62);
    \draw (6.89, 3.70) .. controls (6.84, 2.57) and (6.79, 1.44) .. (6.73, 0.31);

  \end{scope}
   \node at (5,4.2) {\huge\textgreater};
    
   \node at (2.55,2.2) {\huge$\wedge$}; 
  \node at (6.75,2.2) {\huge$\wedge$};
      \node at (7.7,3.8) {\huge$\omega$}; 
  \node at (4.5,5) {\huge$2\omega$};
  \node at (0,8) {\huge$1$};
    \node at (5,8) {\huge$p$};
  
  \node at (2.1,2.5) {\huge$q$};
   \node at (7.5,2.5) {\huge$r$};
  
  \node at (8,13.5) {\huge$s$};
  \node at (2.5,13.5) {\huge$t$};
\end{tikzpicture}
\hspace{.501cm}
\begin{tikzpicture}[thick,scale=.8, every node/.style={transform shape}]
 \draw [purple, ultra thick] (0,0) -- (1,2);
 \draw[red, ultra thick] (0,0) -- (2,0);
 \draw[brown, ultra thick] (2,0) -- (1,2);

 \draw[purple, ultra thick] (1,2) -- (0,4);
 \draw[brown, ultra thick] (1,2) -- (2,4);
\draw[red, ultra thick] (0,4) -- (2,4);

  \fill [black!20] (0,0) circle (.2cm);
 \fill [black!20] (2,0) circle (.2cm);
 \fill [black!20] (0,4) circle (.2cm);
 \fill [black!20] (2,4) circle (.2cm);
 \fill [black!20] (1,2) circle (.175cm);
 \draw[orange, ultra thick](.9,2.175) ..controls (.8,2) .. (.9,1.825);
\draw[] (.7,2.5) -- (.7,1.6); 
 \draw[] (1.3,2.5) -- (1.3,1.6); 
 \draw[dashed](1.3,2.5)--(.7,2.5);
 \draw[dashed](1.3,1.6)--(.7,1.6);
 
 \draw[yellow, ultra thick](.9,1.825) .. controls (1,1.8) .. (1.1,1.825);
 \node at 
 (1,2){\bf\large $p$};
\node at (0,4){\bf\large $t$};
\node at (0,0){\bf\large $q$};
\node at (2,4){\bf\large $s$};
\node at (2,0){\bf\large $r$};
\node at (1,3.2){\bf\large $f_3$};
\node at (1.7,2){\bf\large $f_1$};
\node at (0.3,2){\bf\large $f_2$};
\node at (1,.8){\bf\large $f_3$};
\node at (.55,1.56) {$q_1'$};
\node at (1.5,1.56) {$r_1'$};
\node at (.55,2.5) {$t_1'$};
\node at (1.5,2.5) {$s_1'$};
\end{tikzpicture}
\hspace{.501cm}
\begin{tikzpicture}[thick,scale=.8, every node/.style={transform shape}]
 \draw [purple, ultra thick] (0,0) -- (1,2);
 \draw[red, ultra thick] (0,0) -- (2,0);
 \draw[brown, ultra thick] (2,0) -- (1,2);

 \draw[purple, ultra thick] (1,2) -- (0,4);
 \draw[brown, ultra thick] (1,2) -- (2,4);
\draw[red, ultra thick] (0,4) -- (2,4);

  \fill [black!20] (0,0) circle (.2cm);
 \fill [black!20] (2,0) circle (.2cm);
 \fill [black!20] (0,4) circle (.2cm);
 \fill [black!20] (2,4) circle (.2cm);
 \fill [black!20] (1,2) circle (.175cm);
 \draw[orange, ultra thick](.9,2.175) ..controls (.8,2) .. (.9,1.825);
\draw[] (.7,2.5) -- (.7,1.6); 
 \draw[] (1.3,2.5) -- (1.3,1.6); 
 \draw[dashed](1.3,2.5)--(.7,2.5);
 \draw[dashed](1.3,1.6)--(.7,1.6);
 
 \draw[yellow, ultra thick](.9,1.825) .. controls (1,1.8) .. (1.1,1.825);
 \node at 
 (1,2){\bf\large $p$};
\node at (0,4){\bf\large $t$};
\node at (0,0){\bf\large $q$};
\node at (2,4){\bf\large $s$};
\node at (2,0){\bf\large $r$};
\node at (1,3.2){\bf\large $f_4$};
\node at (1.7,2){\bf\large $f_1$};
\node at (0.3,2){\bf\large $f_2$};
\node at (1,.8){\bf\large $f_4$};
\node at (.55,1.56) {$q_2'$};
\node at (1.5,1.56) {$r_2'$};
\node at (.55,2.5) {$t_2'$};
\node at (1.5,2.5) {$s_2'$};
\end{tikzpicture}
\hspace{.51cm}
\begin{tikzpicture}[scale=1.5]

\node at (4.7,5.5) {$\mu$};
\node at (6.25,4.4) {$\lambda$};

\draw [black, thick, dashed] (5,6) -- (7.5,6);
\draw [black, thick, dashed] (5,5) -- (7.5,5);
\draw [black, thick] (5,5) -- (5,6);
\draw [black, thick] (6.25,5) -- (6.25,6);
\draw [black, thick] (7.5,5) -- (7.5,6);

\draw[fill,blue!40](5,6) circle[radius=0.1];
\draw[fill,green!40](6.25,6) circle[radius=0.1];
\draw[fill,blue!40](5,5) circle[radius=0.1];
\draw[fill,green!40](6.25,5) circle[radius=0.1];




\draw[fill,blue!40](7.5,5) circle[radius=0.1];
\draw[fill,blue!40](7.5,6) circle[radius=0.1];

\node at (6.25, 6.2) {$s_1' = s_2'$};
\node at (6.25, 4.8) {$r_1' = r_2'$};
\node at (5,4.8) {$q_1'$};
\node at (5,6.2) {$t_1'$};
\node at (7.5,4.8) {$q_2'$};
\node at (7.5,6.2) {$t_2'$};
\draw[thin, ->](5.1,4.6) -- (7.4,4.6);
\draw[thin, ->](4.8,5.1) -- (4.8,5.9);

\end{tikzpicture}
\centering
 (a) \quad \quad \quad \quad \quad \quad \quad \quad  \quad \quad \quad \quad (b) \quad \quad \quad \quad \quad  \quad \quad \quad \quad \quad \quad \quad (c) 
\caption[Tracking $\mu$ and $\lambda$]{(a) Finding $\lambda$ and $\mu$ for the cusp coming from the crossing circle on the diagram. (b) Cusp view on bowtie. (c) Fundamental domain for $C$.}
\label{patinf}
\end{figure}

From the computations of the thrice punctured sphere in Section 3.1 edge parameter $u_1$ is isometric to geodesic $s't'$ which is the translation parameter along the horoball at hand and is $2\omega$. Since it is double in the actual cusp, the longitude parameter is $4\omega$. The translation parameter $s'r'$ is isometric to the meridian and is 1.  Therefore, the cusp shape $\displaystyle{\frac{\lambda}{\mu} = \frac{4\omega}{1}}$, see Figure \ref{patinf}.

 Case 2: for FAL with half-twists present, i.e. the crossing circle cusps that bound a half-twist will be tiled by rectangles but the fundamental domain will be a parallelogram due to a shear in the universal cover, its longitude curve will run along the shaded face (same as the case without a half-twist present) which is $4\omega$. The meridian curve will run diagonally across since it takes one step along a white face and one step along a shaded face. This is due to a twist in the gluing of the shaded faces. 
 $s_2'$ will be identified with $q_1'$, and $t_2'$ will not be identified with $q_2'$. There are two cases:
 The twist goes with a RH half-twist, see Figure \ref{tpsrhs'}(a),  where
 the meridian goes diagonal increasing from left to right, thus it's $1+2\omega$ so the cusp shape is $ \displaystyle{\frac{4\omega}{1+2\omega}}$. When the twist goes with the LH half-twist, see Figure \ref{tpslhs'}(b), the diagonal is decreasing from left to right, it goes one step down which is $-2\omega$ and one step across which is $1$, thus it's $1-2\omega$ and the cusp shape is $ \displaystyle{\frac{4\omega}{1-2\omega}}$. \end{proof}
 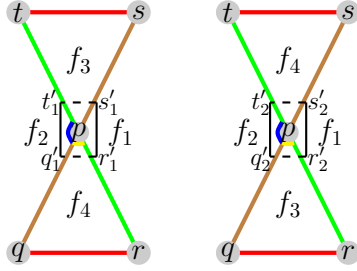
\begin{figure}
     \centering
     \hspace{.501cm}
\begin{tikzpicture}[thick,scale=.8, every node/.style={transform shape}]
 \draw [brown, ultra thick] (0,0) -- (1,2);
 \draw[red, ultra thick] (0,0) -- (2,0);
 \draw[green, ultra thick] (2,0) -- (1,2);

 \draw[green, ultra thick] (1,2) -- (0,4);
 \draw[brown, ultra thick] (1,2) -- (2,4);
\draw[red, ultra thick] (0,4) -- (2,4);

  \fill [black!20] (0,0) circle (.2cm);
 \fill [black!20] (2,0) circle (.2cm);
 \fill [black!20] (0,4) circle (.2cm);
 \fill [black!20] (2,4) circle (.2cm);
 \fill [black!20] (1,2) circle (.175cm);
 \draw[blue, ultra thick](.9,2.175) ..controls (.8,2) .. (.9,1.825);
\draw[] (.7,2.5) -- (.7,1.6); 
 \draw[] (1.3,2.5) -- (1.3,1.6); 
 \draw[dashed](1.3,2.5)--(.7,2.5);
 \draw[dashed](1.3,1.6)--(.7,1.6);
 
 \draw[yellow, ultra thick](.9,1.825) .. controls (1,1.8) .. (1.1,1.825);
 \node at 
 (1,2){\bf\large $p$};
\node at (0,4){\bf\large $t$};
\node at (0,0){\bf\large $q$};
\node at (2,4){\bf\large $s$};
\node at (2,0){\bf\large $r$};
\node at (1,3.2){\bf\large $f_3$};
\node at (1.7,2){\bf\large $f_1$};
\node at (0.3,2){\bf\large $f_2$};
\node at (1,.8){\bf\large $f_4$};
\node at (.55,1.56) {$q_1'$};
\node at (1.5,1.56) {$r_1'$};
\node at (.55,2.5) {$t_1'$};
\node at (1.5,2.5) {$s_1'$};
\end{tikzpicture}
\hspace{.501cm}
\begin{tikzpicture}[thick,scale=.8, every node/.style={transform shape}]
 \draw [brown, ultra thick] (0,0) -- (1,2);
 \draw[red, ultra thick] (0,0) -- (2,0);
 \draw[green, ultra thick] (2,0) -- (1,2);

 \draw[green, ultra thick] (1,2) -- (0,4);
 \draw[brown, ultra thick] (1,2) -- (2,4);
\draw[red, ultra thick] (0,4) -- (2,4);

  \fill [black!20] (0,0) circle (.2cm);
 \fill [black!20] (2,0) circle (.2cm);
 \fill [black!20] (0,4) circle (.2cm);
 \fill [black!20] (2,4) circle (.2cm);
 \fill [black!20] (1,2) circle (.175cm);
 \draw[blue, ultra thick](.9,2.175) ..controls (.8,2) .. (.9,1.825);
\draw[] (.7,2.5) -- (.7,1.6); 
 \draw[] (1.3,2.5) -- (1.3,1.6); 
 \draw[dashed](1.3,2.5)--(.7,2.5);
 \draw[dashed](1.3,1.6)--(.7,1.6);
 
 \draw[yellow, ultra thick](.9,1.825) .. controls (1,1.8) .. (1.1,1.825);
 \node at 
 (1,2){\bf\large $p$};
\node at (0,4){\bf\large $t$};
\node at (0,0){\bf\large $q$};
\node at (2,4){\bf\large $s$};
\node at (2,0){\bf\large $r$};
\node at (1,3.2){\bf\large $f_4$};
\node at (1.7,2){\bf\large $f_1$};
\node at (0.3,2){\bf\large $f_2$};
\node at (1,.8){\bf\large $f_3$};
\node at (.55,1.56) {$q_2'$};
\node at (1.5,1.56) {$r_2'$};
\node at (.55,2.5) {$t_2'$};
\node at (1.5,2.5) {$s_2'$};
\end{tikzpicture}
     \caption{Gluing in Half-twist}
     \label{fig:my_label}
 \end{figure} 
 \begin{figure}
 \centering
 \definecolor{linkcolor0}{rgb}{0.85, 0.15, 0.15}
\definecolor{linkcolor1}{rgb}{0.15, 0.15, 0.85}
\definecolor{linkcolor2}{rgb}{0.15, 0.85, 0.15}
\begin{tikzpicture}[scale=.2][line width=9.5, line cap=round, line join=round]
  \begin{scope}[color=linkcolor0]
    \draw (2.56, 15.83) .. controls (1.00, 15.88) and (0.70, 13.88) .. 
          (0.59, 12.08) .. controls (0.46, 10.21) and (1.09, 8.28) .. (2.69, 8.30);
    \draw (2.69, 8.30) .. controls (4.02, 8.31) and (5.35, 8.32) .. (6.69, 8.34);
    \draw (6.69, 8.34) .. controls (8.39, 8.35) and (9.38, 10.14) .. 
          (9.40, 12.00) .. controls (9.43, 13.75) and (9.08, 15.64) .. (7.57, 15.69);
    \draw (6.35, 15.72) .. controls (5.48, 15.75) and (4.61, 15.77) .. (3.75, 15.80);
  \end{scope}
  \begin{scope}[color=linkcolor1]
    \draw (8.61, 1.42) .. controls (7.46, 2.30) and (6.32, 3.17) .. (5.17, 4.05);
    \draw (5.17, 4.05) .. controls (3.88, 5.03) and (2.60, 6.24) .. (2.67, 7.84);
    \draw (2.72, 9.05) .. controls (2.81, 11.31) and (2.90, 13.56) .. (3.00, 15.82);
    \draw (3.00, 15.82) .. controls (3.06, 17.29) and (3.12, 18.76) .. (3.18, 20.23);
  \end{scope}
  \begin{scope}[color=linkcolor1]
    \draw (7.36, 20.23) .. controls (7.27, 18.72) and (7.19, 17.21) .. (7.10, 15.70);
    \draw (7.10, 15.70) .. controls (6.98, 13.50) and (6.85, 11.29) .. (6.73, 9.09);
    \draw (6.65, 7.74) .. controls (6.58, 6.48) and (6.35, 5.18) .. (5.44, 4.31);
    \draw (4.63, 3.52) .. controls (3.59, 2.52) and (2.55, 1.51) .. (1.50, 0.50);
  \end{scope}
\end{tikzpicture}
\hspace{1cm}
\begin{tikzpicture}[scale=1.5]
\draw[black, thick] (5,5) -- (6.25,6);

\draw[black, thick] (7.5,5) -- (8.75,6);


\node at (5.5,5.6) {$\mu$};
\node at (6.25,4.4) {$\lambda$};

\draw [black, thick, dashed] (5,6) -- (8.75,6);
\draw [black, thick, dashed] (5,5) -- (8.75,5);
\draw [black,] (8.75,5) -- (8.75,6);
\draw [black] (5,5) -- (5,6);
\draw [black] (6.25,5) -- (6.25,6);
\draw [black, ] (7.5,5) -- (7.5,6);
\draw[fill,blue!40](8.75,6) circle[radius=0.1];
\draw[fill,blue!40](8.75,5) circle[radius=0.1];
\draw[fill,blue!40](5,6) circle[radius=0.1];
\draw[fill,blue!40](6.25,6) circle[radius=0.1];
\draw[fill,blue!40](5,5) circle[radius=0.1];
\draw[fill,blue!40](6.25,5) circle[radius=0.1];




\draw[fill,blue!40](7.5,5) circle[radius=0.1];
\draw[fill,blue!40](7.5,6) circle[radius=0.1];

\node at (6.25, 6.2) {$s_1' = s_2'$};
\node at (6.25, 4.8) {$r_1' = r_2'$};
\node at (5,4.8) {$q_1'$};
\node at (7.5,4.8) {$q_2'$};
\node at (7.5,6.2) {$t_2'$};
\draw[thin, ->](5.1,4.6) -- (7.4,4.6);
\draw[thin, ->](5.3,5.1) -- (6.18,5.8);
\node at (8.75, 6.2) {$s_2'$};

\end{tikzpicture}














\hfill (a) \hfill \hfill (b) \hfill \hfill \hfill
 \caption[TPS RH half-twist Fundamental Domain]{(a) TPS with RH half-twist. (b) The corresponding fundamental domain of the cusp due to the RH half-twist.}
 \label{tpsrhs'}
 \end{figure}
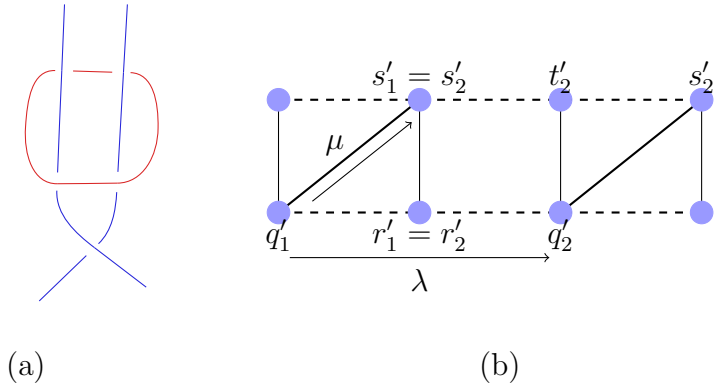

\begin{figure}
\centering
 \definecolor{linkcolor0}{rgb}{0.85, 0.15, 0.15}
\definecolor{linkcolor1}{rgb}{0.15, 0.15, 0.85}
\definecolor{linkcolor2}{rgb}{0.15, 0.85, 0.15}
\begin{tikzpicture}[scale=.2][line width=9.5, line cap=round, line join=round]
  \begin{scope}[color=linkcolor0]
    \draw (2.56, 15.83) .. controls (1.00, 15.88) and (0.70, 13.88) .. 
          (0.59, 12.08) .. controls (0.46, 10.21) and (1.09, 8.28) .. (2.69, 8.30);
    \draw (2.69, 8.30) .. controls (4.02, 8.31) and (5.35, 8.32) .. (6.69, 8.34);
    \draw (6.69, 8.34) .. controls (8.39, 8.35) and (9.38, 10.14) .. 
          (9.40, 12.00) .. controls (9.43, 13.75) and (9.08, 15.64) .. (7.57, 15.69);
    \draw (6.35, 15.72) .. controls (5.48, 15.75) and (4.61, 15.77) .. (3.75, 15.80);
  \end{scope}
  \begin{scope}[color=linkcolor1]
    \draw (8.61, 1.42) .. controls (7.66, 2.14) and (6.72, 2.87) .. (5.77, 3.59);
    \draw (4.66, 4.44) .. controls (3.54, 5.29) and (2.61, 6.46) .. (2.67, 7.84);
    \draw (2.72, 9.05) .. controls (2.81, 11.31) and (2.90, 13.56) .. (3.00, 15.82);
    \draw (3.00, 15.82) .. controls (3.06, 17.29) and (3.12, 18.76) .. (3.18, 20.23);
  \end{scope}
  \begin{scope}[color=linkcolor1]
    \draw (7.36, 20.23) .. controls (7.27, 18.72) and (7.19, 17.21) .. (7.10, 15.70);
    \draw (7.10, 15.70) .. controls (6.98, 13.50) and (6.85, 11.29) .. (6.73, 9.09);
    \draw (6.65, 7.74) .. controls (6.58, 6.37) and (6.16, 5.00) .. (5.17, 4.05);
    \draw (5.17, 4.05) .. controls (3.95, 2.87) and (2.73, 1.68) .. (1.50, 0.50);
  \end{scope}
\end{tikzpicture}
\hspace{1cm}
\begin{tikzpicture}[scale=1.5]
\draw[black, thick] (5,6) -- (6.25,5);

\draw[black,  thick] (7.5,6) -- (8.75,5);


\node at (5.35,5.5) {$\mu$};
\node at (7.5,4.4) {$\lambda$};

\draw [black, thick, dashed] (5,6) -- (8.75,6);
\draw [black, thick, dashed] (5,5) -- (8.75,5);
\draw [black, ] (8.75,5) -- (8.75,6);
\draw [black,] (5,5) -- (5,6);
\draw [black] (6.25,5) -- (6.25,6);
\draw [black] (7.5,5) -- (7.5,6);
\draw[fill,blue!40](8.75,6) circle[radius=0.1];
\draw[fill,blue!40](8.75,5) circle[radius=0.1];
\draw[fill,blue!40](5,6) circle[radius=0.1];
\draw[fill,blue!40](6.25,6) circle[radius=0.1];
\draw[fill,blue!40](5,5) circle[radius=0.1];
\draw[fill,blue!40](6.25,5) circle[radius=0.1];




\draw[fill,blue!40](7.5,5) circle[radius=0.1];
\draw[fill,blue!40](7.5,6) circle[radius=0.1];


\node at (5, 6.2) {$ s_2'$};
\node at (6.25,4.8) {$q_1'$};
\node at (6.25,6.2) {$t_1'$};
\node at (8.75,4.8) {$q_2'$};
\node at (7.5,6.2) {$s_1' = s_2'$};
\node at (7.5, 4.8) {$r_1'=r_2'$};
\draw[thin, ->](6.26,4.6) -- (8.74,4.6);
\draw[thin, ->](6.22,5.2) -- (5.4,5.85);

\end{tikzpicture}
\hfill (a) \hfill \hfill (b) \hfill \hfill \hfill
 \caption[TPS LH half-twist Fundamental Domain]{(a) TPS with LH half-twist. (b) The corresponding fundamental domain of the cusp due to the LH half-twist.}
 \label{tpslhs'}
 \end{figure}
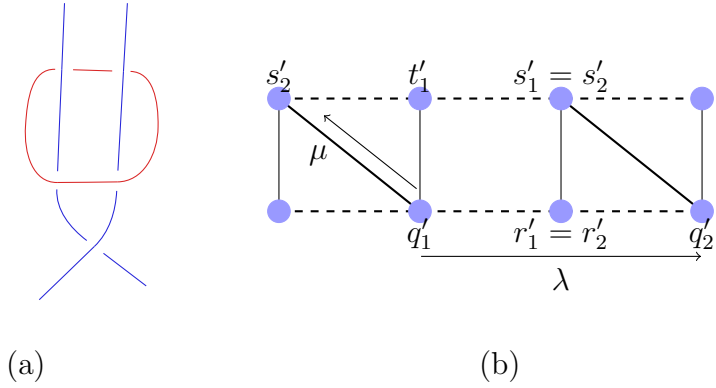

Note that the imaginary part of the cusp shape will always be positive.  

The cusp shapes for the strands in the projection plane can also be determined from the labels in the diagram.
\begin{theorem}\label{mainthm2}

    For a FAL without a half-twist, the cusp shape for a component in the projection plane will be the sum of the edge labels as one goes around the strand.
\end{theorem}

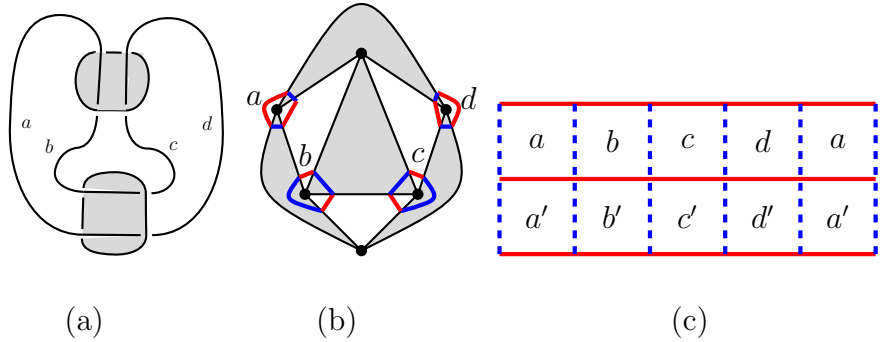
\begin{figure}
\centering
\definecolor{linkcolor0}{rgb}{0.85, 0.15, 0.15}
\definecolor{linkcolor1}{rgb}{0.15, 0.15, 0.85}
\definecolor{linkcolor2}{rgb}{0.15, 0.85, 0.15}
\begin{tikzpicture}[thick,scale=0.3, every node/.style={transform shape}][line width=3.2, line cap=round, line join=round]
   \fill[gray!30]
    (4.07, 9.16) .. controls (3.49, 9.17) and (3.17, 8.54) .. 
          (3.17, 7.89) .. controls (3.17, 7.23) and (3.53, 6.60) .. (4.13, 6.59);
       \fill [gray!30] (3.51, 3) --(3.47, 1)-- (6.3,1)--(6.28,3);
    \fill[gray!30]
    (6.28, 2.92) .. controls (6.27, 3.53) and (5.59, 3.85) .. 
          (4.90, 3.84) .. controls (4.25, 3.84) and (3.54, 3.71) .. (3.52, 3.14);
       \fill [gray!30] (5.40, 6.58) .. controls (5.99, 6.58) and (6.35, 7.20) .. 
          (6.40, 7.85) .. controls (6.46, 8.48) and (6.21, 9.12) .. (5.65, 9.13);

 \fill[gray!30]
 (3.45, 0.90) .. controls (3.43, 0.31) and (4.19, 0.21) .. 
          (4.87, 0.19) .. controls (5.59, 0.16) and (6.31, 0.43) .. (6.30, 1.05);
    
    \draw (4.07, 9.16) .. controls (3.49, 9.17) and (3.17, 8.54) .. 
          (3.17, 7.89) .. controls (3.17, 7.23) and (3.53, 6.60) .. (4.13, 6.59);
    \draw  (4.13, 6.59) .. controls (4.55, 6.59) and (4.98, 6.59) .. (5.40, 6.58);
    \fill [gray!30](4.13, 6.59) .. controls (4.55, 6.59) and (4.98, 6.59) .. (5.40, 6.58)--(5.5,8) -- (6.21, 9.14) .. controls (4.98, 9.15) and (4.75, 9.15) .. (4, 9.15);
    \draw (5.40, 6.58) .. controls (5.99, 6.58) and (6.35, 7.20) .. 
          (6.40, 7.85) .. controls (6.46, 8.48) and (6.21, 9.12) .. (5.65, 9.13);
    \draw (5.21, 9.14) .. controls (4.98, 9.15) and (4.75, 9.15) .. (4.52, 9.15);
   \node at (1,6) {\huge\bf$a$}; 
    \node at (2,5) {\huge\bf$b$}; 
    \node at (7.5,5) {\huge\bf$c$}; 
    \node at (9,6) {\huge\bf$d$};

    \draw (6.28, 2.92) .. controls (6.27, 3.53) and (5.59, 3.85) .. 
          (4.90, 3.84) .. controls (4.25, 3.84) and (3.54, 3.71) .. (3.52, 3.14);

    \draw (3.51, 2.72) .. controls (3.49, 2.25) and (3.48, 1.79) .. (3.47, 1.32);
    \draw  (3.45, 0.90) .. controls (3.43, 0.31) and (4.19, 0.21) .. 
          (4.87, 0.19) .. controls (5.59, 0.16) and (6.31, 0.43) .. (6.30, 1.05);
    \draw (6.30, 1.05) .. controls (6.29, 1.67) and (6.28, 2.29) .. (6.28, 2.92);
    \draw (4.29, 9.16) .. controls (4.36, 10.12) and (3.44, 10.84) .. 
          (2.42, 10.81) .. controls (0.43, 10.75) and (0.34, 8.17) .. 
          (0.26, 5.92) .. controls (0.17, 3.49) and (1.30, 1.08) .. (3.46, 1.07);
    \draw (3.46, 1.07) .. controls (4.32, 1.06) and (5.18, 1.06) .. (6.05, 1.05);
    \draw  (6.56, 1.05) .. controls (8.66, 1.04) and (9.76, 3.38) .. 
          (9.69, 5.75) .. controls (9.62, 7.95) and (9.48, 10.47) .. 
          (7.50, 10.66) .. controls (6.45, 10.77) and (5.45, 10.12) .. (5.44, 9.14);
    \draw (5.44, 9.14) .. controls (5.42, 8.37) and (5.41, 7.61) .. (5.40, 6.84);
    \draw (5.40, 6.33) .. controls (5.38, 5.61) and (5.71, 4.89) .. 
          (6.35, 4.89) .. controls (6.93, 4.89) and (7.41, 4.46) .. 
          (7.53, 3.89) .. controls (7.65, 3.35) and (7.14, 2.90) .. (6.53, 2.91);
    \draw (6.02, 2.92) .. controls (5.18, 2.94) and (4.35, 2.96) .. (3.52, 2.97);
    \draw (3.52, 2.97) .. controls (2.82, 2.99) and (2.15, 3.34) .. 
          (2.26, 3.92) .. controls (2.35, 4.43) and (2.72, 4.86) .. 
          (3.23, 4.92) .. controls (3.84, 4.98) and (4.07, 5.67) .. (4.12, 6.34);
    \draw (4.15, 6.85) .. controls (4.20, 7.62) and (4.25, 8.39) .. (4.29, 9.16);
\end{tikzpicture}
\begin{tikzpicture}[scale=.75]
\fill [gray!30] (1,1) -- (3,1) -- (2,3.5);

\fill [gray!30] (1,1) -- (.5,2.5) .. controls (0,1)..(2,0);

\fill [gray!30] (3,1) -- (3.5,2.5) .. controls (4,1)..(2,0);

\fill [gray!30] (.5,2.5) -- (2,3.5)--(3.5,2.5) .. controls (2,5)..(.5,2.5);

  \draw[  thick] (1,1)--(2,3.5);
  \draw[  thick] (.5,2.5)--(2,3.5);
 \fill (2,0) circle (.1cm);
  \fill (1,1) circle (.1cm);
  \fill (3,1) circle (.1cm);
  \fill (2,3.5) circle (.1cm);
   \fill (.5,2.5) circle (.1cm);
  \fill (3.5,2.5) circle (.1cm);
 \draw[ thick] (1,1)--(.5,2.5);

  \draw[  thick] (1,1)--(2,0);
 \draw[  thick] (3,1)--(2,0);
\draw[  thick] (1,1)--(3,1);
\draw[ thick] (3,1)--(3.5,2.5);
\draw[ thick] (2,3.5)--(3.5,2.5);
\draw[ thick] (3,1)--(2,3.5);
\draw[thick] (2,0) .. controls (0,1) .. (.5,2.5);
\draw[ thick] (.5,2.5) .. controls (2,5) ..(3.5,2.5) ;
\draw[thick] (3.5,2.5) .. controls (4,1) .. (2,0) ;

\node at (.1,2.7) {\bf$a$}; 
\draw[red, ultra thick]
(.4,2.2) ..controls (.2, 2.6) .. (.7,2.8);
\draw[blue, ultra thick] (.7,2.8) -- (.85, 2.65);
\draw[red, ultra thick] (.85,2.65) -- (.6,2.2);
\draw[blue, ultra thick](.6,2.2) -- (.4,2.2);
\node at (1,1.7) {\bf$b$};
\draw[blue, ultra thick]
(.9,1.3) ..controls (.6, .9) .. (1.3,.7);
\draw[red, ultra thick] (1.3,.7) -- (1.5,1);

\draw[blue, ultra thick]  (1.5,1)--(1.15,1.4);
\draw[red, ultra thick] (.9,1.3) -- (1.15,1.4);

\node at (3,1.7) {\bf$c$};
\draw[blue, ultra thick]
(3.1,1.3) ..controls (3.4, .9) .. (2.7,.7);
\draw[red, ultra thick] (2.7,.7) -- (2.5,1);

\draw[blue, ultra thick]  (2.5,1)--(2.85,1.4);
\draw[red, ultra thick] (3.1,1.3) -- (2.85,1.4);

\node at (3.9,2.7) {\bf$d$}; 
\draw[red, ultra thick]
(3.6,2.2) ..controls (3.8, 2.6) .. (3.35,2.8);
\draw[blue, ultra thick] (3.35,2.8) -- (3.3, 2.65);
\draw[red, ultra thick] (3.3,2.65) -- (3.4,2.2);
\draw[blue, ultra thick](3.4,2.2) -- (3.6,2.2);


\end{tikzpicture}
\begin{tikzpicture}
 \draw[red, ultra thick](0,0) -- (5,0);
 \draw[red, ultra thick](0,1) -- (5,1);
 \draw[red, ultra thick](0,2) -- (5,2);
 \draw[blue, ultra thick ,dashed](0,0) -- (0,2);
 \draw[blue, ultra thick, dashed](1,0) -- (1,2);
 \draw[blue, ultra thick, dashed](2,0) -- (2,2);
 \draw[blue, ultra thick, dashed](3,0) -- (3,2);
 \draw[blue, ultra thick, dashed](4,0) -- (4,2);
 \draw[blue, ultra thick, dashed](5,0) -- (5,2);
 \node at (.5,.5) {$a'$};
 
  \node at (.5,1.5) {$a$};
  \node at (1.5,.5) {$b'$};
   \node at (1.5,1.5) {$b$};
  \node at (2.5,1.5) {$c$}; 
   \node at (2.5,.5) {$c'$};
  \node at (3.5,.5) {$d'$}; 
  \node at (4.5,.5) {$a'$};
   
 \node at (3.5,1.5) {$d$}; 
 \node at (4.5,1.5) {$a$}; 
\end{tikzpicture}
\hfill \hfill (a) \hfill \hfill (b) \hfill \hfill \hfill (c) \hfill \hfill \hfill \hfill

\caption[Fundamental domain for longitudinal strand cusp]{(a) $D(L)$ (b) Corresponding $P_L$ with cusps from strands in projection plane drawn in.  (c) The fundamental domain of corresponding cusp. }
\label{cuspprojcomp}
\end{figure}
\begin{proof}

   In \cite{JP2004} Purcell showed that the component in the projection plane is tiled by a sequence of rectangles two for each segment of a component which is then glued along the shaded edge as one goes around the strand. See Figure \ref{cuspprojcomp}. For each component in the projection plane there is a cusp that is tiled by rectangles coming from each portion along the component.  Two identical rectangles are glued along a white edge for the upper and lower polyhedra.  Then along the shaded edge the portion of the component adjacent will be glued along the shaded edge.  The meridian is by convention $1$, where we have a $\frac{1}{2}$ for the meridional segment along each shaded triangle, see Figure \ref{thricepuncturedspherecsf}.  The longitude will consist of $u_j$ for each edge.  The sum for each portion along the strand will be the longitude of the cusp.
\end{proof}
\begin{remark}

For a link component in the projection plane that has half-twists, tracking the longitude is a bit trickier. A FAL with a half-twist and more than one component in the projection plane,
 the component in the projection plane which goes through an odd number of half-twists will still have meridian of length $1$. The longitude will travel parallel to the projection plane except when it will pass a half-twist where it will then travel down/up to the other polyhedron thus it will increase its length by $\pm k\mu/2$ where $k$ takes into account the direction of the half-twist and how many times it passes a half-twist. The cusp shape for the component will be the sum of the edge labels plus half an integer, $\sum u_i \pm k\times\frac{1}{2}$, where the sign will depend on the direction of the half-twists. This is due to a shear, thus the cusp will not necessarily be rectangular.  However, if the component goes through an even number of half-twists, then it will have a rectangular cusp, just the longitude won't be perpendicular to the real meridian--the cusp shape will be as if no half-twists are present.  In addition, if a FAL has only one component in the projection plane then the cusp will be rectangular regardless if there is a half-twist present \cite{purcell2007meridians}.

   Experimentally, FAL with odd number of half-twists such that the presence of the half-twist reduces the number of components seem to be the ones impacted by the shear. Moreover, when there is one half-twist present the cusp shape will be $\sum u_i \pm2.$

\end{remark}
\subsection{Examples}

\subsubsection{Borromean Ring FAL with a half-twist}
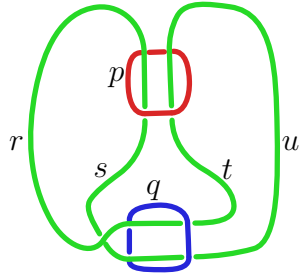
\begin{figure}
\centering
\definecolor{linkcolor0}{rgb}{0.85, 0.15, 0.15}
\definecolor{linkcolor1}{rgb}{0.15, 0.15, 0.85}
\definecolor{linkcolor2}{rgb}{0.15, 0.85, 0.15}
\begin{tikzpicture}[line width=2, line cap=round, line join=round, scale=.35]
 \node at (-.2,5) {$r$}; 
  \node at (3.6,7.5) {$p$};  
   \node at (5,3.2) {$q$}; 
   \node at (3,4) {$s$}; 
   \node at (7.8,4) {$t$}; 
   \node at (10.2,5) {$u$}; 
  \begin{scope}[color=linkcolor0]
    \draw (4.52, 8.45) .. controls (4.05, 8.43) and (4.02, 7.81) .. 
          (4.03, 7.27) .. controls (4.04, 6.71) and (4.19, 6.12) .. (4.67, 6.13);
    \draw (4.67, 6.13) .. controls (5.01, 6.14) and (5.35, 6.15) .. (5.69, 6.16);
    \draw (5.69, 6.16) .. controls (6.19, 6.17) and (6.37, 6.77) .. 
          (6.36, 7.35) .. controls (6.34, 7.91) and (6.22, 8.51) .. (5.74, 8.50);
    \draw (5.41, 8.48) .. controls (5.22, 8.47) and (5.03, 8.47) .. (4.83, 8.46);
  \end{scope}
  \begin{scope}[color=linkcolor1]
    \draw (6.29, 1.97) .. controls (6.28, 2.45) and (5.72, 2.68) .. 
          (5.17, 2.65) .. controls (4.65, 2.62) and (4.08, 2.53) .. (4.08, 2.08);
    \draw (4.08, 1.68) .. controls (4.08, 1.41) and (4.08, 1.14) .. (4.08, 0.87);
    \draw (4.08, 0.52) .. controls (4.08, 0.21) and (4.71, 0.20) .. 
          (5.22, 0.20) .. controls (5.75, 0.19) and (6.36, 0.22) .. (6.34, 0.68);
    \draw (6.34, 0.68) .. controls (6.33, 1.11) and (6.31, 1.54) .. (6.29, 1.97);
  \end{scope}
  \begin{scope}[color=linkcolor2]
    \draw (3.09, 1.32) .. controls (2.66, 0.66) and (1.64, 1.11) .. 
          (1.10, 2.07) .. controls (0.25, 3.57) and (0.18, 5.35) .. 
          (0.52, 7.04) .. controls (0.81, 8.51) and (1.45, 10.02) .. 
          (2.86, 10.18) .. controls (3.83, 10.28) and (4.63, 9.46) .. (4.64, 8.45);
    \draw (4.64, 8.45) .. controls (4.65, 7.77) and (4.66, 7.09) .. (4.67, 6.40);
    \draw (4.67, 5.86) .. controls (4.68, 5.08) and (4.34, 4.33) .. 
          (3.71, 3.86) .. controls (3.29, 3.55) and (2.74, 3.13) .. 
          (2.57, 2.81) .. controls (2.40, 2.49) and (2.71, 1.97) .. (2.95, 1.56);
    \draw (3.18, 1.17) .. controls (3.37, 0.84) and (3.70, 0.58) .. (4.08, 0.60);
    \draw (4.08, 0.60) .. controls (4.74, 0.62) and (5.40, 0.64) .. (6.07, 0.67);
    \draw (6.61, 0.69) .. controls (7.36, 0.71) and (8.09, 0.80) .. 
          (8.81, 1.00) .. controls (9.77, 1.28) and (9.72, 3.77) .. 
          (9.69, 5.61) .. controls (9.65, 7.61) and (9.61, 9.95) .. 
          (7.91, 10.19) .. controls (7.13, 10.30) and (6.21, 10.44) .. 
          (5.89, 9.83) .. controls (5.67, 9.42) and (5.58, 8.95) .. (5.60, 8.49);
    \draw (5.60, 8.49) .. controls (5.62, 7.80) and (5.65, 7.12) .. (5.68, 6.43);
    \draw (5.70, 5.88) .. controls (5.73, 5.10) and (6.20, 4.42) .. 
          (6.88, 4.02) .. controls (7.43, 3.69) and (8.00, 3.29) .. 
          (8.05, 2.66) .. controls (8.09, 2.06) and (7.29, 1.97) .. (6.57, 1.97);
    \draw (6.02, 1.96) .. controls (5.37, 1.96) and (4.72, 1.96) .. (4.08, 1.95);
    \draw (4.08, 1.95) .. controls (3.66, 1.95) and (3.32, 1.67) .. (3.09, 1.32);
  \end{scope}
\end{tikzpicture}

\caption{Borromean Ring FAL with half-twist.}
\label{Brht}
\end{figure} See Figure \ref{Brht}. Using the results from the Borromean Ring FAL without half-twist, we get $$u_2 = u_3 = -\omega_1 = -\omega_2 = \pm\frac{i}{2}.$$

Thus there are three cusps:
Cusp $p$ with cusp shape $$ 4\omega_2 = 4\times \frac{i}{2}=2i.$$
Cusp $q$ with cusp shape $$ \frac{4\omega_1}{1-2\omega_1} = \frac{4\times\frac{i}{2}}{1-2\times\frac{i}{2}} = -1+i.$$
 The cusp shape for the single component in the projection plane, has rectangular cusp with longitude $$u_1 + u_2 + u_3 + u_4 = 4\times\frac{i}{2} = 2i.$$ 

\subsubsection{3 Pretzel FAL without half-twist}
\begin{figure}
\centering
\definecolor{linkcolor0}{rgb}{0.85, 0.15, 0.15}
\definecolor{linkcolor1}{rgb}{0.15, 0.15, 0.85}
\definecolor{linkcolor2}{rgb}{0.15, 0.85, 0.15}
\definecolor{linkcolor3}{rgb}{0.15, 0.85, 0.85}
\definecolor{linkcolor4}{rgb}{0.85, 0.15, 0.85}
\definecolor{linkcolor5}{rgb}{0.85, 0.85, 0.15}
\begin{tikzpicture}[every node/.style={transform shape},scale=.5][line width=2.8, line cap=round, line join=round]
  \begin{scope}[color=linkcolor0]
    \draw (0.71, 6.26) .. controls (0.28, 6.26) and (0.14, 5.73) .. 
          (0.15, 5.23) .. controls (0.16, 4.72) and (0.41, 4.23) .. (0.85, 4.27);
    \draw (0.85, 4.27) .. controls (1.08, 4.28) and (1.30, 4.30) .. (1.53, 4.32);
    \draw (1.53, 4.32) .. controls (2.01, 4.36) and (2.30, 4.83) .. 
          (2.31, 5.33) .. controls (2.33, 5.80) and (2.15, 6.28) .. (1.75, 6.27);
    \draw (1.45, 6.27) .. controls (1.30, 6.27) and (1.15, 6.27) .. (1.00, 6.26);
  \end{scope}
  \begin{scope}[color=linkcolor1]
    \draw (4.78, 6.39) .. controls (4.32, 6.40) and (4.12, 5.86) .. 
          (4.08, 5.34) .. controls (4.03, 4.80) and (4.31, 4.28) .. (4.80, 4.27);
    \draw (4.80, 4.27) .. controls (4.98, 4.26) and (5.16, 4.26) .. (5.34, 4.25);
    \draw (5.34, 4.25) .. controls (5.82, 4.24) and (6.10, 4.76) .. 
          (6.12, 5.29) .. controls (6.14, 5.81) and (6.02, 6.36) .. (5.58, 6.37);
    \draw (5.34, 6.37) .. controls (5.23, 6.38) and (5.13, 6.38) .. (5.03, 6.38);
  \end{scope}
  \begin{scope}[color=linkcolor2]
    \draw (8.37, 6.34) .. controls (7.91, 6.33) and (7.80, 5.76) .. 
          (7.76, 5.25) .. controls (7.71, 4.70) and (7.94, 4.14) .. (8.43, 4.12);
    \draw (8.43, 4.12) .. controls (8.68, 4.10) and (8.93, 4.09) .. (9.17, 4.07);
    \draw (9.17, 4.07) .. controls (9.65, 4.05) and (9.78, 4.65) .. 
          (9.79, 5.21) .. controls (9.79, 5.75) and (9.80, 6.38) .. (9.36, 6.37);
    \draw (9.10, 6.36) .. controls (8.95, 6.36) and (8.80, 6.35) .. (8.65, 6.35);
  \end{scope}
  \begin{scope}[color=linkcolor3]
    \draw (4.93, 6.38) .. controls (4.99, 7.38) and (4.31, 8.28) .. 
          (3.36, 8.29) .. controls (2.35, 8.29) and (1.64, 7.34) .. (1.60, 6.27);
    \draw (1.60, 6.27) .. controls (1.58, 5.70) and (1.56, 5.12) .. (1.54, 4.54);
    \draw (1.52, 4.10) .. controls (1.50, 3.32) and (2.25, 2.79) .. 
          (3.09, 2.77) .. controls (3.94, 2.75) and (4.74, 3.25) .. (4.78, 4.04);
    \draw (4.81, 4.49) .. controls (4.85, 5.12) and (4.89, 5.75) .. (4.93, 6.38);
  \end{scope}
  \begin{scope}[color=linkcolor4]
    \draw (8.50, 6.34) .. controls (8.52, 7.24) and (7.90, 8.03) .. 
          (7.03, 8.06) .. controls (6.15, 8.10) and (5.48, 7.30) .. (5.44, 6.37);
    \draw (5.44, 6.37) .. controls (5.41, 5.74) and (5.38, 5.11) .. (5.35, 4.47);
    \draw (5.33, 4.03) .. controls (5.30, 3.27) and (6.02, 2.72) .. 
          (6.83, 2.76) .. controls (7.63, 2.80) and (8.41, 3.17) .. (8.43, 3.89);
    \draw (8.44, 4.34) .. controls (8.46, 5.01) and (8.48, 5.68) .. (8.50, 6.34);
  \end{scope}
  \begin{scope}[color=linkcolor5]
    \draw (0.85, 6.26) .. controls (0.86, 8.49) and (2.83, 10.18) .. 
          (5.12, 10.18) .. controls (7.38, 10.19) and (9.33, 8.53) .. (9.25, 6.36);
    \draw (9.25, 6.36) .. controls (9.23, 5.67) and (9.20, 4.98) .. (9.18, 4.29);
    \draw (9.17, 3.85) .. controls (9.09, 1.71) and (7.14, 0.15) .. 
          (4.94, 0.25) .. controls (2.73, 0.34) and (0.84, 1.92) .. (0.85, 4.04);
    \draw (0.85, 4.49) .. controls (0.85, 5.08) and (0.85, 5.67) .. (0.85, 6.26);
  \end{scope}
 \draw[red, very thick] (.85,3.5) -- (1.7,3.5); 
 \draw[red, very thick] (4.6,3.5) -- (5.45,3.5); 
   \draw[red, very thick] (8.35,3.5) -- (9.15,3.5); 
    
  \draw[red, very thick] (.95,7.1) -- (1.8,7.1); 
  \draw[red, very thick] (4.85,7.1) -- (5.6,7.1);
  \draw[red, very thick] (8.35,7.1) -- (9.15,7.1); 
  
  \node[black] at (3.3,5.6) {$\gimel$};
  
  \node[thick, black] at (4,.35) {\large\textgreater};
  
  \node[thick, black] at (3.6,2.9) {\large\textgreater};
   \node[thick, black] at (7.5,2.9) {\large\textgreater};
   
   
    \node[thick, black] at (6.15,5.3) {\large$\wedge$};
   
   \node[thick, black] at (4.15,5.3) {\large$\wedge$};
   \node[thick, black] at (2.25,5.3) {\large$\wedge$};
   \node[thick, black] at (7.8,5.3) {\large$\wedge$};

  \node at (8.5,3) {$-\frac{1}{4}$}; 
   \node at (8.5,7.6) {$-\frac{1}{4}$};  
     \node at (1.5,3) {$-\frac{1}{4}$};
     \node at (1.5,7.6) {$-\frac{1}{4}$};
     \node at (5,3) {$\frac{1}{4}$};
     \node at (5.2,7.6) {$\frac{1}{4}$};
   
   \node[black] at (3,3.2) {$u_4$}; \node[black] at (3,2.4) {$u_4$};
   
   \node[black] at (3.4,8.6) {$u_2$}; 
    \node[black] at (3.4,7.9) {$u_2$}; 
    \node[black] at (7,8.4) {$u_3$};  
    
    \node[black] at (7,7.64) {$u_3$};

   \node [black] at  (7,5.6) {$\daleth$};
   
   \node[black] at (5,9.7) {$u_1$};
  \node[black] at (5,9) {$\aleph$};
  
  \node[black] at (6.9,2.4) {$u_5$};
   \node[black] at (6.9,3.2) {$u_5$};
   \node[black] at (5,.5) {$u_6$};
   
   \node[black] at (7.5,5.3) {$1$};
   \node[black] at (6.4,5.3) {$1$};
  
   \node[black] at (3.9,5.3) {$1$};
  
   \node[black] at (2.6,5.3) {$1$};

   \node[black] at (5,1.7) {$\beth$};
 
   \node[black] at (.5,4) {$\omega_1$};
   \node[black] at (.5,6.54) {$\omega_1$};
   \node[black] at (2,4) {$\omega_1$};
   \node[black] at (2,6.55) {$\omega_1$};
   \node[black] at (4.35,4) {$\omega_2$};
  \node[black] at (5.8,4) {$-\omega_2$};
  \node[black] at (5.9,6.7) {$-\omega_2$};
 \node[black] at (4.35,6.7) {$\omega_2$}; 
 \node[black] at (8,4) {$\omega_3$}; 
\node[black] at (9.6,3.85) {$\omega_3$};  
 
 \node[black] at (8,6.6) {$\omega_3$}; 
\node[black] at (9.6,6.65) {$\omega_3$};

\end{tikzpicture}

\caption{Three pretzel FAL.}
\label{3pfal}
\end{figure}
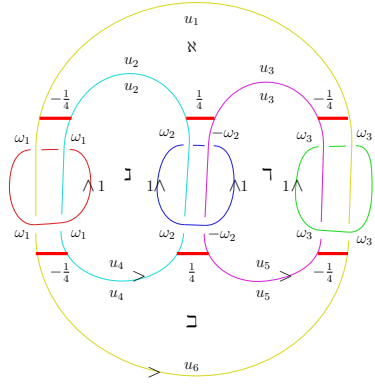

Region $\aleph$: 
This is a three-sided region with shape parameters:
$$\xi_1=\frac{-\frac{1}{4}}{u_1u_3} = 1, \quad \xi_2=\frac{-\frac{1}{4}}{u_1u_2} = 1,
\quad \xi_3=\frac{-\frac{1}{4}}{u_2u_3} = 1$$ solving gives us the relations $$u_1 = u_2 =u_3 \quad \textrm{and} \quad u_1 = \pm\frac{i}{2}$$
Region $\beth$: 
This is a three-sided region with shape parameters:
$$\xi_1=\frac{-\frac{1}{4}}{u_4u_6} = 1, \quad \xi_2=\frac{-\frac{1}{4}}{u_4u_5} = 1,
\quad \xi_3=\frac{-\frac{1}{4}}{u_5u_6} = 1$$ solving gives us the relations $$u_4 = u_5 =u_6 \quad \textrm{and} \quad u_4 = \pm\frac{i}{2}$$
Region $\gimel$: 
$$\xi_1=\frac{\omega_1}{u_{4}}, \quad \xi_2=\frac{-\omega_2}{u_{4}},
\quad \xi_3=\frac{-\omega_2}{u_{2}}, \quad \xi_4=\frac{\omega_1}{u_{2}}$$ This is a four-sided region with equations:
$$\frac{\omega_1}{u_{4}} - \frac{\omega_2}{u_{4}} = 1 , \quad \frac{-\omega_2}{u_{4}} - \frac{\omega_2}{u_{2}} = 1,  \quad \frac{-\omega_2}{u_{2}} + \frac{\omega_1}{u_{2}} = 1,  \quad \frac{\omega_1}{u_{2}} + \frac{\omega_1}{u_{4}} = 1$$ solving gives us the relations
$$u_{2} = u_{4}, \quad \omega_1 = -\omega_2, \quad \textrm{and} \quad u_{2} = 2\omega_1.$$
Region $\daleth$: 
$$\xi_1=\frac{-\omega_2}{u_{5}}, \quad \xi_2=\frac{-\omega_3}{u_{5}},
\quad \xi_3=\frac{-\omega_3}{u_{3}}, \quad \xi_4=\frac{-\omega_2}{u_{3}}$$  solving gives us the relations
$$u_{3} = u_{5}, \quad \omega_2 = \omega_3,  \quad u_{3} = -2\omega_2$$ all $$u_i = \pm\frac{i}{2} \quad  \textrm{and} \quad \omega_i = \pm\frac{i}{4}.$$
The cusp shapes for all $6$ components are equal to $i$.
The three crossing circles have cusp shape $$4\omega_i = 4\times \frac{i}{4} = i.$$
The component $s + y$ has cusp shape  $$u_1 + u_6 = 2\times\frac{i}{2} = i.$$
The component $t + v$ has cusp shape  $$u_2 + u_4 = 2\times\frac{i}{2} = i.$$
The component $u + x$ has cusp shape  $$u_3 + u_5 = 2\times\frac{i}{2} = i.$$
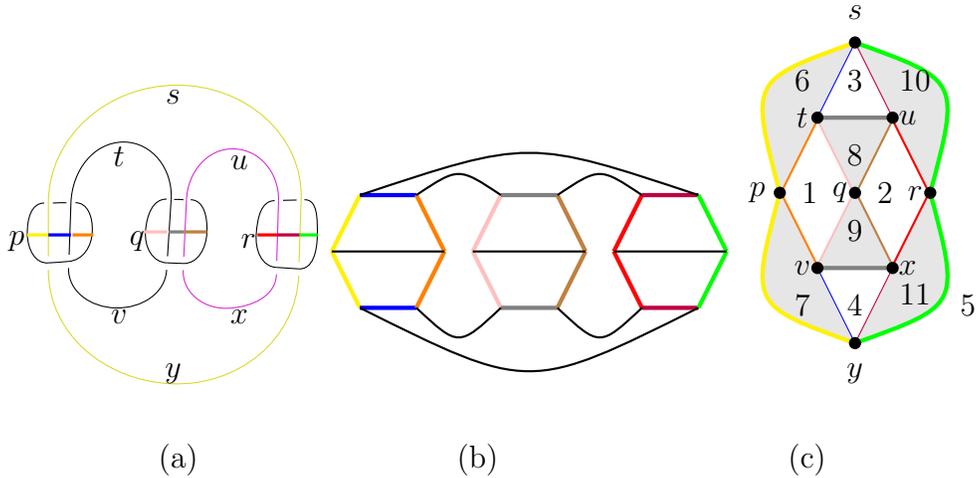
\begin{figure}[h]
\centering
\definecolor{linkcolor0}{rgb}{0.85, 0.15, 0.15}
\definecolor{linkcolor1}{rgb}{0.15, 0.15, 0.85}
\definecolor{linkcolor2}{rgb}{0.15, 0.85, 0.15}
\definecolor{linkcolor3}{rgb}{0.15, 0.85, 0.85}
\definecolor{linkcolor4}{rgb}{0.85, 0.15, 0.85}
\definecolor{linkcolor5}{rgb}{0.85, 0.85, 0.15}
\begin{tikzpicture}[scale=.4][line width=2.8, line cap=round, line join=round]
    \draw (0.71, 6.26) .. controls (0.28, 6.26) and (0.14, 5.73) .. 
          (0.15, 5.23) .. controls (0.16, 4.72) and (0.41, 4.23) .. (0.85, 4.27);
    \draw (0.85, 4.27) .. controls (1.08, 4.28) and (1.30, 4.30) .. (1.53, 4.32);
    \draw (1.53, 4.32) .. controls (2.01, 4.36) and (2.30, 4.83) .. 
          (2.31, 5.33) .. controls (2.33, 5.80) and (2.15, 6.28) .. (1.75, 6.27);
    \draw (1.45, 6.27) .. controls (1.30, 6.27) and (1.15, 6.27) .. (1.00, 6.26);
    \draw (4.78, 6.39) .. controls (4.32, 6.40) and (4.12, 5.86) .. 
          (4.08, 5.34) .. controls (4.03, 4.80) and (4.31, 4.28) .. (4.80, 4.27);
    \draw (4.80, 4.27) .. controls (4.98, 4.26) and (5.16, 4.26) .. (5.34, 4.25);
    \draw (5.34, 4.25) .. controls (5.82, 4.24) and (6.10, 4.76) .. 
          (6.12, 5.29) .. controls (6.14, 5.81) and (6.02, 6.36) .. (5.58, 6.37);
    \draw (5.34, 6.37) .. controls (5.23, 6.38) and (5.13, 6.38) .. (5.03, 6.38);
    \draw (8.37, 6.34) .. controls (7.91, 6.33) and (7.80, 5.76) .. 
          (7.76, 5.25) .. controls (7.71, 4.70) and (7.94, 4.14) .. (8.43, 4.12);
    \draw (8.43, 4.12) .. controls (8.68, 4.10) and (8.93, 4.09) .. (9.17, 4.07);
    \draw (9.17, 4.07) .. controls (9.65, 4.05) and (9.78, 4.65) .. 
          (9.79, 5.21) .. controls (9.79, 5.75) and (9.80, 6.38) .. (9.36, 6.37);
    \draw (9.10, 6.36) .. controls (8.95, 6.36) and (8.80, 6.35) .. (8.65, 6.35);
    \draw (4.93, 6.38) .. controls (4.99, 7.38) and (4.31, 8.28) .. 
          (3.36, 8.29) .. controls (2.35, 8.29) and (1.64, 7.34) .. (1.60, 6.27);
    \draw (1.60, 6.27) .. controls (1.58, 5.70) and (1.56, 5.12) .. (1.54, 4.54);
    \draw (1.52, 4.10) .. controls (1.50, 3.32) and (2.25, 2.79) .. 
          (3.09, 2.77) .. controls (3.94, 2.75) and (4.74, 3.25) .. (4.78, 4.04);
    \draw (4.81, 4.49) .. controls (4.85, 5.12) and (4.89, 5.75) .. (4.93, 6.38);
  \begin{scope}[color=linkcolor4]
    \draw (8.50, 6.34) .. controls (8.52, 7.24) and (7.90, 8.03) .. 
          (7.03, 8.06) .. controls (6.15, 8.10) and (5.48, 7.30) .. (5.44, 6.37);
    \draw (5.44, 6.37) .. controls (5.41, 5.74) and (5.38, 5.11) .. (5.35, 4.47);
    \draw (5.33, 4.03) .. controls (5.30, 3.27) and (6.02, 2.72) .. 
          (6.83, 2.76) .. controls (7.63, 2.80) and (8.41, 3.17) .. (8.43, 3.89);
    \draw (8.44, 4.34) .. controls (8.46, 5.01) and (8.48, 5.68) .. (8.50, 6.34);
  \end{scope}
  \begin{scope}[color=linkcolor5]
    \draw (0.85, 6.26) .. controls (0.86, 8.49) and (2.83, 10.18) .. 
          (5.12, 10.18) .. controls (7.38, 10.19) and (9.33, 8.53) .. (9.25, 6.36);
    \draw (9.25, 6.36) .. controls (9.23, 5.67) and (9.20, 4.98) .. (9.18, 4.29);
    \draw (9.17, 3.85) .. controls (9.09, 1.71) and (7.14, 0.15) .. 
          (4.94, 0.25) .. controls (2.73, 0.34) and (0.84, 1.92) .. (0.85, 4.04);
    \draw (0.85, 4.49) .. controls (0.85, 5.08) and (0.85, 5.67) .. (0.85, 6.26);
  \end{scope}
 
   \draw[yellow, very thick] (.15,5.2) -- (.84,5.2); 
  \draw[blue, very thick] (.85,5.2) -- (1.6,5.2); 
  \draw[orange, very thick] (1.65,5.2) -- (2.35,5.2); 
  
   \draw[pink, very thick] (4.8,5.3) -- (4,5.3);
  \draw[gray, very thick] (4.85,5.3) -- (5.35,5.3);
   \draw[brown, very thick] (6.12,5.3) -- (5.35,5.3);
   \draw[red, very thick] (8.45,5.2) -- (7.8,5.2); 
  \draw[purple, very thick] (8.45,5.2) -- (9.2,5.2); 
  \draw[green, very thick] (9.8,5.2) -- (9.2,5.2); 
   \node at (-.2,5) {$p$};
     
    \node at (3.8,5) {$q$};
      
     \node at (7.5,5) {$r$};
     \node at (5,9.8) {$s$};
     \node at (3.2,7.8) {$t$};
     \node at (7.2,7.6) {$u$};
     \node at (3.2,2.5) {$v$};
     \node at (7.2,2.5) {$x$}; 
     \node at (5,.6) {$y$};
  \end{tikzpicture}
  \begin{tikzpicture}[scale=.75]
  \draw[yellow, ultra thick](0.5,2.5)--(1,1.5);
  \draw[yellow, ultra thick](0.5,2.5)--(1,3.5);
  \draw[blue, ultra thick](2,1.5)--(1,1.5);
  \draw[blue, ultra thick](1,3.5)--(2,3.5);
  \draw[orange, ultra thick](2,3.5)--(2.5,2.5);
   \draw[orange, ultra thick](2,1.5)--(2.5,2.5);
  
   \draw[pink, ultra thick](3,2.5)--(3.5,3.5);
   
   \draw[pink, ultra thick](3,2.5)--(3.5,1.5);
   \draw[gray, ultra thick](4.5,3.5)--(3.5,3.5);
   
   \draw[gray, ultra thick](4.5,1.5)--(3.5,1.5);
   \draw[brown, ultra thick](4.5,3.5)--(5,2.5);
   \draw[brown, ultra thick](4.5,1.5)--(5,2.5);
   
    \draw[red, ultra thick](6,1.5)--(5.5,2.5);
    
   \draw[red, ultra thick](6,3.5)--(5.5,2.5);
   
   \draw[purple, ultra thick](6,1.5)--(7,1.5);
    \draw[purple, ultra thick](6,3.5)--(7,3.5);
   
    \draw[green, ultra thick](7.5,2.5)--(7,1.5);
   
   \draw[green, ultra thick](7.5,2.5)--(7,3.5);

  \draw[thick](1,1.5).. controls (4,0) .. (7,1.5); 
  \draw[thick](1,3.5).. controls (4,4.5) .. (7,3.5);
  
   \draw[thick](2,3.5).. controls (2.75,4) .. (3.5,3.5);

  \draw[thick](2,1.5).. controls (2.75,0.8) .. (3.5,1.5);
    \draw[thick](4.5,1.5).. controls (5.25,0.8) .. (6,1.5);
  \draw[thick](4.5,3.5).. controls (5.25,4) .. (6,3.5);
  \draw[thick](.5,2.5)--(2.5,2.5);
  
   \draw[thick](3,2.5)--(5,2.5);
   \draw[thick](5.5,2.5)--(7.5,2.5);
  
  \end{tikzpicture}
  \begin{tikzpicture}
  
   \draw[orange, ultra thick](1,2.5)--(1.5,1.5); 
    
     \draw[orange, ultra thick](1,2.5)--(1.5,3.5); 
    
     \draw[pink, ultra thick](2,2.5)--(1.5,1.5); 
     \draw[pink, ultra thick](2,2.5)--(1.5,3.5); 
     \draw[brown, ultra thick](2,2.5)--(2.5,1.5); 
    
     \draw[brown, ultra thick](2,2.5)--(2.5,3.5); 
     \fill[gray!20]
     (2,2.5)--(1.5,1.5)--(2.5,1.5);
     \fill[gray!20]
     (2,2.5)--(1.5,3.5)--(2.5,3.5);
     
    \draw[red, ultra thick](3,2.5)--(2.5,1.5); 
     
     \draw[red, ultra thick](3,2.5)--(2.5,3.5); 
      \draw[gray, ultra thick](1.5,1.5)--(2.5,1.5); 
        \draw[gray, ultra thick](1.5,3.5)--(2.5,3.5); 
     
     \fill[gray!20](2,4.5).. controls (.6,4)..(1,2.5);
     \draw[yellow, ultra thick](2,4.5).. controls (.6,4)..(1,2.5);
     
    \fill[gray!20](2,.5).. controls (.6,1)..(1,2.5);

     \draw[yellow, ultra thick](2,.5).. controls (.6,1)..(1,2.5);
     
    \fill[gray!20](2,4.5).. controls (3.4,4)..(3,2.5);

     \draw[green, ultra thick](2,4.5).. controls (3.4,4)..(3,2.5);
     \fill[gray!20](3,2.5).. controls (3.4,1)..(2,.5);
      \draw[green, ultra thick](3,2.5).. controls (3.4,1)..(2,.5);
      
      \draw[blue](2,4.5)--(1.5,3.5);
      \draw[blue](2,.5)--(1.5,1.5);
     \draw[purple](2,.5)--(2.5,1.5);
     \draw[purple](2,4.5)--(2.5,3.5);
     
     \fill (1,2.5) circle (.081);
    \fill (2,2.5) circle (.081);
      \fill (3,2.5) circle (.081);
     \fill (1.5,1.5) circle (.081);  \fill (2.5,1.5) circle (.081);  \fill (1.5,3.5) circle (.081);  \fill (2.5,3.5) circle (.081);
    \fill (2,.5) circle (.081);   \fill (2,4.5) circle (.081);
     \node at (.7,2.5) {$p$};
     
    \node at (1.8,2.5) {$q$};
      
     \node at (2.8,2.5) {$r$};
     \node at (2,4.9) {$s$};
     \node at (1.3,3.5) {$t$};
     \node at (2.7,3.5) {$u$};
     \node at (1.3,1.5) {$v$};
     \node at (2.7,1.5) {$x$}; 
     \node at (2,.1) {$y$};
     \node at (1.4,2.5) {$1$};
     \node at (2.4,2.5) {$2$};
     \node at (2,4) {$3$};
     
     \node at (2,1) {$4$};
     \node at (3.5,1) {$5$};
     \node at (1.3,4) {$6$};
     \node at (1.3,1) {$7$};
     \node at (2,3) {$8$};
      \node at (2,2) {$9$};
     \node at (2.8,4) {$10$}; 
      \node at (2.8,1.15) {$11$};
  \end{tikzpicture}

 
 (a) \quad \quad \quad \quad \quad \quad \quad \quad (b) \quad \quad \quad \quad \quad \quad \quad \quad \quad (c) 
\caption[$FALP_3$ decomposition and circle packing]{(a)$FALP_3$ with crossing geodesics colored. (b) $T_{FALP_3}$ (c) $P_{FALP_3}$ }
\label{3p}
\end{figure}

\begin{figure}[h]
    \centering
      \begin{tikzpicture}
  \draw (0,0) circle (1);
  \draw (2,0) circle (1);
  \draw (1,1.33) circle (.67);
  \draw (1,-1.33) circle (.67);
  \draw (1,0) circle (2);
  \draw[dashed](1,.5) circle (.5);
  \draw[dashed](1,-.5) circle (.5);
  \node at (-.7,0) {$p$};
     
    \node at (.75,0) {$q$};
      
     \node at (2.7,0) {$r$}; 
   \node at (1,2.2) {$s$};
     \node at (.4,.6) {$t$};
     \node at (1.7,.6) {$u$};
     \node at (.4,-.7) {$v$};
     \node at (1.7,-.7) {$x$}; 
     \node at (1,-2.2) {$y$};
  \node at (0,0) {$1$};
     \node at (2,0) {$2$};
     \node at (1,1.5) {$3$};
     
     \node at (1,-1.65) {$4$};
     \node at (3,1) {$5$};
    \node at (1,.5) {$8$};
      \node at (1,-.5) {$9$};
  \end{tikzpicture}
  \hspace{.1cm}
  \begin{tikzpicture}[scale=1.4]
  \draw (0,0) circle (.5);
  \draw (1,0) circle (.5);
  \draw (2,0) circle (.5);
  \draw[ultra thick] (-.5,.5) --(2.5,.5);
   \draw[ultra thick] (-.5,-.5) --(2.5,-.5);
    \draw[ultra thick, dashed] (0,.5) --(0,-.5);
    \draw[ultra thick, dashed] (2,.5) --(2,-.5);
    \node at (1,.75) {$1$};
    \node at (1,-.75) {$2$};
    \node at (-.1,0) {$8$};
     \node at (2.1,0) {$9$};
    
  \end{tikzpicture}
  \hspace{.1cm}
  \begin{tikzpicture}[scale=.8]
\fill [gray!20] (0,.3) circle (.3);
\fill [gray!20] (2,.3) circle (.3);
\fill [gray!20] (1.5,.8) circle (.3);
\fill [gray!20] (4,.3) circle (.3);

\fill [gray!20] (5.5,.8) circle (.3);

\fill [gray!20] (3.5,.8) circle (.3);
\draw [orange, dashed, ultra thick] (5.5,0.5) arc (0:180:1);
\draw [orange, dashed, ultra thick] (3.5,0.5) arc (0:180:1);
\draw [red, dashed, ultra thick] (4,0) arc (0:180:1);
\draw [red, dashed, ultra thick] (2,0) arc (0:180:1);
\draw [brown, dashed, ultra thick] (0,0) -- (0,3);

\draw [pink, dashed, ultra thick] (1.5,0.5) -- (1.5,3.5);
\draw [brown, dashed, ultra thick] (4,0) -- (4,3);
\draw [pink, dashed, ultra thick] (5.5,0.5) -- (5.5,3.5);

\draw [gray, dashed, ultra thick] (1.5,0.5) arc (0:180:.75);
\draw [gray, dashed, ultra thick] (5.5,0.5) arc (0:200:.75);

\draw [ultra thick] (0,3) -- (4,3);
\draw [ ultra thick] (1.5,3.5) -- (5.5,3.5);
\draw[ultra thick] (0,3) -- (1.5,3.5);
\draw[ultra thick] (4,3) -- (5.5,3.5);

\node at (1.5,0.3){\bf\large $t$};
\node at (5.5,0.3){\bf\large $v$};
\node at (0,-.2){\bf\large $u$};
\node at (4,-.2){\bf\large $x$};
\node at (3.3,0.3){\bf\large $p$};
\node at (2,-0.2){\bf\large $r$};

\node at (6,2.2){\bf\large $f_9$};

\node at (-0.5,2){\bf\large $f_8$};
\node at (3,4.2,){\bf\large $f_1$};
\node at (2,3.25){\bf\large $f_2$};


\end{tikzpicture}
     (a) \quad \quad \quad \quad \quad \quad \quad \quad (b) \quad \quad \quad \quad \quad \quad \quad \quad \quad (c) 
    \caption[$FALP_3$ circle packing]{(a) White circle packing, with two dual shaded circles drawn on top. (b)  View of rectangle in $\mathbb{H}^3$ formed by taking $q$ to $\infty$ with view from $\infty$ (c)  The faces of 1,2,8, and 9 in $\mathbb{H}^3$, with view from vertical plane $xz$.}
    \label{fal3cp}
\end{figure}
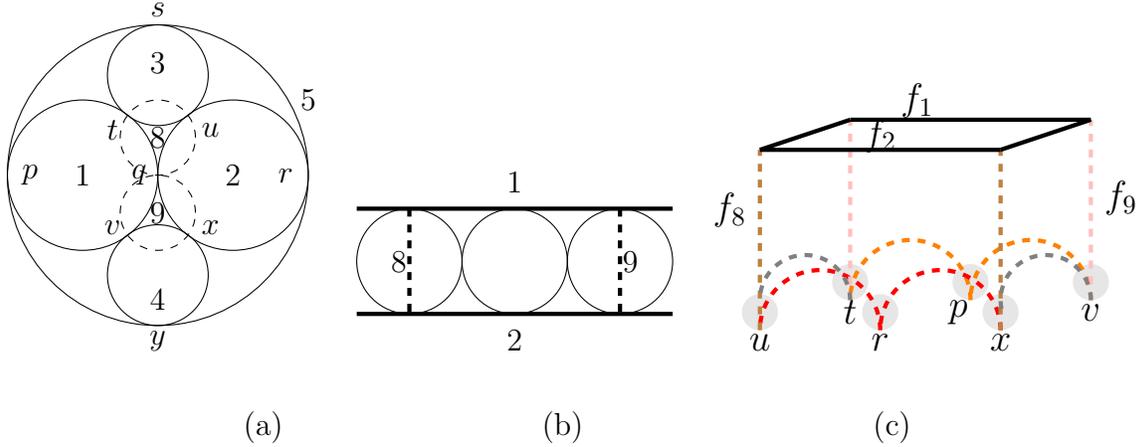

\subsubsection{3 Pretzel FAL with half-twist}

\begin{figure}
\centering
\definecolor{linkcolor0}{rgb}{0.85, 0.15, 0.15}
\definecolor{linkcolor1}{rgb}{0.15, 0.15, 0.85}
\definecolor{linkcolor2}{rgb}{0.15, 0.85, 0.15}
\definecolor{linkcolor3}{rgb}{0.15, 0.85, 0.85}
\definecolor{linkcolor4}{rgb}{0.85, 0.15, 0.85}
\begin{tikzpicture}[line width=1.9, line cap=round, line join=round, scale=.5]
\node at (-.2,5) {$p$};
\node at (3.6,5) {$n$};
\node at (10.2,5) {$m$};
\node at (3.2,7.88) {$r$};

\node at (7,7.6) {$u$};
\node at (5,9.2) {$t$};
\node at (3.2,2.1) {$q$};
\node at (7,2.1) {$v$};
\node at (5,.5) {$s$};

  \begin{scope}[color=linkcolor0]
    \draw (0.89, 6.55) .. controls (0.33, 6.56) and (0.21, 5.84) .. 
          (0.18, 5.20) .. controls (0.14, 4.54) and (0.27, 3.84) .. (0.84, 3.82);
    \draw (0.84, 3.82) .. controls (1.07, 3.81) and (1.30, 3.81) .. (1.53, 3.80);
    \draw (1.53, 3.80) .. controls (2.13, 3.78) and (2.45, 4.46) .. 
          (2.44, 5.15) .. controls (2.43, 5.78) and (2.41, 6.53) .. (1.92, 6.54);
    \draw (1.65, 6.54) .. controls (1.50, 6.54) and (1.35, 6.54) .. (1.20, 6.55);
  \end{scope}
  \begin{scope}[color=linkcolor1]
    \draw (4.49, 6.40) .. controls (3.97, 6.41) and (3.93, 5.72) .. 
          (3.92, 5.11) .. controls (3.91, 4.50) and (3.99, 3.81) .. (4.52, 3.79);
    \draw (4.52, 3.79) .. controls (4.92, 3.78) and (5.32, 3.76) .. (5.72, 3.75);
    \draw (5.72, 3.75) .. controls (6.27, 3.73) and (6.48, 4.39) .. 
          (6.47, 5.04) .. controls (6.46, 5.64) and (6.46, 6.36) .. (5.98, 6.37);
    \draw (5.64, 6.37) .. controls (5.38, 6.38) and (5.11, 6.38) .. (4.85, 6.39);
  \end{scope}
  \begin{scope}[color=linkcolor2]
    \draw (8.31, 6.28) .. controls (7.74, 6.28) and (7.50, 5.60) .. 
          (7.54, 4.94) .. controls (7.57, 4.26) and (7.91, 3.59) .. (8.53, 3.58);
    \draw (8.53, 3.58) .. controls (8.75, 3.58) and (8.97, 3.58) .. (9.20, 3.57);
    \draw (9.20, 3.57) .. controls (9.74, 3.57) and (9.81, 4.28) .. 
          (9.80, 4.92) .. controls (9.79, 5.57) and (9.67, 6.28) .. (9.11, 6.28);
    \draw (8.86, 6.28) .. controls (8.77, 6.28) and (8.69, 6.28) .. (8.61, 6.28);
  \end{scope}
  \begin{scope}[color=linkcolor3]
    \draw (1.05, 6.55) .. controls (1.18, 8.21) and (3.11, 8.85) .. 
          (4.97, 8.81) .. controls (6.86, 8.76) and (8.78, 8.00) .. (8.94, 6.28);
    \draw (8.94, 6.28) .. controls (9.02, 5.45) and (9.10, 4.63) .. (9.17, 3.80);
    \draw (9.22, 3.35) .. controls (9.40, 1.41) and (7.33, 0.15) .. 
          (5.12, 0.23) .. controls (3.03, 0.30) and (0.85, 0.78) .. (1.33, 2.32);
    \draw (1.33, 2.32) .. controls (1.41, 2.56) and (1.46, 3.15) .. (1.51, 3.57);
    \draw (1.55, 4.02) .. controls (1.63, 4.86) and (1.72, 5.70) .. (1.80, 6.54);
    \draw (1.80, 6.54) .. controls (1.87, 7.23) and (2.56, 7.64) .. 
          (3.29, 7.62) .. controls (4.02, 7.60) and (4.65, 7.09) .. (4.63, 6.39);
    \draw (4.63, 6.39) .. controls (4.59, 5.60) and (4.56, 4.81) .. (4.53, 4.02);
    \draw (4.51, 3.57) .. controls (4.48, 2.70) and (4.15, 1.80) .. 
          (3.35, 1.72) .. controls (2.70, 1.65) and (2.03, 1.78) .. (1.51, 2.18);
    \draw (1.22, 2.41) .. controls (0.86, 2.68) and (0.78, 3.16) .. (0.82, 3.61);
    \draw (0.85, 4.05) .. controls (0.92, 4.88) and (0.98, 5.71) .. (1.05, 6.55);
  \end{scope}
  \begin{scope}[color=linkcolor4]
    \draw (8.52, 6.28) .. controls (8.52, 6.90) and (7.89, 7.27) .. 
          (7.22, 7.34) .. controls (6.55, 7.40) and (5.90, 7.00) .. (5.87, 6.37);
    \draw (5.87, 6.37) .. controls (5.82, 5.57) and (5.78, 4.77) .. (5.73, 3.97);
    \draw (5.70, 3.52) .. controls (5.65, 2.59) and (6.21, 1.72) .. 
          (7.07, 1.70) .. controls (7.92, 1.67) and (8.54, 2.46) .. (8.53, 3.36);
    \draw (8.53, 3.81) .. controls (8.53, 4.63) and (8.52, 5.46) .. (8.52, 6.28);
  \end{scope}
\end{tikzpicture}
\caption[$FALP_3$ with half-twist]{Diagram of 3-pretzel FAL with half-twist with components labeled.}
\label{3pfalwt}
\end{figure}
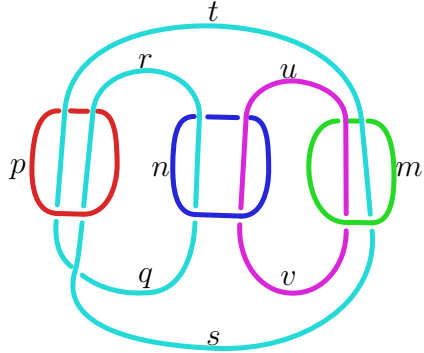

In the diagram the five-sided region does not correspond to a five-sided ideal polygon, rather the polyhedral decomposition is the same as in the 3-pretzel without any half-twist, but the gluing of the faces change.  Thus to find the cusp shapes we first obtain the parameters from the 3-pretzel without any half-twist and then calculate the cusps off those parameters and the above theorems.
Using the information from the 3-pretzel FAL without a half-twist, we have  
$\displaystyle{u_i = \pm\frac{i}{2}}$ and  $\displaystyle{\omega_i = \pm\frac{i}{4}}$.
The cusp shape for the red crossing circle $p$ is $$\frac{4\omega_1}{1-2\omega_1} = \frac{4\times\frac{i}{4}}{1-2\times\frac{i}{4}} = \frac{-2}{5}+\frac{4i}{5}.$$
The cusp shapes for the blue and green crossing circles $n$ and $m$ respectively are $$4\omega_2 = 4 \times \frac{i}{4} = i.$$
The cusp shape for the light blue component in the projection plane $r + s + t + q$ is $$u_2 + u_6 + u_1 + u_4 -4\frac{1}{2} = 4\times \frac{i}{2}-2 = -2 + 2i.$$
The cusp shape for the pink component in the projection plane $u + v$ is  $$u_3 + u_5 = 2\frac{i}{2} = i.$$

\section{Applications}
\label{sec:itf}
\subsection{Invariant Trace Fields of FAL Complements.}

Let $M$ be a complete orientable finite volume hyperbolic $3$-manifold, then $M=\mathbb{H}^3/\Gamma$ where $\Gamma=\pi_1(M)$  (a Kleinian group)
is a discrete subgroup of $PSL(2, \mathbb{C}) = Isom^+(\mathbb{H}^3
).$  Let $\rho : SL(2,\mathbb{C}) \rightarrow PSL(2,\mathbb{C})$ be quotient map and let $\overline{\Gamma}=\rho^{-1}(\Gamma)$ \cite{callahan1998hyperbolic}. 
\begin{define}
 The \emph{trace field} $KM = K\Gamma$ is the field over $\mathbb{Q}$ generated by all the traces of $\Gamma$, i.e. $K\Gamma := \mathbb{Q}(\{tr(\gamma) | \gamma \in \overline{\Gamma}\}).$
 The \emph{invariant trace field} of $\Gamma$ is $kM = k\Gamma$ := $K\Gamma^{(2)}$ where $\Gamma^{(2)} : = <\gamma^2 | \gamma \in \Gamma>.$
\end{define}
It follows from Mostow-Prasad rigidity that $KM$ and $kM$ are number fields, i.e. finite extensions of $\mathbb{Q}$ and are invariants of $M$.
\begin{define}
 $M_1$ and $M_2$ are \emph{commensurable} if they have common finite-sheeted covers.
\end{define}
\begin{theorem}\cite{neumann1992arithmetic}
The invariant trace field $kM$ is an invariant of the commensurability class
of $M.$
\end{theorem}

\begin{define}
Let $M$ be a cusped hyperbolic 3-manifold. The field generated by the cusp shapes of all the cusps of $M$ is called the
\emph{cusp field} of M, $cM.$
 
\end{define}

It follows from results of Neumann-Reid in \cite{neumann1992arithmetic} that $cM$ is contained in $kM$ and is a commensurability invariant. It is often the case that for a link complement $M$, $cM = kM$. 

The polynomials we derive from the T-T method in terms of intercusp and translational parameters play a central role in studying the invariant trace fields of FAL complements.

If the images of the intercusp geodesics and translational geodesics are embedded in $M$, in which case  we call them \emph{intercusp arcs} and \emph{cusp arcs} (respectively) then the following theorem holds.
\begin{theorem}\label{nt}\cite{nt}
Suppose $X \subset M$ is a union of cusp arcs and pairwise disjoint intercusp
arcs, where any intercusp arcs which are not disjoint have been bent slightly
near intersection points to make them disjoint, and suppose $\pi_1(X) \rightarrow \pi_1(M)$ is
surjective. Then the intercusp and
translation parameters corresponding to these arcs generate $kM$.

\end{theorem}
\begin{theorem}\label{cm}
Let $M$ be a FAL complement then $cM = kM.$
\end{theorem}
\begin{proof}
$kM$ is generated by all the meridian curves of the overstrands of the link diagram. Let $X$ be the union of cusp arcs and pairwise disjoint intercusp arcs, see Figure \ref{mu}. The meridians are all realized by the translation parameters, while the intercusp geodesics ensure that $\overset{~}{X}$ is connected. For FALs the intercusp parameters $\omega_i$ and the translational parameters $u_j$ are parameters of the intercusp arcs and cusp arcs, respectively, since FALs decompose into totally geodesic polyhedra. To show that $\pi_1(X) \rightarrow \pi_1(M)$ is surjective, we need to see that all the meridians are included. This is quite explicit, see Figure \ref{mu}.  The meridians for the crossing circles, the meridians for the components in the projection plane, and the meridian that runs around the crossing circle in the projection plane are all combinations of the intercusp and translation parameters. By Theorem \ref{nt}, $\omega_i$s and the $u_j$s generate $kM$.  Moreover by Theorems \ref{mainthm} and \ref{mainthm2} $\omega_i$s and $u_j$s generate the cusp field, thus $cM = kM.$ 
\end{proof}
\begin{figure}
    \centering
    \begin{tikzpicture}[scale=.75,line width=1.7, line cap=round, line join=round]

  \draw[red, dashed] (2.65,3.5) circle (.56cm);
  
  \draw[red, dashed] (7.5,3.5) circle (.63cm);
    \draw[red, dashed] (.8,3.2) circle (.43cm);

    \draw (8.13, 1.71) .. controls (9.08, 1.72) and (9.74, 2.63) .. 
          (9.69, 3.63) .. controls (9.64, 4.56) and (9.34, 5.53) .. (8.52, 5.52);
    \draw (8.14, 5.52) .. controls (8.01, 5.52) and (7.89, 5.52) .. (7.77, 5.52);
    \draw (7.50, 5.52) .. controls (7.35, 5.51) and (7.20, 5.51) .. (7.05, 5.51);
    \draw (6.61, 5.51) .. controls (5.61, 5.50) and (4.62, 5.49) .. (3.63, 5.49);
    \draw (3.19, 5.48) .. controls (3.05, 5.48) and (2.92, 5.48) .. (2.78, 5.48);
    \draw (2.53, 5.48) .. controls (2.41, 5.48) and (2.30, 5.48) .. (2.18, 5.48);
    \draw (1.80, 5.47) .. controls (1.16, 5.47) and (0.73, 4.83) .. (0.74, 4.14);
    \draw (0.74, 4.14) .. controls (0.75, 3.80) and (0.75, 3.46) .. (0.76, 3.12);
    \draw (0.77, 2.61) .. controls (0.78, 2.01) and (1.37, 1.61) .. (2.01, 1.62);
    \draw (2.01, 1.62) .. controls (2.42, 1.63) and (2.83, 1.63) .. (3.24, 1.64);
    \draw (3.24, 1.64) .. controls (4.40, 1.66) and (5.57, 1.67) .. (6.74, 1.69);
    \draw (6.74, 1.69) .. controls (7.20, 1.70) and (7.67, 1.70) .. (8.13, 1.71);

 \draw[dashed, blue] (5.2,3.5) ellipse (4cm and 1.7cm);

   \draw (2.09, 7.10) .. controls (2.09, 6.56) and (2.08, 6.02) .. (2.07, 5.48);
    \draw (2.07, 5.48) .. controls (2.07, 5.31) and (2.07, 5.15) .. (2.06, 4.98);
    \draw (2.06, 4.98) .. controls (2.05, 3.82) and (2.04, 3.07) .. (2.03, 2.33);
    \draw (2.02, 1.95) .. controls (2.02, 1.87) and (2.02, 1.79) .. (2.02, 1.70);
    \draw (2.01, 1.36) .. controls (2.01, 1.02) and (2.00, 0.67) .. (2.00, 0.33);
    \draw (3.21, 0.39) .. controls (3.22, 0.73) and (3.22, 1.06) .. (3.23, 1.39);
    \draw (3.24, 1.72) .. controls (3.24, 1.80) and (3.24, 1.88) .. (3.25, 1.96);
    \draw (3.25, 2.34) .. controls (3.27, 3.09) and (3.29, 3.84) .. (3.31, 4.60);
    \draw (3.32,4.6) .. controls (3.32, 5.17) and (3.33, 5.33) .. (3.33, 5.48);
    \draw (3.33, 5.48) .. controls (3.34, 6.03) and (3.36, 6.58) .. (3.37, 7.13);
 
    \draw (6.97, 7.07) .. controls (6.95, 6.55) and (6.93, 6.03) .. (6.90, 5.51);
    \draw (6.90, 5.51) .. controls (6.90, 5.37) and (6.89, 5.37) .. (6.89, 5.09);
    \draw (6.87,5.09) .. controls (6.83, 3.91) and (6.80, 3.14) .. (6.77, 2.37);
    \draw (6.75, 2.00) .. controls (6.75, 1.92) and (6.75, 1.84) .. (6.74, 1.77);
    \draw (6.73, 1.40) .. controls (6.71, 1.01) and (6.69, 0.62) .. (6.68, 0.23);
    \draw (8.31, 7.13) .. controls (8.29, 6.60) and (8.28, 6.06) .. (8.26, 5.52);
    \draw (8.26, 5.52) .. controls (8.25, 5.39) and (8.25, 5.39) .. (8.25, 5.12);
    \draw (8.23, 5.12) .. controls (8.21, 3.94) and (8.18, 3.16) .. (8.16, 2.38);
    \draw (8.14, 2.01) .. controls (8.14, 1.94) and (8.14, 1.86) .. (8.14, 1.79);
    \draw (8.12, 1.42) .. controls (8.11, 1.01) and (8.10, 0.60) .. (8.08, 0.20);
 
  \draw[dashed, blue] (7.6,7.2) -- (7.5, 3.1);
  \draw[dashed, blue] (7.45, 2.7)--(7.42,2.2);
  
   \draw[dashed, blue] (7.39, 1.6)--(7.35,0.2);

  \draw[dashed, blue] (2.65,7.2) -- (2.6, 3.1);
  \draw[dashed, blue] (2.58, 2.7)--(2.55,2.2);
  
   \draw[dashed, blue] (2.53, 1.6)--(2.5,0.2); 
   
 \draw[red]  (1.2,3.4)--(2,3.4);
  \draw[red  ]  (3.34,3.4)--(6.7,3.4);
    \draw[red]  (8.3,3.4)--(9.1,3.4);

\end{tikzpicture}
    \caption[Meridians and Longitudes]{Meridians (red dashed curves), longitude curves (blue dashed curves), and intercusp geodesics (red solid lines)}
    \label{mu}
\end{figure}
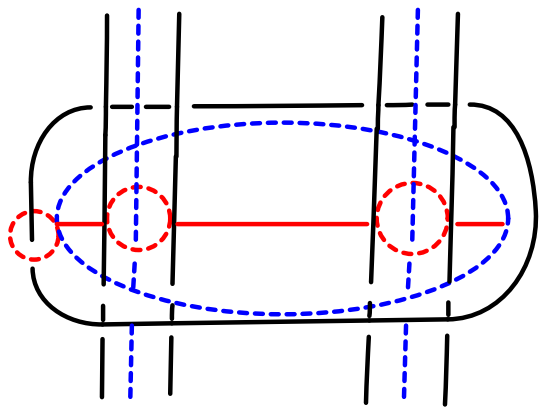
\begin{theorem}\label{km}
Let $L_1$ and $L_2$ be FAL that differ in half-twists and let $M_i = S^3 - L_i$, then $kM_1 = kM_2.$ 
\end{theorem}
\begin{proof}
This follows from Theorem \ref{mainthm}, the cusp shapes for FAL complements differing in half-twists have cusp shapes $4\omega$ and $\frac{4\omega}{1\pm2\omega}$ respectively, generating the same field. 
\end{proof}
FAL complements that differ in half-twists have the same volume and the same invariant trace fields, but are not isometric.
\begin{corollary}
There exists an arbitrarily large set of links with complements having the same volume and same invariant trace fields yet are not isometric. 
\end{corollary}
\begin{figure}
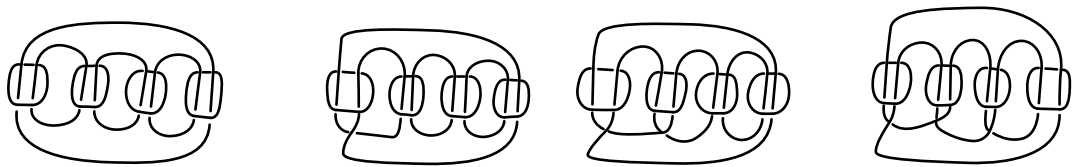

    \centering
    \include{f41}
    \caption{Non-isometric links $FALP_4$}
    \label{f4}
\end{figure}
\begin{proof}
The class of fully augmented pretzel links called $FALP_n$ have number of components ranging from $n+1$ to $2n$ depending on half-twists, see Figure \ref{FALP}(a). $FALP_n$ without half-twists have $2n$ components, for each half-twist added the number of components decrease by 1. $FALP_n$ with $n-1$ half-twists will have $n+1$ components, see Figure \ref{f4}. 
All of these links  have the same volume as they decompose into the same ideal polyhedra, but are obtained by different gluings on the bowties. In addition, they have the same invariant trace field by Theorem 5.7. However, they are non-isometric since they have different number of cusps which is an invariant  of the link complement. 
\end{proof}
\begin{remark}
A very interesting question to study is the commensurability of these links.  What happens to the commensurability of FAL when we add half-twists?
\end{remark}
\subsection{Commensurability of Pretzel FALs}
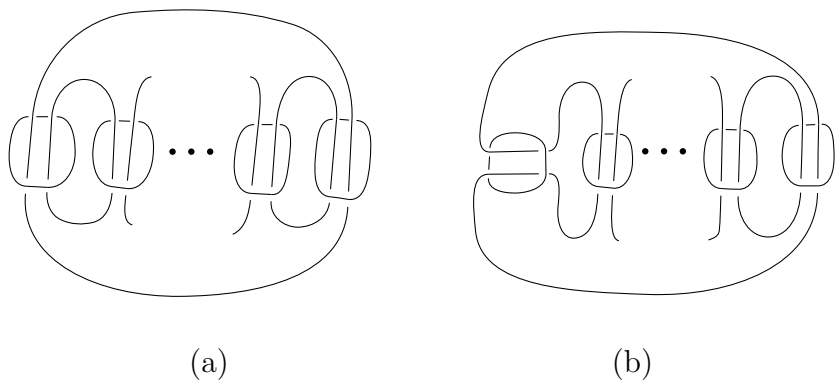
\begin{figure}[h]
  \centering  
    \begin{tikzpicture}[every node/.style={scale=0.6}, scale=.5][line width=3.8, line cap=round, line join=round]
    \draw (0.67, 4.93) .. controls (0.29, 4.93) and (0.19, 4.45) .. 
          (0.16, 4.02) .. controls (0.13, 3.57) and (0.23, 3.10) .. (0.61, 3.08);
    \draw (0.61, 3.08) .. controls (0.80, 3.07) and (0.99, 3.06) .. (1.17, 3.05);
    \draw (1.17, 3.05) .. controls (1.62, 3.03) and (1.91, 3.48) .. 
          (1.91, 3.97) .. controls (1.91, 4.43) and (1.81, 4.93) .. (1.42, 4.93);
    \draw (1.19, 4.93) .. controls (1.09, 4.93) and (0.99, 4.93) .. (0.88, 4.93);
    \draw (2.89, 3.07) .. controls (2.51, 3.11) and (2.37, 3.56) .. 
          (2.38, 3.98) .. controls (2.40, 4.39) and (2.51, 4.83) .. (2.86, 4.82);
    \draw (3.07, 4.82) .. controls (3.17, 4.82) and (3.27, 4.81) .. (3.37, 4.81);
    \draw (3.59, 4.80) .. controls (3.96, 4.79) and (4.03, 4.31) .. 
          (3.98, 3.87) .. controls (3.94, 3.41) and (3.68, 2.97) .. (3.28, 3.02);
    \draw (3.28, 3.02) .. controls (3.15, 3.04) and (3.02, 3.05) .. (2.89, 3.07);
    \draw (6.59, 2.89) .. controls (6.21, 2.91) and (6.10, 3.38) .. 
          (6.15, 3.82) .. controls (6.19, 4.25) and (6.29, 4.72) .. (6.67, 4.72);
    \draw (6.87, 4.72) .. controls (6.97, 4.72) and (7.07, 4.72) .. (7.16, 4.72);
    \draw (7.35, 4.72) .. controls (7.71, 4.72) and (7.65, 4.22) .. 
          (7.60, 3.78) .. controls (7.55, 3.34) and (7.48, 2.83) .. (7.13, 2.86);
    \draw (7.13, 2.86) .. controls (6.95, 2.87) and (6.77, 2.88) .. (6.59, 2.89);
    \draw (8.76, 4.86) .. controls (8.38, 4.88) and (8.34, 4.34) .. 
          (8.30, 3.87) .. controls (8.26, 3.39) and (8.27, 2.85) .. (8.68, 2.80);
    \draw (8.68, 2.80) .. controls (8.85, 2.78) and (9.01, 2.76) .. (9.17, 2.73);
    \draw (9.17, 2.73) .. controls (9.60, 2.68) and (9.71, 3.23) .. 
          (9.75, 3.74) .. controls (9.79, 4.24) and (9.82, 4.81) .. (9.39, 4.83);
    \draw (9.19, 4.84) .. controls (9.11, 4.85) and (9.03, 4.85) .. (8.94, 4.86);
    \draw (2.97, 4.82) .. controls (2.99, 5.37) and (2.69, 5.89) .. 
          (2.19, 5.92) .. controls (1.69, 5.95) and (1.33, 5.47) .. (1.29, 4.93);
    \draw (1.29, 4.93) .. controls (1.26, 4.37) and (1.22, 3.82) .. (1.19, 3.26);
    \draw (1.16, 2.92) .. controls (1.14, 2.62) and (1.26, 2.31) .. 
          (1.53, 2.19) .. controls (1.80, 2.06) and (2.11, 2.05) .. 
          (2.39, 2.12) .. controls (2.72, 2.20) and (2.87, 2.55) .. (2.89, 2.90);
    \draw (2.90, 3.27) .. controls (2.92, 3.79) and (2.95, 4.31) .. (2.97, 4.82);
    \draw (8.86, 4.86) .. controls (8.91, 5.43) and (8.65, 5.99) .. 
          (8.15, 6.00) .. controls (7.58, 6.01) and (7.30, 5.36) .. (7.26, 4.72);
    \draw (7.26, 4.72) .. controls (7.22, 4.17) and (7.18, 3.62) .. (7.14, 3.06);
    \draw (7.12, 2.68) .. controls (7.09, 2.29) and (7.44, 1.99) .. 
          (7.84, 1.99) .. controls (8.25, 2.00) and (8.63, 2.25) .. (8.67, 2.64);
    \draw (8.70, 3.00) .. controls (8.75, 3.62) and (8.81, 4.24) .. (8.86, 4.86);
    \draw (0.78, 4.93) .. controls (0.91, 6.40) and (2.04, 7.59) .. 
          (3.50, 7.71) .. controls (4.91, 7.83) and (6.33, 7.81) .. 
          (7.68, 7.39) .. controls (8.74, 7.06) and (9.33, 5.97) .. (9.27, 4.84);
    \draw (9.27, 4.84) .. controls (9.24, 4.21) and (9.21, 3.57) .. (9.18, 2.94);
    \draw (9.16, 2.53) .. controls (9.07, 0.68) and (6.81, 0.17) .. 
          (4.70, 0.15) .. controls (2.55, 0.13) and (0.43, 1.06) .. (0.59, 2.88);
    \draw (0.63, 3.29) .. controls (0.68, 3.84) and (0.73, 4.38) .. (0.78, 4.93);
    \draw (3.94, 5.99) .. controls (3.59, 5.96) and (3.52, 5.32) .. (3.47, 4.81);
    \draw (3.47, 4.81) .. controls (3.41, 4.28) and (3.35, 3.75) .. (3.30, 3.23);
    \draw (3.26, 2.84) .. controls (3.22, 2.51) and (3.18, 2.12) .. (3.43, 2.05);
    \draw (6.10, 1.80) .. controls (6.43, 1.95) and (6.53, 2.34) .. (6.57, 2.71);
    \draw (6.61, 3.10) .. controls (6.67, 3.64) and (6.72, 4.18) .. (6.78, 4.72);
    \draw (6.78, 4.72) .. controls (6.84, 5.28) and (6.91, 5.96) .. (6.56, 5.99);
\fill [black] (5,4) circle (.085 cm);
\fill [black] (4.5,4) circle (.085 cm);

\fill [black] (5.5,4) circle (.085 cm);
\end{tikzpicture}
\hspace{1cm}
\begin{tikzpicture}[every node/.style={scale=0.6}, scale=.5][line width=3.8, line cap=round, line join=round]
    \draw (2.15, 4.01) .. controls (2.15, 4.33) and (1.78, 4.47) .. 
          (1.42, 4.49) .. controls (1.06, 4.51) and (0.69, 4.40) .. (0.67, 4.09);
    \draw (0.66, 3.86) .. controls (0.65, 3.75) and (0.65, 3.63) .. (0.64, 3.51);
    \draw (0.63, 3.28) .. controls (0.62, 2.97) and (1.01, 2.88) .. 
          (1.39, 2.89) .. controls (1.77, 2.91) and (2.16, 3.06) .. (2.16, 3.40);
    \draw (2.16, 3.40) .. controls (2.16, 3.61) and (2.15, 3.81) .. (2.15, 4.01);
    \draw (3.55, 3.06) .. controls (3.24, 3.09) and (3.11, 3.45) .. 
          (3.14, 3.80) .. controls (3.18, 4.12) and (3.26, 4.48) .. (3.55, 4.47);
    \draw (3.72, 4.46) .. controls (3.80, 4.46) and (3.88, 4.46) .. (3.96, 4.46);
    \draw (4.13, 4.45) .. controls (4.43, 4.44) and (4.47, 4.05) .. 
          (4.46, 3.70) .. controls (4.45, 3.34) and (4.26, 2.98) .. (3.94, 3.02);
    \draw (3.94, 3.02) .. controls (3.81, 3.03) and (3.68, 3.05) .. (3.55, 3.06);
    \draw (6.82, 3.01) .. controls (6.48, 3.02) and (6.36, 3.43) .. 
          (6.36, 3.82) .. controls (6.37, 4.19) and (6.44, 4.60) .. (6.75, 4.59);
    \draw (6.94, 4.58) .. controls (7.04, 4.58) and (7.13, 4.57) .. (7.23, 4.57);
    \draw (7.41, 4.56) .. controls (7.73, 4.54) and (7.77, 4.13) .. 
          (7.77, 3.75) .. controls (7.78, 3.36) and (7.61, 2.97) .. (7.27, 2.99);
    \draw (7.27, 2.99) .. controls (7.12, 2.99) and (6.97, 3.00) .. (6.82, 3.01);
    \draw (8.94, 3.08) .. controls (8.59, 3.09) and (8.44, 3.49) .. 
          (8.44, 3.88) .. controls (8.45, 4.26) and (8.54, 4.67) .. (8.88, 4.69);
    \draw (9.07, 4.70) .. controls (9.15, 4.70) and (9.24, 4.71) .. (9.32, 4.71);
    \draw (9.49, 4.72) .. controls (9.80, 4.73) and (9.81, 4.29) .. 
          (9.82, 3.90) .. controls (9.83, 3.50) and (9.73, 3.07) .. (9.38, 3.08);
    \draw (9.38, 3.08) .. controls (9.23, 3.08) and (9.09, 3.08) .. (8.94, 3.08);
    \draw (9.41, 4.71) .. controls (9.43, 6.49) and (7.29, 7.12) .. 
          (5.25, 7.16) .. controls (3.31, 7.21) and (1.05, 7.26) .. 
          (0.63, 5.62) .. controls (0.44, 4.86) and (0.21, 3.98) .. (0.67, 3.98);
    \draw (0.67, 3.98) .. controls (1.10, 3.99) and (1.53, 4.00) .. (1.96, 4.01);
    \draw (2.23, 4.01) .. controls (2.52, 4.02) and (2.52, 4.51) .. 
          (2.51, 4.92) .. controls (2.51, 5.38) and (2.71, 5.83) .. 
          (3.11, 5.84) .. controls (3.66, 5.85) and (3.68, 5.12) .. (3.64, 4.47);
    \draw (3.64, 4.47) .. controls (3.61, 4.06) and (3.59, 3.66) .. (3.56, 3.25);
    \draw (3.54, 2.87) .. controls (3.51, 2.34) and (3.42, 1.74) .. 
          (2.93, 1.69) .. controls (2.59, 1.66) and (2.44, 2.09) .. 
          (2.48, 2.52) .. controls (2.52, 2.93) and (2.57, 3.41) .. (2.24, 3.40);
    \draw (1.97, 3.40) .. controls (1.53, 3.40) and (1.08, 3.39) .. (0.64, 3.39);
    \draw (0.64, 3.39) .. controls (0.13, 3.39) and (0.22, 2.56) .. 
          (0.29, 1.86) .. controls (0.46, 0.34) and (2.89, 0.26) .. 
          (4.90, 0.19) .. controls (7.10, 0.12) and (9.35, 0.96) .. (9.38, 2.89);
    \draw (9.39, 3.27) .. controls (9.39, 3.75) and (9.40, 4.23) .. (9.41, 4.71);
    \draw (8.98, 4.69) .. controls (9.00, 5.35) and (8.76, 6.01) .. 
          (8.19, 6.00) .. controls (7.57, 6.00) and (7.34, 5.26) .. (7.32, 4.56);
    \draw (7.32, 4.56) .. controls (7.31, 4.10) and (7.29, 3.64) .. (7.28, 3.18);
    \draw (7.27, 2.80) .. controls (7.25, 2.25) and (7.56, 1.72) .. 
          (8.07, 1.72) .. controls (8.61, 1.72) and (8.91, 2.30) .. (8.93, 2.90);
    \draw (8.94, 3.27) .. controls (8.96, 3.75) and (8.97, 4.22) .. (8.98, 4.69);
    \draw (4.44, 5.91) .. controls (4.14, 5.83) and (4.09, 5.03) .. (4.05, 4.45);
    \draw (4.05, 4.45) .. controls (4.02, 4.04) and (3.99, 3.62) .. (3.96, 3.21);
    \draw (3.93, 2.83) .. controls (3.90, 2.33) and (3.85, 1.66) .. (4.10, 1.62);
    \draw (6.54, 5.99) .. controls (6.88, 5.87) and (6.86, 5.14) .. (6.85, 4.58);
    \draw (6.85, 4.58) .. controls (6.84, 4.12) and (6.83, 3.66) .. (6.82, 3.20);
    \draw (6.81, 2.82) .. controls (6.80, 2.49) and (6.79, 1.98) .. 
          (6.76, 1.85) .. controls (6.73, 1.73) and (6.60, 1.68) .. (6.48, 1.64);
 
 \fill [black] (5.3,4) circle (.085 cm);
\fill [black] (4.8,4) circle (.085 cm);

\fill [black] (5.8,4) circle (.085 cm);
\end{tikzpicture}
    (a) \quad \quad \quad \quad \quad \quad \quad \quad\quad \quad \quad \quad (b)
    \caption{(a) $FALP_n$  (b) $FALR_n$}
    \label{FALP}
\end{figure}

We denote the fully augmented link for the  $n$-pretzel link $FALP_n,$ see Figure \ref{FALP}(a). We explore the effects of a $\frac{\pi}{2}$ rotation on the left most crossing circle in a $FALP_n$ for $n\geq3$. Let $FALR_n$ denote the link we obtain from $FALP_n$ by rotating the left most crossing circle by $\pi/2,$ see Figure \ref{FALP}(b).

\subsubsection{$FALP_3$ and $FALR_{3}$}

\begin{figure}[h]
    \centering
    \begin{tikzpicture}[every node/.style={scale=0.8}, scale=.65][line width=3.8, line cap=round, line join=round]
    \draw (3.05, 4.64) .. controls (3.04, 5.03) and (2.55, 5.14) .. 
          (2.09, 5.11) .. controls (1.66, 5.09) and (1.15, 5.06) .. (1.15, 4.71);
    \draw (1.15, 4.45) .. controls (1.15, 4.28) and (1.15, 4.11) .. (1.15, 3.94);
    \draw (1.15, 3.67) .. controls (1.15, 3.28) and (1.67, 3.28) .. 
          (2.12, 3.28) .. controls (2.58, 3.28) and (3.09, 3.36) .. (3.07, 3.75);
    \draw (3.07, 3.75) .. controls (3.07, 4.05) and (3.06, 4.34) .. (3.05, 4.64);

    \draw (5.27, 5.21) .. controls (4.85, 5.22) and (4.71, 4.71) .. 
          (4.69, 4.22) .. controls (4.67, 3.72) and (4.89, 3.23) .. (5.33, 3.24);
    \draw (5.33, 3.24) .. controls (5.54, 3.25) and (5.76, 3.26) .. (5.97, 3.26);
    \draw (5.97, 3.26) .. controls (6.39, 3.27) and (6.57, 3.76) .. 
          (6.55, 4.24) .. controls (6.53, 4.68) and (6.50, 5.19) .. (6.14, 5.20);
    \draw (5.92, 5.20) .. controls (5.79, 5.20) and (5.66, 5.21) .. (5.54, 5.21);
    \draw (8.57, 3.06) .. controls (8.15, 3.06) and (7.91, 3.50) .. 
          (7.91, 3.96) .. controls (7.92, 4.34) and (7.93, 4.86) .. 
          (8.04, 5.06) .. controls (8.13, 5.22) and (8.34, 5.24) .. (8.54, 5.23);
    \draw (8.76, 5.23) .. controls (8.88, 5.22) and (9.01, 5.22) .. (9.13, 5.22);
    \draw (9.36, 5.21) .. controls (9.79, 5.19) and (9.78, 4.63) .. 
          (9.74, 4.13) .. controls (9.71, 3.64) and (9.68, 3.06) .. (9.28, 3.06);
    \draw (9.28, 3.06) .. controls (9.04, 3.06) and (8.81, 3.06) .. (8.57, 3.06);
    \draw (9.25, 5.21) .. controls (9.23, 7.25) and (7.01, 8.35) .. 
          (4.76, 8.36) .. controls (2.67, 8.38) and (0.30, 8.29) .. 
          (0.34, 6.50) .. controls (0.36, 5.63) and (0.39, 4.62) .. (1.15, 4.62);
    \draw (1.15, 4.62) .. controls (1.70, 4.63) and (2.25, 4.63) .. (2.80, 4.64);
    \draw (3.22, 4.64) .. controls (3.69, 4.64) and (3.94, 5.15) .. 
          (3.98, 5.67) .. controls (4.01, 6.20) and (4.27, 6.71) .. 
          (4.76, 6.74) .. controls (5.37, 6.77) and (5.43, 5.95) .. (5.41, 5.21);
    \draw (5.41, 5.21) .. controls (5.38, 4.64) and (5.36, 4.06) .. (5.34, 3.49);
    \draw (5.32, 3.00) .. controls (5.30, 2.43) and (5.19, 1.80) .. 
          (4.69, 1.78) .. controls (4.26, 1.76) and (4.08, 2.26) .. 
          (4.06, 2.75) .. controls (4.04, 3.26) and (3.74, 3.75) .. (3.26, 3.75);
    \draw (2.83, 3.75) .. controls (2.27, 3.76) and (1.71, 3.76) .. (1.15, 3.76);
    \draw (1.15, 3.76) .. controls (0.40, 3.77) and (0.20, 2.84) .. 
          (0.19, 1.97) .. controls (0.16, 0.16) and (2.61, 0.17) .. 
          (4.74, 0.18) .. controls (6.98, 0.18) and (9.31, 0.89) .. (9.28, 2.82);
    \draw (9.28, 3.31) .. controls (9.27, 3.94) and (9.26, 4.58) .. (9.25, 5.21);
    \draw (8.64, 5.23) .. controls (8.67, 6.07) and (8.17, 6.85) .. 
          (7.41, 6.85) .. controls (6.61, 6.84) and (6.09, 6.05) .. (6.05, 5.20);
    \draw (6.05, 5.20) .. controls (6.03, 4.64) and (6.00, 4.07) .. (5.98, 3.51);
    \draw (5.96, 3.02) .. controls (5.93, 2.31) and (6.50, 1.73) .. 
          (7.22, 1.74) .. controls (7.90, 1.74) and (8.54, 2.17) .. (8.56, 2.82);
    \draw (8.58, 3.31) .. controls (8.60, 3.95) and (8.62, 4.59) .. (8.64, 5.23);
  \draw[red, very thick] (.8,3.8) -- (.8,4.6); 
 \draw[red, very thick] (5.3,2.75) -- (6,2.75); 
   \draw[red, very thick] (8.55,2.75) -- (9.25,2.75); 
    
  \draw[red, very thick] (3.45,3.8) -- (3.45,4.7); 
  \draw[red, very thick] (5.4,5.6) -- (6.1,5.6);
  \draw[red, very thick] (8.6,5.6) -- (9.25,5.6);

  \node[thick, black] at (5,.2) {\large\textgreater};
  
  \node[thick, black] at (2,3.3) {\large\textless};
   \node[thick, black] at (2,5.15) {\large\textless};
   \node[thick, black] at (7.3,1.75) {\large\textgreater};
   
   
    \node[thick, black] at (4.7,4.2) {\large$\wedge$};
   
   \node[thick, black] at (6.6,4.2) {\large$\wedge$};
   \node[thick, black] at (7.9,4.2) {\large$\wedge$};
   \node[thick, black] at (9.78,4.2) {\large$\wedge$};

   \node[black] at (4.45,4.2) {$1$};
   \node[black] at (6.8,4.2) {$1$};
   \node[black] at (7.65,4.2) {$1$};
   \node[black] at (2.2,5.4) {$1$};
  
   \node[black] at (2.2,3) {$1$};
 
  \node at (8.8,6.2) {$-\frac{1}{4}$}; 
   \node at (8.8,2.3) {$-\frac{1}{4}$};  
     \node at (3.7,4.2) {$\frac{1}{4}$};
    
     \node at (5.7,2.3) {$-\frac{1}{4}$};
     \node at (5.7,6.1) {$-\frac{1}{4}$};

   \node[black] at (4.7,1.5) {$u_4$}; 
   \node[black] at (4.7,2.15) {$u_4$};
   
   \node[black] at (4.6,6.4) {$u_2$}; 
    \node[black] at (4.6,7) {$u_2$}; 
    \node[black] at (7.3,7.2) {$u_3$};  
    
    \node[black] at (7.3,6.5) {$u_3$};

   \node[black] at (5,8) {$u_1$};
  \node[black] at (5,7.5) {$\aleph$};
  
  \node[black] at (7.4,2.15) {$u_5$};
   \node[black] at (7.4,1.4) {$u_5$};
   \node[black] at (5,.5) {$u_6$};

 \node[black] at (5.5,1.2) {$\beth$};
   \node [black] at  (7.2,4.6) {$\daleth$};
    
  \node[black] at (4.2,4.6) {$\gimel$};

   \node[black] at (1,3.2) {$\omega_1$};
   \node[black] at (1,5.2) {$-\omega_1$};
   \node[black] at (3,3.2) {$\omega_1$};
   \node[black] at (3,5.2) {$-\omega_1$};
   \node[black] at (5,3) {$\omega_2$};
  \node[black] at (6.5,3) {$\omega_2$};
  \node[black] at (6.5,5.4) {$\omega_2$};
 \node[black] at (5,5.4) {$\omega_2$}; 
 \node[black] at (8.2,2.8) {$\omega_3$}; 
\node[black] at (9.7,2.8) {$\omega_3$};  
 
 \node[black] at (8.2,5.4) {$\omega_3$}; 
\node[black] at (9.7,5.4) {$\omega_3$};
 
\end{tikzpicture}
    \caption{$FALR_3$}
    \label{FALR_3}
\end{figure}

We have studied $FALP_3$ in detail in \S 4. We now compute the T-T equations for $FALR_3$. \\
Region $\aleph$:  We have shape parameters: 

$$\xi_1=\frac{-\omega_1}{u_{1}}, \quad \xi_2=\frac{-\frac{1}{4}}{u_{1}u_3},
\quad \xi_3=\frac{-\frac{1}{4}}{u_{3}u_2}, \quad \xi_4=\frac{-\omega_1}{u_{2}}$$ This is a four-sided region with equations:
$$\frac{-\omega_1}{u_{1}} - \frac{\frac{1}{4}}{u_{1}u_3} = 1 , \quad \frac{-\frac{1}{4}}{u_{1}u_3} - \frac{\frac{1}{4}}{u_{3}u_2} = 1,  \quad \frac{-\frac{1}{4}}{u_{3}u_2} - \frac{\omega_1}{u_{2}} = 1,  \quad \frac{-\omega_1}{u_{2}} - \frac{\omega_1}{u_{1}} = 1$$ solving gives us the relations
$$u_{2} = u_{1}, \quad -2\omega_1 = u_2, \quad \textrm{and} \quad u_{3} = -\frac{1}{2u_2}.$$
Region $\gimel$:
We have a three-sided region with shape parameters 
$$\xi_1=\frac{\omega_2}{u_2}=1, \quad \xi_2=\frac{\omega_2}{u_4}=1,
\quad \xi_3=\frac{-\frac{1}{4}}{u_2u_4}=1 \quad \implies
\quad u_2 = u_4 = \omega_2$$ \quad \quad \quad and $u_2^2=-\frac{1}{4} \quad  
\implies \quad 
u_2= \pm\frac{i}{2}.$\\
$FALR_3$ has invariant trace field $x^2+\frac{1}{4}.$ It can be checked using snap \cite{Snap} that both $FALP_3$ and $FALR_3$ are arithmetic. Since
they have the same invariant trace field, they are commensurable. Note that they have the same volume yet they are not isometric links as they don't have the same number of components.

\begin{figure}[h]
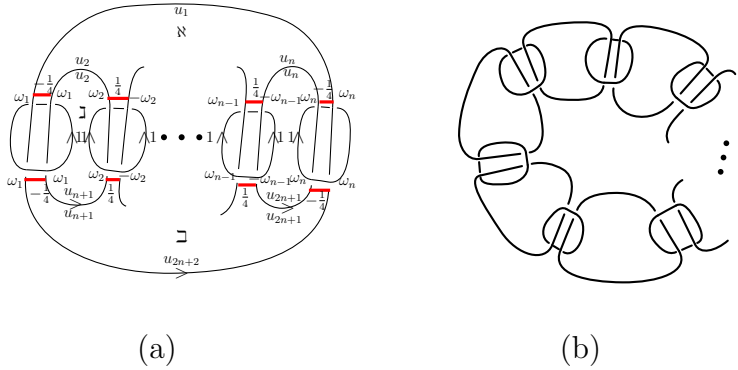

    \centering
   \include{nfal}
   (a) \quad \quad \quad \quad \quad \quad \quad \quad\quad \quad \quad \quad (b)
    \caption[$FALP_n$ with labels]{(a) $FALP_n$ with labels (b) Symmetric diagram of $FALP_n$}
    \label{FALPP_n}
\end{figure}

\subsubsection{T-T Polynomial for $FALP_n$ and $FALR_n$}
In this section we find a recurrence relation for the T-T polynomial for $FALP_n$ and $FALR_n$. 
\begin{theorem}
Let $\aleph$ be the region in $FALP_n$ denoted in Figure \ref{FALPP_n}(a). Let $C_n(x)$ be the (2,1) entry of the matrix equation in Proposition \ref{matrixprop}, where $x=u_2$ is the edge parameter
as shown in Figure 31(a). The T-T polynomial for $FALP_n$ is $C_n(x)$, which satisfies the recurrence relation 
 $$C_n(x) = \frac{C_{n-2}(x)}{4} + x C_{n-1}(x)$$ for $n \geq 5$ where $C_3(x) = x^2 +1/4, \quad C_4(x) = \frac{x(2x^2+1)}{2}$.
\end{theorem}
\begin{proof}
From the symmetries in these links (see Figure \ref{FALPP_n}(b)) and the shape parameter equations for the four-sided regions we have \begin{enumerate}
    \item $-\omega_1 = \omega_2 = ... = \omega_{n}$
    \item $u_2 = u_3 = ... = u_{n}$
    \item and $-2\omega_i = u_j$ where $i,j \neq 1$ or $2{n}+2$.

\end{enumerate}
For simplicity let $u_2 = x$ and $u_1 = z$.

The smallest $FALP_n$ is when $n=3$. Let $n=3$ the matrix equation for Region $\aleph$ is:
$$\begin{bmatrix}
0&-\frac{1}{4}
\\1&0 
\end{bmatrix} \begin{bmatrix}
1&x
\\0&1 
\end{bmatrix} \begin{bmatrix}
0&\frac{1}{4}
\\1&0 
\end{bmatrix}  \begin{bmatrix}
1&x
\\0&1 
\end{bmatrix} \begin{bmatrix}
0&-\frac{1}{4}
\\1&0 
\end{bmatrix}
\begin{bmatrix}
1&-z
\\0&1 
\end{bmatrix} = \begin{bmatrix}
\frac{-x}{4}&\frac{1}{16}+\frac{xz}{4}
\\x^2+\frac{1}{4}&-x^2z-\frac{x}{4}-\frac{z}{4}
\end{bmatrix}$$ thus

$C_3(x) = x^2+\frac{1}{4}.$

For $n=4$ $FALP_4$
the matrix equation for Region $\aleph$ is:
$$\begin{bmatrix}
0&-\frac{1}{4}
\\1&0 
\end{bmatrix} \begin{bmatrix}
1&x
\\0&1 
\end{bmatrix}
\begin{bmatrix}
0&\frac{1}{4}
\\1&0 
\end{bmatrix} \begin{bmatrix}
1&x
\\0&1 
\end{bmatrix}
\begin{bmatrix}
0&\frac{1}{4}
\\1&0 
\end{bmatrix}  \begin{bmatrix}
1&x
\\0&1 
\end{bmatrix} \begin{bmatrix}
0&-\frac{1}{4}
\\1&0 
\end{bmatrix}
\begin{bmatrix}
1&-z
\\0&1 
\end{bmatrix} =$$ $$\begin{bmatrix}
-(\frac{x^2}{4}+\frac{1}{16})&\frac{x^2z}{4}+\frac{x}{16}+\frac{z}{16}
\\ \frac{x(2x^2+1)}{2}&-x^3z-\frac{x^2}{4}-\frac{xz}{2}-\frac{1}{16}
\end{bmatrix}$$ thus

$C_4(x) =\frac{x(2x^2+1)}{2}$ 

For $FALP_{n-2}$ Region $\aleph$ is a $(n-2)$-sided region with matrix equation 
$$\begin{bmatrix}
0&-\frac{1}{4}
\\1&0 
\end{bmatrix} \begin{bmatrix}
1&x
\\0&1 
\end{bmatrix}\Big(\begin{bmatrix}
0&\frac{1}{4}
\\1&0 
\end{bmatrix} \begin{bmatrix}
1&x
\\0&1 
\end{bmatrix}\Big)^{n-4} \begin{bmatrix}
0&-\frac{1}{4}
\\1&0 
\end{bmatrix}  \begin{bmatrix}
1&-z
\\0&1 
\end{bmatrix}$$ $$= \begin{bmatrix}
A_{n-2}&B_{n-2}
\\C_{n-2}&D_{n-2} 
\end{bmatrix} = \alpha\begin{bmatrix}
1&0
\\0&1 
\end{bmatrix}$$
Now for $FALP_{n-1}$ Region $\aleph$ is an $(n-1)$-sided region with matrix equation 
 $$\begin{bmatrix}
0&-\frac{1}{4}
\\1&0 
\end{bmatrix} \begin{bmatrix}
1&x
\\0&1 
\end{bmatrix}
\begin{bmatrix}
0&\frac{1}{4}
\\1&0 
\end{bmatrix} \begin{bmatrix}
0&1
\\-4&0 
\end{bmatrix}
\begin{bmatrix}
A_{n-2}&B_{n-2}
\\C_{n-2}&D_{n-2} 
\end{bmatrix} = 
\begin{bmatrix}
0&-\frac{1}{4}
\\-1&x 
\end{bmatrix}\begin{bmatrix}
A_{n-2}&B_{n-2}
\\C_{n-2}&D_{n-2} 
\end{bmatrix}$$
$$=\begin{bmatrix}
A_{n-1}&B_{n-1}
\\C_{n-1}&D_{n-1} 
\end{bmatrix} =
\begin{bmatrix}
-\frac{C_{n-2}}{4}&-\frac{D_{n-2}}{4}
\\-A_{n-2}+xC_{n-2}&-B_{n-2}+xD_{n-2} 
\end{bmatrix}$$
Now for $FALP_{n}$ Region $\aleph$ is an $n$-sided region with matrix equation 
$$
\begin{bmatrix}
0&-\frac{1}{4}
\\-1&x 
\end{bmatrix}\begin{bmatrix}
A_{n-1}&B_{n-1}
\\C_{n-1}&D_{n-1} 
\end{bmatrix}=\begin{bmatrix}
A_{n}&B_{n}
\\C_{n}&D_{n} 
\end{bmatrix} =$$

$$\begin{bmatrix}
-\frac{C_{n-1}}{4}&-\frac{D_{n-1}}{4}
\\-A_{n-1}+xC_{n-1}&-B_{n-1}+xD_{n-1} 
\end{bmatrix}$$

Where $A_{n-1} = -\frac{Cn-2}{4}$, thus $C_n(x) = \frac{C_{n-2}}{4} + xC_{n-1}.$
\end{proof}

In Table \ref{table-pretzel} below we compute $C_n(x)$ for some values
of $n$. We list the factors of $C_n(x)$. The factor in bold
corresponds to the invariant trace field, which is listed in the last
column. We checked using pari-gp that every root of this factor lies
in the invariant trace field.

\begin{prop} It follows from the recurrence that the
degree of $C_n(x) = n-1.$ 
\end{prop}
\begin{conjecture}\label{conj}

$C_n(x)$ satisfies the following conditions: 
\begin{enumerate}
    \item  If $n$ is prime then $C_n(x)$ is irreducible.
\item $C_m(x) | C_n(x)$ if and only if $m|n$.
\end{enumerate}
\end{conjecture}

\begin{table}
  \begin{longtable}{|l||p{3in}|p{2.5in}|}
  \hline
  C & T-T Polynomial & Invariant trace field \\
 \hline 
3 & $(4x^2 + 1)/4$ & $x^2 + 1$ \\
4 & $x(2x^2+1)/2$ & $x^2 + 2$\\
5 & $(16x^4+12x^2+1)/16$ & $x^4 + 3x^2 + 1$ \\
6 & $x (4x^2 + 1) \bm{(4x^2 + 3)} /16 $& $x^2 - x + 1$\\
7 & $(64x^6 + 80x^4 + 24x^2 + 1)/64$ & $x^6 + 5x^4 + 6x^2 + 1 $\\
8 & $x (2x^2 + 1) \bm{(8x^4 + 8x^2 + 1)}/16$ & $x^4 + 4x^2 + 2$\\
9 & $(4x^2 + 1) \bm{(64x^6 + 96x^4 + 36x^2 + 1)}/256$  & $x^6 + 6x^4 + 9x^2 + 1$\\
10 &$ x (16x^4 + 12x^2 + 1) \bm{(16x^4 + 20x^2 + 5)} /256$ & $x^4 - x^3 + x^2 - x + 1$\\
11 & $(1024x^{10} + 2304x^8 + 1792x^6 + 560x^4 + 60x^2 + 1)/1024$ & $x^{10} + 9x^8 + 28x^6 + 35x^4 + 15x^2 + 1$ \\
12 & $x (2x^2+1)(4x^2+1) (4x^2+3)$ $\bm{(16x^4 + 16x^2 + 1)}/512$ & $x^4 + 4x^2 + 1$\\
13 & $(4096x^{12} + 11264x^{10} + 11520x^8 + 5376x^6 + 1120x^4 + 84x^2 + 1)/4096$ & $x^{12} + 11x^{10} + 45x^8 + 84x^6 + 70x^4 + 21x^2 + 1$ \\
14 &  $x(64x^6 + 80x^4 + 24x^2 + 1) \bm{(64x^6 + 112x^4 + 56x^2 + 7)}/4096$ & $x^6 - x^5 + x^4 - x^3 + x^2 - x + 1$ \\
15 & $(4 x^2 + 1) (16 x^4 + 12x^2 + 1) \bm{(256x^8 + 576x^6 + 416x^4 + 96x^2 + 1)}/16384$ 
& $x^8 + 9x^6 + 26x^4 + 24x^2 + 1$ \\
16 & $ x (2x^2 + 1) (8x^4 + 8x^2 + 1) \bm{(128x^8 + 256x^6 + 160x^4 + 32x^2 + 1)}/2048 $ & $x^8 + 8x^6 + 20x^4 + 16x^2 + 2$\\
17 & $(65536x^{16} + 245760x^{14} + 372736x^{12} + 292864x^{10} + 126720x^8 + 29568x^6 + 3360x^4 + 144x^2 + 1 )/65536$
& $x^{16} + 15x^{14} + 91x^{12} + 286x^{10} + 495x^8 + 462x^6 + 210x^4 + 36x^2 + 1$\\
\hline
\end{longtable} 
\caption{T-T polynomial and invariant trace field for $FALP_n$ }
\label{table-pretzel}
\end{table}

\begin{figure}[h]
    \centering
   \begin{tikzpicture}[every node/.style={scale=0.5}, scale=.5][line width=3.8, line cap=round, line join=round]
    \draw (2.15, 4.01) .. controls (2.15, 4.33) and (1.78, 4.47) .. 
          (1.42, 4.49) .. controls (1.06, 4.51) and (0.69, 4.40) .. (0.67, 4.09);
    \draw (0.66, 3.86) .. controls (0.65, 3.75) and (0.65, 3.63) .. (0.64, 3.51);
    \draw (0.63, 3.28) .. controls (0.62, 2.97) and (1.01, 2.88) .. 
          (1.39, 2.89) .. controls (1.77, 2.91) and (2.16, 3.06) .. (2.16, 3.40);
    \draw (2.16, 3.40) .. controls (2.16, 3.61) and (2.15, 3.81) .. (2.15, 4.01);
    \draw (3.55, 3.06) .. controls (3.24, 3.09) and (3.11, 3.45) .. 
          (3.14, 3.80) .. controls (3.18, 4.12) and (3.26, 4.48) .. (3.55, 4.47);
    \draw (3.72, 4.46) .. controls (3.80, 4.46) and (3.88, 4.46) .. (3.96, 4.46);
    \draw (4.13, 4.45) .. controls (4.43, 4.44) and (4.47, 4.05) .. 
          (4.46, 3.70) .. controls (4.45, 3.34) and (4.26, 2.98) .. (3.94, 3.02);
    \draw (3.94, 3.02) .. controls (3.81, 3.03) and (3.68, 3.05) .. (3.55, 3.06);
    \draw (6.82, 3.01) .. controls (6.48, 3.02) and (6.36, 3.43) .. 
          (6.36, 3.82) .. controls (6.37, 4.19) and (6.44, 4.60) .. (6.75, 4.59);
    \draw (6.94, 4.58) .. controls (7.04, 4.58) and (7.13, 4.57) .. (7.23, 4.57);
    \draw (7.41, 4.56) .. controls (7.73, 4.54) and (7.77, 4.13) .. 
          (7.77, 3.75) .. controls (7.78, 3.36) and (7.61, 2.97) .. (7.27, 2.99);
    \draw (7.27, 2.99) .. controls (7.12, 2.99) and (6.97, 3.00) .. (6.82, 3.01);
    \draw (8.94, 3.08) .. controls (8.59, 3.09) and (8.44, 3.49) .. 
          (8.44, 3.88) .. controls (8.45, 4.26) and (8.54, 4.67) .. (8.88, 4.69);
    \draw (9.07, 4.70) .. controls (9.15, 4.70) and (9.24, 4.71) .. (9.32, 4.71);
    \draw (9.49, 4.72) .. controls (9.80, 4.73) and (9.81, 4.29) .. 
          (9.82, 3.90) .. controls (9.83, 3.50) and (9.73, 3.07) .. (9.38, 3.08);
    \draw (9.38, 3.08) .. controls (9.23, 3.08) and (9.09, 3.08) .. (8.94, 3.08);
    \draw (9.41, 4.71) .. controls (9.43, 6.49) and (7.29, 7.12) .. 
          (5.25, 7.16) .. controls (3.31, 7.21) and (1.05, 7.26) .. 
          (0.63, 5.62) .. controls (0.44, 4.86) and (0.21, 3.98) .. (0.67, 3.98);
    \draw (0.67, 3.98) .. controls (1.10, 3.99) and (1.53, 4.00) .. (1.96, 4.01);
    \draw (2.23, 4.01) .. controls (2.52, 4.02) and (2.52, 4.51) .. 
          (2.51, 4.92) .. controls (2.51, 5.38) and (2.71, 5.83) .. 
          (3.11, 5.84) .. controls (3.66, 5.85) and (3.68, 5.12) .. (3.64, 4.47);
    \draw (3.64, 4.47) .. controls (3.61, 4.06) and (3.59, 3.66) .. (3.56, 3.25);
    \draw (3.54, 2.87) .. controls (3.51, 2.34) and (3.42, 1.74) .. 
          (2.93, 1.69) .. controls (2.59, 1.66) and (2.44, 2.09) .. 
          (2.48, 2.52) .. controls (2.52, 2.93) and (2.57, 3.41) .. (2.24, 3.40);
    \draw (1.97, 3.40) .. controls (1.53, 3.40) and (1.08, 3.39) .. (0.64, 3.39);
    \draw (0.64, 3.39) .. controls (0.13, 3.39) and (0.22, 2.56) .. 
          (0.29, 1.86) .. controls (0.46, 0.34) and (2.89, 0.26) .. 
          (4.90, 0.19) .. controls (7.10, 0.12) and (9.35, 0.96) .. (9.38, 2.89);
    \draw (9.39, 3.27) .. controls (9.39, 3.75) and (9.40, 4.23) .. (9.41, 4.71);
    \draw (8.98, 4.69) .. controls (9.00, 5.35) and (8.76, 6.01) .. 
          (8.19, 6.00) .. controls (7.57, 6.00) and (7.34, 5.26) .. (7.32, 4.56);
    \draw (7.32, 4.56) .. controls (7.31, 4.10) and (7.29, 3.64) .. (7.28, 3.18);
    \draw (7.27, 2.80) .. controls (7.25, 2.25) and (7.56, 1.72) .. 
          (8.07, 1.72) .. controls (8.61, 1.72) and (8.91, 2.30) .. (8.93, 2.90);
    \draw (8.94, 3.27) .. controls (8.96, 3.75) and (8.97, 4.22) .. (8.98, 4.69);
    \draw (4.44, 5.91) .. controls (4.14, 5.83) and (4.09, 5.03) .. (4.05, 4.45);
    \draw (4.05, 4.45) .. controls (4.02, 4.04) and (3.99, 3.62) .. (3.96, 3.21);
    \draw (3.93, 2.83) .. controls (3.90, 2.33) and (3.85, 1.66) .. (4.10, 1.62);
    \draw (6.54, 5.99) .. controls (6.88, 5.87) and (6.86, 5.14) .. (6.85, 4.58);
    \draw (6.85, 4.58) .. controls (6.84, 4.12) and (6.83, 3.66) .. (6.82, 3.20);
    \draw (6.81, 2.82) .. controls (6.80, 2.49) and (6.79, 1.98) .. 
          (6.76, 1.85) .. controls (6.73, 1.73) and (6.60, 1.68) .. (6.48, 1.64);
 
 \fill [black] (5.3,4) circle (.085 cm);
\fill [black] (4.8,4) circle (.085 cm);

\fill [black] (5.8,4) circle (.085 cm);

\draw[red, very thick] (2.35,3.4) -- (2.35,4); 
\draw[red, very thick] (3.5,2.8) -- (3.95,2.8);

 \draw[red, very thick] (3.65,4.8) -- (4.05,4.8);  
 
 \draw[red, very thick] (6.8,2.65) -- (7.25,2.65); 
 
 \draw[red, very thick] (6.9,4.8) -- (7.3,4.8);
  \draw[red, very thick] (8.89,2.7) -- (9.35,2.7); 
  \draw[red, very thick] (9,4.85) -- (9.4,4.85);

  \node[black] at (3,4.95) {$\gimel$};
  
  \node[thick, black] at (5,.21) {\large\textgreater};
  
   \node[thick, black] at (8,1.75) {\large\textgreater};

   
    \node[thick, black] at (6.35,4) {\large$\wedge$};
   
   \node[thick, black] at (4.45,4) {\large$\wedge$};
   
      \node[thick, black] at (3.15,4) {\large$\wedge$};
   
   \node[thick, black] at (1.4,4.5) {\large\textless};
   \node[thick, black] at (1.4,2.85) {\large\textless};

   \node[thick, black] at (7.8,4) {\large$\wedge$};
    \node[thick, black] at (8.45,4) {\large$\wedge$};

   
     \node at (2.5,3.75) {$\frac{1}{4}$};
     \node at (3.82,5.2) {$-\frac{1}{4}$};
     \node at (3.6,2.5) {$-\frac{1}{4}$};

      \node at (7.1,2.3) {$\frac{1}{4}$};
      \node at (7.1,5.1) {$\frac{1}{4}$};
      \node at (9,2.3) {$-\frac{1}{4}$};
      \node at (9.1,5.1) {$-\frac{1}{4}$};

   \node[black] at (3,6) {$u_2$}; 
    \node[black] at (3,5.6) {$u_2$}; 
    \node[black] at (2.8,1.5) {$u_{n+1}$}; 
     \node[black] at (2.9,2) {$u_{n+1}$};
    \node[black] at (8,6.2) {$u_n$};  
    
    \node[black] at (8.1,5.8) {$u_n$};  
        \node[black] at (8,1.5) {$u_{2n+1}$}; 
     \node[black] at (8,2.1) {$u_{2n+1}$};

   \node[black] at (5,6.9) {$u_1$};
  \node[black] at (5,6.4) {$\aleph$};

   \node[black] at (5,.5) {$u_{2n+2}$};
    \node[black] at (8.3,4) {$1$};
   \node[black] at (8,4) {$1$};
   \node[black] at (6.12,4) {$1$};
  
   \node[black] at (3,4) {$1$};
    \node[black] at (4.6,4) {$1$};
   \node[black] at (1.45,4.8) {$1$};
  \node[black] at (1.45,2.5) {$1$};

   \node[black] at (5,1.2) {$\beth$};
 
   \node[black] at (.5,2.9) {$\omega_1$};
   \node[black] at (.65,4.5) {$-\omega_1$};
   \node[black] at (2.1,2.9) {$\omega_1$};
   \node[black] at (2.1,4.5) {$-\omega_1$};
   \node[black] at (3.2,2.9) {$\omega_2$};
  \node[black] at (4.4,2.9) {$\omega_2$};
  \node[black] at (4.4,4.7) {$\omega_2$};
 \node[black] at (3.3,4.7) {$\omega_2$}; 
 \node[black] at (6.3,2.9) {$\omega_{n-1}$}; 
 \node[black] at (6.4,4.75) {$\omega_{n-1}$}; 
 \node[black] at (7.7,2.7) {$-\omega_{n-1}$}; 
 \node[black] at (7.85,4.75) {$-\omega_{n-1}$}; 
 \node[black] at (8.7,2.8) {$\omega_{n}$}; 
 \node[black] at (8.6,4.75) {$\omega_{n}$}; 
\node[black] at (9.7,2.8) {$\omega_{n}$}; 
\node[black] at (9.7,4.9) {$\omega_{n}$};

\end{tikzpicture}
    \caption{$FALR_n$ with labels}
    \label{FALR_n}
\end{figure}
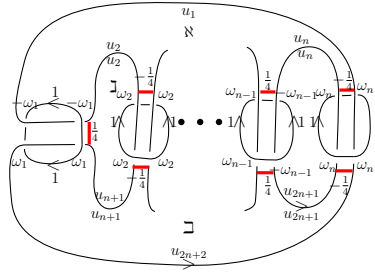

Computations for $FALR_n$: 

For Region $\gimel$ we have a 3-sided region
$$\xi_1=\frac{\omega_2}{u_2}=1, \quad \xi_2=\frac{\omega_2}{u_{n+1}}=1,
\quad \xi_3=\frac{-\frac{1}{4}}{u_2u_{n+1}}=1,$$ $\Longrightarrow$
$$u_2 = u_{n+1} = \omega_2 = \pm\frac{i}{2}.$$
 From the similarities in the 4-sided regions we get the following equations
$$u_3 = . . . = u_n, \quad \omega_2 = . . . = \omega_n, \quad u_3 = 2\omega_2.$$
Without loss of generality let $$\omega_2 = \frac{i}{2}, \quad  \omega_1 = x, \quad u_1 = z.$$
Then Region $\aleph$ is a $(n+1)$-sided region with matrices equation
$$
\begin{bmatrix}
0&-\frac{1}{4}
\\1&0 
\end{bmatrix} \begin{bmatrix}
1&i
\\0&1 
\end{bmatrix}
\Big(\begin{bmatrix}
0&\frac{1}{4}
\\1&0 
\end{bmatrix} \begin{bmatrix}
1&i
\\0&1 
\end{bmatrix}\Big)^{n-3} \begin{bmatrix}
0&-\frac{1}{4}
\\1&0 
\end{bmatrix}  \begin{bmatrix}
1&-\frac{i}{2}
\\0&1 
\end{bmatrix}\begin{bmatrix}
0&-x
\\1&0 
\end{bmatrix} \begin{bmatrix}
1&1
\\0&1 
\end{bmatrix}\begin{bmatrix}
0&-x
\\1&0 
\end{bmatrix}$$  $$\begin{bmatrix}
1&-z
\\0&1 
\end{bmatrix} = \alpha\begin{bmatrix}
1&0
\\0&1 
\end{bmatrix}.$$

The (2,1)-entry will give you a solution in $\mathbb{Q}(i)$ for all $n.$ 
All the cusps other than the crossing circle that is rotated have equal cusp shapes of $2i.$ To find the cusp shape for the cusp of the rotated crossing circle, solve the above equation and multiply the solution by $4.$ 
\begin{theorem}\label{com1}
Let $m ,n$ $\geq 3$, then $FALR_m$ and $FALR_n$ are commensurable. \end{theorem}
\begin{proof}
From Figure \ref{fbr1'} we can see that $FALR_n$ is a $(n-1)$-sheeted cover of the Borromean rings. Hence $FALR_n$ is commensurable with the Borromean rings $FALR_3$. Since commensurability is an equivalence relation, $FALR_m$, which is a $(m-1)$-sheeted cover of the Borromean rings is commensurable with $FALR_n$. 
\end{proof}
\begin{figure}[h]
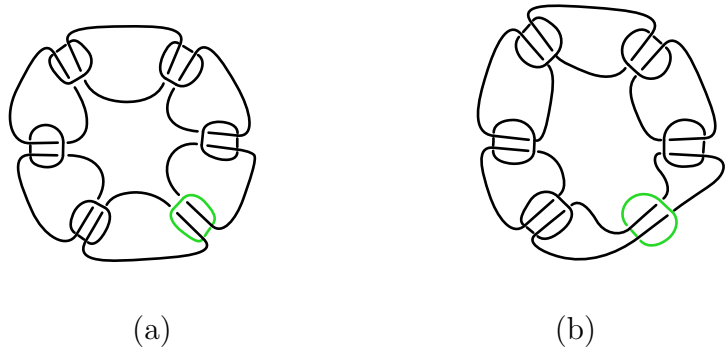

    \centering
    \include{fal6}
      (a) \quad \quad \quad \quad \quad \quad \quad \quad\quad \quad \quad \quad (b)
    \caption[$FALR_6$]{(a) Symmetric diagram of $FALP_6$ (b) $FALR_6$ with rotated crossing circle green}
    \label{fal6}
\end{figure}
\begin{figure}[h]
    \centering
    \include{br1'}
    \caption[BR covers]{$FALR_6$ where green dot is vertical axis viewed from $\infty$, Borromean Ring with green crossing circle viewed from $\infty.$}
    \label{fbr1'}
\end{figure}
\begin{corollary}

For $n \geq 3$, the invariant trace field for $FALR_n$ is $\mathbb{Q}(i)$. 
\end{corollary}
\begin{conjecture}\label{com2}

\begin{enumerate}
    \item $FALP_n$ and $FALP_m$ are incommensurable for $n \neq m.$
     \item 
   For $n\geq 4,$ $FALP_n$ and $FALR_n$ are incommensurable for all $n.$\end{enumerate} \end{conjecture}

\begin{remark}
For $n \neq m$, $FALP_n$ has different invariant trace field than $FALP_m$ by Conjecture \ref{conj}.  For the second part, by Conjecture \ref{conj} $FALP_n$ for $n > 3$ have invariant trace fields $\neq$ $\mathbb{Q}(i)$, while the invariant trace field for all $FALR_n$ is $\mathbb{Q}(i)$, thus they are incommensurable.
\end{remark}
\subsection{Geometric Solutions to T-T Equations}
The T-T method gives us a way to construct algebraic equations in variables, which then gives us a representation $\rho_i$ in $PSL(2,\mathbb{C})$ with matrix entries in terms of root $x_i$ of the T-T polynomial.
Moreover, there exists a root $x_0$ of the T-T polynomial such that $\rho_{x_0}$ is discrete and faithful, this solution will be the geometric solution.

For FALs all the non-real complex solutions lie in $k\Gamma$ thus to find the geometric solution we use Theorem \ref{mainthm}, which states that $4\omega$ will give us the cusp shape.  We can therefore work backwards, use SnapPy to compute the cusp shape and see which solutions give us the cusp shape, thereby giving us the geometric solution. 

\begin{theorem}
\label{geometric}

Let $L$ be a FAL, then
the solution of
 the T-T polynomials which corresponds
to the
cusp shape is the geometric solution.

\end{theorem}
\begin{proof}
The T-T polynomial for FALs can be written in a variable which is an edge parameter for a crossing circle, which is related to the cusp shape
of that crossing circle. Due to Mostow-Prasad Rigidity the geometric structure of the manifold is unique.  For fully augmented links, Theorems \ref{mainthm} and \ref{mainthm2} show how the T-T polynomial gives us the cusp shape, which is a geometric invariant.  Different solutions to the T-T polynomial would imply different choices of cusp shapes, which would contradict the uniqueness.  Thus the solution to the discrete faithful representation must be the one that results in the correct cusp shape. 
\end{proof}

\bibliographystyle{plain}
\bibliography{references}

\
\end{document}